\documentclass[preprint,3p]{elsarticle}
\usepackage{amsmath,amssymb,nicefrac,bm,upgreek,verbatim,mathtools}
\usepackage{dsfont}     %\mathds{1}
\usepackage[mathscr]{eucal}
\usepackage{color,verbatim}
\usepackage{url}
\usepackage[normalem]{ulem}
\biboptions{sort&compress}

\definecolor{dmagenta}{rgb}{.4,.1,.5}
\definecolor{dblue}{rgb}{.0,.0,.6}
\definecolor{ddblue}{rgb}{.0,.0,.4}
\definecolor{dred}{rgb}{.5,.0,.0}
\definecolor{dgreen}{rgb}{.0,.4,.0}
\definecolor{Eeom}{rgb}{.0,.0,.5}

\numberwithin{equation}{section}

\newtheorem{theorem}{Theorem}[section]
\newtheorem{lemma}[theorem]{Lemma}
\newtheorem{proposition}[theorem]{Proposition}
\newtheorem{corollary}[theorem]{Corollary}

\newdefinition{definition}{Definition}[section]
\newdefinition{assumption}{Assumption}[section]
\newdefinition{example}{Example}[section]
\newdefinition{hypothesis}{Hypothesis}[section]
\newdefinition{notation}{Notation}[section]
\newdefinition{remark}{Remark}[section]
\newdefinition{parenthesis}{Parenthetical Comment}[section]
\newdefinition{fact}{Fact}

\newproof{proof}{Proof}
\newproof{pP1.3}{Proof of Proposition~\ref{P1.3}}
\newproof{pT1.2}{Proof of Theorem~\ref{T1.2}}
\newproof{pT1.4}{Proof of Theorem~\ref{T1.4}}
\newproof{pT1.8}{Proof of Theorem~\ref{T1.8}}

\newcommand{\df}{:=}
\DeclareMathOperator{\Exp}{\mathbb{E}}
\DeclareMathOperator{\Prob}{\mathbb{P}}
\newcommand{\D}{\mathrm{d}}   %differential
\newcommand{\E}{\mathrm{e}}   %exponent
\newcommand{\RR}{\mathbb{R}}  %Real numbers
\newcommand{\Rd}{{\mathbb{R}^d}}
\newcommand{\NN}{\mathbb{N}}   %Natural numbers

\newcommand{\Ind}{\mathds{1}}        % indicator function
\newcommand{\Act}{\mathbb{U}}        %Action Set
\newcommand{\Uadm}{\mathfrak{U}}     %Admissible Controls
\newcommand{\Usm}{\mathfrak{U}_{\mathrm{SM}}}  %Stationary Markov controls
\newcommand{\Ussm}{\mathfrak{U}_{\mathrm{SSM}}}  %Stable stationary Markov controls
\newcommand{\Ustab}{\mathfrak{U}_{\mathrm{stab}}}  %Stabilizing controls

\newcommand{\Sob}{\mathscr{W}} %Sobolev space
\newcommand{\Sobl}{\mathscr{W}_{\mathrm{loc}}}  %Sobolev space (local)
\newcommand{\transp}{^{\mathsf{T}}}  %transpose
\newcommand{\order}{{\mathscr{O}}} % Order of
\newcommand{\sorder}{{\mathfrak{o}}} % small Order of

\newcommand{\Lg}{\mathscr{L}}                % Controlled extended generator

\newcommand{\uuptau}{\Breve{\uptau}}
\newcommand{\grad}{\nabla}

\newcommand{\lm}{\Lambda_{\mathrm{m}}^{*}}
\newcommand{\lmx}{\Lambda_{\mathrm{m},x}^{*}}

\newcommand{\sB}{\mathscr{B}}     % Ball for near-monotone cost dominance
\newcommand{\Cc}{\mathcal{C}}     % Set of continuous functions
\newcommand{\sCo}{\mathscr{C}_{\circ}}
\newcommand{\sF}{\mathfrak{F}}    % sigma field
\newcommand{\sG}{\mathscr{G}}    % logtransform
\newcommand{\sK}{\mathscr{K}}     % Set of minima for near-monotone
\newcommand{\Lp}{L}   %Lp
\newcommand{\Lpl}{L_{\mathrm{loc}}}   %Lploc
\newcommand{\sM}{\mathscr{M}}     % Exponential Martingale
\newcommand{\cP}{\mathcal{P}}    % Probability Measures

\newcommand{\cT}{\mathcal{T}}     % Semigroup

\newcommand{\Lyap}{\mathcal{V}}   % Lyapunov function
\newcommand{\cX}{\mathcal{X}}     % Space

\newcommand{\abs}[1]{\lvert#1\rvert}
\newcommand{\norm}[1]{\lVert#1\rVert}
\newcommand{\babs}[1]{\bigl\lvert#1\bigr\rvert}

\newcommand{\bnorm}[1]{\bigl\lVert#1\bigr\rVert}

\DeclareMathOperator*{\argmin}{arg\,min}
\DeclareMathOperator*{\esssup}{ess\,sup}
\DeclareMathOperator*{\essinf}{ess\,inf}

\DeclareMathOperator{\trace}{trace}

\usepackage{accents}
\newlength{\dhatheight}
\newcommand{\doublehat}[1]{%
    \settoheight{\dhatheight}{\ensuremath{\hat{#1}}}%
    \addtolength{\dhatheight}{-0.30ex}%
    \hat{\vphantom{\rule{1pt}{\dhatheight}}%
    \smash{\hat{#1}}}}

\makeatletter
\def\ps@pprintTitle{%
   \let\@oddhead\@empty
   \let\@evenhead\@empty
   \def\@oddfoot{\reset@font\hfil\thepage\hfil}
   \let\@evenfoot\@oddfoot
}
\makeatother
\begin{document}
\begin{frontmatter}

\title
{Infinite horizon risk-sensitive control of diffusions\\[2pt]
without any blanket stability assumptions}

\author[ut]{Ari Arapostathis}%\corref{cor1}}%\fnref{ONR}}
\ead{ari@ece.utexas.edu}

\author[iiser]{Anup Biswas}%\fnref{Inspire}}
\ead{anup@iiserpune.ac.in}

\address[ut]{Department of Electrical and Computer Engineering,
The University of Texas at Austin,\\
2501 Speedway St., EER 7.824, Austin, TX 78712}

\address[iiser]{Department of Mathematics,
Indian Institute of Science Education and Research,\\ 
Dr. Homi Bhabha Road, Pune 411008, India}

\begin{abstract}
We consider the infinite horizon risk-sensitive problem for
nondegenerate diffusions with a compact action space,
and controlled through the drift.
We only impose a structural assumption on the running cost function,
namely near-monotonicity, and show that there always exists a solution
to the risk-sensitive Hamilton--Jacobi--Bellman (HJB) equation, and that
any minimizer in the Hamiltonian is optimal in the class of stationary
Markov controls.
Under the additional hypothesis that the coefficients of the diffusion
are bounded, and satisfy a condition that limits (even though it still allows)
transient behavior, we show that any minimizer in the Hamiltonian is optimal in
the class of all admissible controls.
In addition, we present a sufficient condition, under which
the solution of the HJB is unique (up to a multiplicative constant),
and establish the usual verification result.
We also present some new results concerning the multiplicative Poisson equation
for elliptic operators in $\Rd$.
\end{abstract}


\end{frontmatter}

%%%%%%%%%%%%%%%%%%%%%%%%%%%%%%%%%%%%%%%%%%%%%%%%%%%%%%%%%%%%%%%%%%%%%%%%%%%%%%%%
\section{Introduction}
%%%%%%%%%%%%%%%%%%%%%%%%%%%%%%%%%%%%%%%%%%%%%%%%%%%%%%%%%%%%%%%%%%%%%%%%%%%%%%%%
Optimal control under a risk-sensitive criterion has been an active
area of research for the past 30 years.
It has found applications in finance \cite{Bie-Pli, Fleming-02, Nagai-12},
missile guidance \cite{Speyer}, cognitive neuroscience \cite{Nagen-10},
and many more.
There are many situations which dictate the use of
a risk-sensitive penalty.
For example, if one considers the \emph{risk parameter} to be small then
it approximates the standard mean-variance type cost structure.
Another reason that the risk-sensitive criterion is often desirable is
that it captures the effects of higher order moments of
the running cost in addition to its expectation.
To the best of our knowledge, the risk-sensitive criterion was first considered
in \cite{Howard-71}.
We also refer the reader to \cite{Whittle-90, Whittle-81} for an early account of
risk-sensitive optimal controls.
For discrete state space controlled Markov chains,
the risk-sensitive optimal control problem is studied in
\cite{Cavazos-10, Cavazos-05, Cavazos-09, diMassi-99, diMassi-01, diMassi-07,
Jaskiewicz-07, BorMeyn-02, Suresh-15}.
For optimal control problems where the dynamics are modeled
by controlled diffusions, we refer the reader to
\cite{BenEll-95, BenNag-98, BenNag-00, Biswas-11, Biswas-11a, Biswas-10,
BorSur-14, Menaldi-05, Fleming-95, Fleming-06, Fleming-97, Nagai-96, Kaise-06}.

In this paper we deal with nondegenerate diffusions,
controlled through the drift, 
with the control taking values  in a compact metric space
(see \eqref{E-sde}).
The goal is to minimize an infinite horizon
average risk-sensitive penalty,
where the running cost is assumed to satisfy a \emph{near-monotonicity}
hypothesis (Definition~\ref{D1.1}).
We study the associated Hamilton--Jacobi--Bellman (HJB)
equation and characterize the class of optimal stationary Markov controls.
In \cite{Fleming-95} a similar control problem is studied under the assumption
of asymptotic flatness, and existence of a unique solution to the HJB is established. 
This work is generalized in \cite{Nagai-96}, where the authors impose some 
structural assumptions on the drift and cost
(e.g., the cost necessarily grows to infinity, the action set is
a Euclidean space, etc).
Risk-sensitive control problems with periodic coefficients are studied in
\cite{Menaldi-05}.
Risk-sensitive control for a general class of controlled diffusions is considered in
\cite{Biswas-10, Biswas-11a, Biswas-11},  under the assumption
that all stationary Markov controls are stable.
However, the studies in \cite{Biswas-10, Biswas-11a, Biswas-11} neither
establish uniqueness of the solution to the HJB,
nor do they fully characterize the optimal stationary Markov controls. 
One of our main contributions in this article is the development
of a basic theory that parallels
existing results for the ergodic control problem.

The dynamics are modeled by a
controlled diffusion process $X = \{X_{t},\;t\ge0\}$
which takes values in the $d$-dimensional Euclidean space $\RR^{d}$, and
is governed by the It\^o stochastic differential equation
\begin{equation}\label{E-sde}
\D{X}_{t} \;=\;b(X_{t},U_{t})\,\D{t} + \upsigma(X_{t})\,\D{W}_{t}\,.
\end{equation}
All random processes in \eqref{E-sde} live in a complete
probability space $(\Omega,\sF,\Prob)$.
The process $W$ is a $d$-dimensional standard Wiener process independent
of the initial condition $X_{0}$.
The control process $U$ takes values in a compact, metrizable set $\Act$, and
$U_{t}(\omega)$ is jointly measurable in
$(t,\omega)\in[0,\infty)\times\Omega$.
The set $\Uadm$ of \emph{admissible controls} consists of the
control processes $U$ that are \emph{non-anticipative}:
for $s < t$, $W_{t} - W_{s}$ is independent of
\begin{equation*}
\sF_{s} \;\df\;\text{the completion of~} \sigma\{X_{0},U_{r},W_{r},\;r\le s\}
\text{~relative to~}(\sF,\Prob)\,.
\end{equation*}
We impose the standard assumptions on the drift $b$
and the diffusion matrix $\upsigma$ to guarantee existence
and uniqueness of solutions.
For more details on the model see Section~\ref{S1.2}.

Let $c:\Rd\times\Act\to [1,\infty)$ be continuous,
and locally 
Lipschitz in its first argument uniformly with respect to the second.
For $U\in\Uadm$ we define the risk-sensitive penalty by
\begin{equation*}
\Lambda_x^U=\Lambda_x^U(c)\;\df\; \limsup_{T\to\infty}\;\frac{1}{T}\;
\log\Exp^{U}_{x}\Bigl[\E^{\int_{0}^{T} c(X_{t},U_{t})\,\D{t}}\Bigr]\,,
\end{equation*}
and the risk-sensitive optimal values by
\begin{equation}\label{E-Lambda}
\begin{aligned}
\Lambda^*_x&\;\df\; \inf_{U\in\,\Uadm}\;\Lambda_x^U\,,&\quad
\Lambda^* &\;\df\; \inf_{x\in\,\Rd}\;\Lambda^*_x\,,\\[5pt]
\lmx &\;\df\; \inf_{U\in\,\Usm}\;\Lambda_x^U\,,&\quad
\lm &\;\df\; \inf_{x\in\,\Rd}\;\lmx\,,
\end{aligned}
\end{equation}
where $\Usm$ is the class of stationary Markov controls.
For $v\in\Usm$ we also let $\Lambda^v=\Lambda^v(c) =
\inf_{x\in\,\Rd}\;\Lambda^v_x(c)$.
A stationary Markov control $v$ which satisfies $\Lambda^v<\infty$,
is called \emph{stabilizing}, and we let $\Ustab$ denote
this class of controls.

Unless $\Lambda^*$ is finite, the optimal control problem, is of course ill-posed.
For nonlinear models as in the current paper, standard Foster--Lyapunov conditions
are usually imposed to guarantee that $\Lambda^*<\infty$.
However, the objective of this paper is different.
Rather, we impose a structural assumption on the running cost function $c$,
and investigate whether this is sufficient for characterization of
optimality via the risk-sensitive HJB equation.
We need the following definition.

\begin{definition}[Near-Monotone]\label{D1.1}
A continuous map $g\colon \Rd\times\Act\to\RR$
is said to be \emph{near-monotone relative
to $\lambda\in\RR$} if there exists $\epsilon>0$ such
that the set
\begin{equation*}
\sK_\epsilon\df\bigl\{x\in\Rd\,\colon \min_{u\in\Act}\,g(x,u)
\le\lambda+\epsilon\bigr\}
\end{equation*}
is compact (or empty).
We extend the same notion
to a Borel measurable $f\colon\cX\to\RR$,
by requiring that for some $\epsilon>0$,
and a compact set $\sK_\epsilon\subset\Rd$ it holds that
$\essinf_{\sK^c_\epsilon} (f-\lambda-\epsilon)\ge0$.
We let $\sK \df \cap_{\epsilon>0} \sK_\epsilon$.
We also say that a function $g$ or $f$
as above is \emph{inf-compact}
if it is near-monotone relative to all $\lambda\in\RR$.
\end{definition}

Note that the concept of near-monotonicity in the literature is often
stricter---a function $f$ is sometimes called near-monotone if it is near-monotone
relative to all $\lambda<\norm{f}_{\infty}$ \cite{Balaji-00}.

For an inf-compact running cost $c$, which is what we most often see
in applications, near-monotonicity is of course equivalent to the
statement that $\Lambda^*<\infty$.
Therefore, for a inf-compact running cost, near-monotonicity is also
necessary for the optimal control problem to be well posed.
There are clearly two tasks for this class of models.
First, establish that the class of stabilizing Markov controls $\Ustab$
is nonempty, and then solve the optimal control problem.
This paper addresses the second task.

The main results of the paper can be divided into two groups.
Those concerning the risk-sensitive control problem,
and those concerning the multiplicative Poisson equation (MPE)
for (uncontrolled) diffusions.

For the risk sensitive control problem, there are two sets of results.
First, under the hypothesis that the running cost $c$
is near-monotone relative to $\Lambda^*$, and an assumption on the
drift that limits but not precludes
transience of the controlled process (see Assumption~\ref{A1.1}),
we establish existence of a solution to the risk-sensitive HJB equation,
and also existence of a stationary Markov control
which is optimal over the class of all admissible controls
(see Proposition~\ref{P1.1}).
We wish to point out the optimality over nonstationary controls is
very hard to obtain for the risk-sensitive problem without
blanket geometric ergodicity hypotheses.
For this reason, the optimal control problem is often restricted
to stationary Markov controls (see the analogous study
in the case of denumerable controlled Markov chains in \cite{BorMeyn-02}).

If the running cost is near-monotone relative to
$\lm$, then, without any additional assumptions on the drift,
we show in Proposition~\ref{P1.3} that there exists a pair
$(V^*, \lm)\in\Cc^{2}(\RR^d)\times\RR$ solving the HJB equation and any
measurable selector of the HJB is a stable control, 
and is optimal in the class of stationary Markov controls.
Under the same near-monotonicity hypothesis, together
with the assumption that $c$ is inf-compact,
the risk-sensitive problem for denumerable Markov decision processes
is treated in \cite{BorMeyn-02}, where 
a dynamic programming inequality is established.

Concerning uniqueness of the solution to the HJB we identify
a rather generic sufficient condition which amounts to strict monotonicity
on the right for $\lm$ with respect to the running cost $c$, i.e., that
$c\lneqq c'$ implies $\lm(c)<\lm(c')$.
Under this condition, we show in Theorem~\ref{T1.2}
that there exists a unique solution to
the HJB equation (up to a multiplicative constant), and we have
the usual verification result that states that a stationary Markov
control is optimal only if it is an a.e.\ measurable selector from the minimizer
of the HJB.
In addition, this condition is necessary and sufficient for the solution of the HJB
to be the minimal solution of the MPE over the class of
optimal stationary Markov controls.

The second set of results, which comprises a significant portion of the paper,
concerns the MPE.
Here, the running cost takes the role of a potential $f$ which satisfies
the near-monotonicity hypothesis in Definition~\ref{D1.1} relative
to the principal eigenvalue $\Lambda^*(f)$.
We present a comprehensive study of the relationship between the
solutions of the MPE and their stochastic representations,
and the recurrence properties of the so called
\emph{twisted process} (see \cite{Kontoyiannis-02})
which is associated with eigenfunctions of the principal
eigenvalue $\Lambda^*(f)$.
In the quantum mechanics literature this eigenfunction is called
the \emph{ground state}, and the twisted process is described
by a diffusion which we refer to as the ground state diffusion
(see \eqref{E-sde*}).
The key results are in Theorems~\ref{T1.5}--\ref{T1.6}.
These should be compared with the results
for countable Markov chains in \cite{Balaji-00}.
An important contribution of this paper is the sharp characterization
of the recurrence properties of the ground state diffusion
in terms of the monotonicity of the principal eigenvalue as a function
of the potential.

The notation used in the paper is summarized in Section~\ref{S1.1}.
The assumptions on the model are in Section~\ref{S1.2}, followed
by a summary of the main results in Section~\ref{S1.3}.
followed by a description of the model in 
Section~\ref{S2} contains various results on the multiplicative Poisson equation,
which lead to the proof of Theorems~\ref{T1.5}--\ref{T1.8}.
In Section~\ref{S1.4} we summarize some basic results
from the theory of second order elliptic partial differential equations (pde)
used in this paper.
The proofs of the results concerning the risk-sensitive control problem
are in Section~\ref{S3}.

%%%%%%%%%%%%%%%%%%%%%%%%%%%%%%%%%%%%%%%%%%%%%%%%%%%%%%%%%%%%%%%%%%%%%%%%%%%%%%%%
\subsection{Notation}\label{S1.1}
%%%%%%%%%%%%%%%%%%%%%%%%%%%%%%%%%%%%%%%%%%%%%%%%%%%%%%%%%%%%%%%%%%%%%%%%%%%%%%%%

The standard Euclidean norm in $\RR^{d}$ is denoted by $\abs{\,\cdot\,}$,
and $\langle\,\cdot\,,\cdot\,\rangle$ denotes the inner product.
The set of nonnegative real numbers is denoted by $\RR_{+}$,
$\NN$ stands for the set of natural numbers, and $\Ind$ denotes
the indicator function.
Given two real numbers $a$ and $b$, the minimum (maximum) is denoted by $a\wedge b$ 
($a\vee b$), respectively.
The closure, boundary, and the complement
of a set $A\subset\Rd$ are denoted
by $\Bar{A}$, $\partial{A}$, and $A^{c}$, respectively.
We denote by $\uptau(A)$ the \emph{first exit time} of the process
$\{X_{t}\}$ from the set $A\subset\RR^{d}$, defined by
\begin{equation*}
\uptau(A) \;\df\; \inf\;\{t>0\;\colon\, X_{t}\not\in A\}\,.
\end{equation*}
The open ball of radius $r$ in $\RR^{d}$, centered at the origin,
is denoted by $B_{r}$, and we let $\uptau_{r}\df \uptau(B_{r})$,
and $\uuptau_{r}\df \uptau(B^{c}_{r})$.

The term \emph{domain} in $\RR^{d}$
refers to a nonempty, connected open subset of the Euclidean space $\RR^{d}$. 
For a domain $D\subset\RR^{d}$,
the space $\Cc^{k}(D)$ ($\Cc^{\infty}(D)$)
refers to the class of all real-valued functions on $D$ whose partial
derivatives up to order $k$ (of any order) exist and are continuous,
and $\Cc_{b}(D)$ denotes the set of all bounded continuous
real-valued functions on $D$.
In addition $\Cc_c(D)$ denotes the class of functions in $\Cc(D)$ that
have compact support, and $\Cc_0(\Rd)$ the class of continuous
functions on $\Rd$ that vanish at infinity.
By a slight abuse of notation,
whenever the whole space $\RR^d$ is concerned, we write
$f\in\Cc^{k}(\RR^{d})$ 
whenever $f\in\Cc^{k}(D)$ for all bounded domains $D\subset\RR^{d}$.
The space $\Lp^{p}(D)$, $p\in[1,\infty)$, stands for the Banach space
of (equivalence classes of) measurable functions $f$ satisfying
$\int_{D} \abs{f(x)}^{p}\,\D{x}<\infty$, and $\Lp^{\infty}(D)$ is the
Banach space of functions that are essentially bounded in $D$.
The standard Sobolev space of functions on $D$ whose generalized
derivatives up to order $k$ are in $\Lp^{p}(D)$, equipped with its natural
norm, is denoted by $\Sob^{k,p}(D)$, $k\ge0$, $p\ge1$.

In general, if $\mathcal{X}$ is a space of real-valued functions on $Q$,
$\mathcal{X}_{\mathrm{loc}}$ consists of all functions $f$ such that
$f\varphi\in\mathcal{X}$ for every $\varphi\in\Cc_{c}^{\infty}(Q)$,
the space of smooth functions on $Q$ with compact support.
In this manner we obtain for example the space $\Sobl^{2,p}(Q)$.

For a continuous function $g\,\colon\,\Rd\to[1,\infty)$ 
we let $L^{\infty}_g$ (or $\order(g)$) denote the space of Borel measurable functions
$f\colon\Rd\to\RR$ satisfying
\begin{equation*}
\esssup_{x\in\Rd}\;\frac{\abs{f(x)}}{g(x)}\;<\;\infty\,,
\end{equation*}
and
by $\sorder(g)$ the subspace of functions $f\in L^{\infty}_g$ such that
\begin{equation*}
\limsup_{R\to\infty}\;\esssup_{x\in B_{R}^{c}}\,\frac{\abs{f(x)}}{g(x)}\;=\;0\,.
\end{equation*}
We also let 
$\Cc_{g}(\Rd)$ denote the Banach space of continuous functions under the norm
\begin{equation*}
\norm{f}_{g} \;\df\;
\sup_{x\in\Rd}\;\frac{\abs{f(x)}}{g(x)}\,.
\end{equation*}

We adopt the notation
$\partial_{i}\df\tfrac{\partial~}{\partial{x}_{i}}$ and
$\partial_{ij}\df\tfrac{\partial^{2}~}{\partial{x}_{i}\partial{x}_{j}}$
for $i,j\in\NN$.
We often use the standard summation rule that
repeated subscripts and superscripts are summed from $1$ through $d$.
For example,
\begin{equation*}
\tfrac{1}{2}\, a^{ij}\partial_{ij}\varphi
+ b^{i} \partial_{i}\varphi \;\df\; \tfrac{1}{2}\, \sum_{i,j=1}^{d}a^{ij}
\frac{\partial^{2}\varphi}{\partial{x}_{i}\partial{x}_{j}}
+\sum_{i=1}^{d} b^{i} \frac{\partial\varphi}{\partial{x}_{i}}\,.
\end{equation*}

%%%%%%%%%%%%%%%%%%%%%%%%%%%%%%%%%%%%%%%%%%%%%%%%%%%%%%%%%%%%%%%%%%%%%%%%%%%%%%%%
\subsection{The model}\label{S1.2}
%%%%%%%%%%%%%%%%%%%%%%%%%%%%%%%%%%%%%%%%%%%%%%%%%%%%%%%%%%%%%%%%%%%%%%%%%%%%%%%%

The following assumptions on the diffusion
\eqref{E-sde} are in effect throughout the paper
unless otherwise mentioned.
\begin{itemize}
\item[(A1)]
\emph{Local Lipschitz continuity:\/}
The functions
\begin{equation*}
b\;=\;\bigl[b^{1},\dotsc,b^{d}\bigr]\transp\,\colon\,\RR^{d}\times\Act\to\RR^{d}\,,
\quad\text{and}\quad
\upsigma\;=\;\bigl[\upsigma^{ij}\bigr]\,\colon\,\RR^{d}\to\RR^{d\times d}
\end{equation*}
are locally Lipschitz in $x$ with a Lipschitz constant $C_{R}>0$
depending on $R>0$.
In other words, we have
\begin{equation*}
\abs{b(x,u) - b(y,u)} + \norm{\upsigma(x) - \upsigma(y)}
\;\le\;C_{R}\,\abs{x-y}\qquad\forall\,x,y\in B_R\,,\text{\ \ and\ }u\in\Act\,.
\end{equation*}
We also assume that $b$ is continuous.
\smallskip
\item[(A2)]
\emph{Affine growth condition:\/}
$b$ and $\upsigma$ satisfy a global growth condition of the form
\begin{equation*}
\abs{b(x,u)}^2 + \norm{\upsigma(x)}^{2}\;\le\;C_0
\bigl(1 + \abs{x}^{2}\bigr) \qquad \forall\, (x,u)\in\RR^{d}\times\Act
\end{equation*}
for some constant $C_0>0$,
where $\norm{\upsigma}^{2}\;\df\;
\mathrm{trace}\left(\upsigma\upsigma\transp\right)$.
\smallskip
\item[(A3)]
\emph{Nondegeneracy:\/}
For each $R>0$, it holds that
\begin{equation*}
\sum_{i,j=1}^{d} a^{ij}(x)\xi_{i}\xi_{j}
\;\ge\;C^{-1}_{R} \abs{\xi}^{2} \qquad\forall\, x\in B_{R}\,,
\end{equation*}
and for all $\xi=(\xi_{1},\dotsc,\xi_{d})\transp\in\RR^{d}$,
where $a\df \upsigma \upsigma\transp$.
\end{itemize}

In integral form, \eqref{E-sde} is written as
\begin{equation}\label{E2.2}
X_{t} \;=\;X_{0} + \int_{0}^{t} b(X_{s},U_{s})\,\D{s}
+ \int_{0}^{t} \upsigma(X_{s})\,\D{W}_{s}\,.
\end{equation}
The third term on the right hand side of \eqref{E2.2} is an It\^o
stochastic integral.
We say that a process $X=\{X_{t}(\omega)\}$ is a solution of \eqref{E-sde},
if it is $\sF_{t}$-adapted, continuous in $t$, defined for all
$\omega\in\Omega$ and $t\in[0,\infty)$, and satisfies \eqref{E2.2} for
all $t\in[0,\infty)$ a.s.
It is well known that under (A1)--(A3), for any admissible control
there exists a unique solution of \eqref{E-sde}
\cite[Theorem~2.2.4]{book}.
We define the family of operators $\Lg^{u}:\Cc^{2}(\RR^{d})\mapsto\Cc(\RR^{d})$,
where $u\in\Act$ plays the role of a parameter, by
\begin{equation}\label{E-Lgu}
\Lg^{u} f(x) \;=\; \tfrac{1}{2}\, a^{ij}(x)\,\partial_{ij} f(x)
+ b^{i}(x,u)\, \partial_{i} f(x)\,,\quad u\in\Act\,.
\end{equation}

Let $\Usm$ denote the set of stationary Markov controls.
It is well known that under $v\in\Usm$
\eqref{E-sde} has a unique strong solution \cite{Gyongy-96}.
Moreover, under $v\in\Usm$, the process $X$ is strong Markov,
and we denote its transition kernel by $P^{t}_{v}(x,\cdot\,)$.
It also follows from the work in \cite{Bogachev-01} that under
$v\in\Usm$, the transition probabilities of $X$
have densities which are locally H\"older continuous.
Thus $\Lg_{v}$ defined by
\begin{equation*}
\Lg_{v} f(x) \;=\; \tfrac{1}{2}\, a^{ij}(x)\,\partial_{ij} f(x)
+ b^{i} \bigl(x,v(x)\bigr)\, \partial_{i} f(x)\,,\quad v\in\Usm\,,
\end{equation*}
for $f\in\Cc^{2}(\RR^{d})$,
is the generator of a strongly-continuous
semigroup on $\Cc_{b}(\RR^{d})$, which is strong Feller.
When $v\in\Usm$ we use $v$ as subscript in $\Lg_v$ to distinguish
it from $\Lg^u$, $u\in\Act$, defined in the preceding paragraph.
We let $\Prob_{x}^{v}$ denote the probability measure and
$\Exp_{x}^{v}$ the expectation operator on the canonical space of the
process under the control $v\in\Usm$, conditioned on the
process $X$ starting from $x\in\RR^{d}$ at $t=0$.
We denote by $\Ussm$ the subset of $\Usm$ that consists
of \emph{stable controls}, i.e.,
under which the controlled process is positive recurrent,
and by $\mu_v$ the invariant probability measure of the process
under the control $v\in\Ussm$.

%%%%%%%%%%%%%%%%%%%%%%%%%%%%%%%%%%%%%%%%%%%%%%%%%%%%%%%%%%%%%%%%%%%%%%%%%%%%%%%%
\subsection{Main results}\label{S1.3}
%%%%%%%%%%%%%%%%%%%%%%%%%%%%%%%%%%%%%%%%%%%%%%%%%%%%%%%%%%%%%%%%%%%%%%%%%%%%%%%%

Consider the following assumption on the drift of \eqref{E-sde}.

%%%%%%%%%%%%%%%%%%%%%%%%%%%%%%%%%%%%%%%%%%%%%%%%%%%%%%%%%%%%%%%%%%%%%%%%%%%%%%%%
\begin{assumption}\label{A1.1}
The coefficients $b$ and $\upsigma$ of $\Lg^u$ in \eqref{E-Lgu} are bounded,
$\upsigma$ is Lipschitz continuous,
and for some constant $C$ we have
\begin{equation*}
\sum_{i,j=1}^{d} a^{ij}(x)\xi_{i}\xi_{j}
\;\ge\;C^{-1} \abs{\xi}^{2} \qquad\forall\, x\in \Rd\,.
\end{equation*}
In addition we assume that
\begin{equation}\label{EA1.2}
\max_{u\in\Act}\;\frac{\bigl\langle b(x,u),\, x\bigr\rangle^{+}}{\abs{x}}
\;\xrightarrow[\abs{x}\to\infty]{}\;0\,.
\end{equation}
\end{assumption}

Recall the definitions in \eqref{E-Lambda}.
We let $\sCo$ denote the class of nonnegative functions in $\Cc_0(\Rd)$ which
are not identically equal to $0$.
We also define
\begin{align*}
\mathfrak{V}&\;\df\;\bigl\{(V,\Lambda)\in\Cc^{2}(\RR^{d})\times\RR\,\colon\,
V(0)=1\,,~V>0\,,~\Lambda\le\Lambda^*\bigr\}\,,
\shortintertext{and}
\mathfrak{V}_{\circ}&\;\df\;\bigl\{(V,\Lambda)\in\Cc^{2}(\RR^{d})\times\RR\,\colon\,
V(0)=1\,,~\inf_{\Rd}\,V>0\,,~\Lambda\le\Lambda^*\bigr\}\,.
\end{align*}

%%%%%%%%%%%%%%%%%%%%%%%%%%%%%%%%%%%%%%%%%%%%%%%%%%%%%%%%%%%%%%%%%%%%%%%%%%%%%%%%
\begin{proposition}\label{P1.1}
Let $a$ and $b$ satisfy Assumption~\ref{A1.1} and
$c$ be near-monotone relative to $\Lambda^*$ and bounded.
Then the HJB equation
\begin{equation}\label{E1-HJB}
\min_{u\in\Act}\;
\bigl[\Lg^{u} V^*(x) + c(x,u)\,V^*(x)\bigr] \;=\; \Lambda^*\,V^*(x)
\qquad\forall\,x\in\Rd
\end{equation}
has a  solution $V^*\in\Cc^{2}(\RR^{d})$,
satisfying $V^*(0)=1$ and $\inf_{\Rd}\,V^*>0$, and the following hold:
\begin{enumerate}
\item[{\upshape(}i\/{\upshape)}]
$\Lambda^*_x=\Lambda^*$ for all $x\in\Rd$.

\smallskip
\item[{\upshape(}ii\/{\upshape)}]
Any $v\in\Usm$ that satisfies
\begin{equation}\label{EP1.1A}
\Lg_v V^*(x) + c\bigl(x,v(x)\bigr)\,V^*(x)\;=\;
\min_{u\in\Act}\; \bigl[\Lg^{u} V^*(x) + c(x,u)\,V^*(x)\bigr]
\quad \text{a.e.\ }x\in\Rd
\end{equation}
is stable, and is optimal, i.e., $\Lambda^v_x=\Lambda^*$ for all $x\in\Rd$.
In particular, $\lm=\Lambda^*$.

\smallskip
\item[{\upshape(}iii\/{\upshape)}]
It holds that
\begin{equation*}
V^*(x) \;=\;\Exp_{x}^{v}\Bigl[\E^{\int_{0}^{T}
[c(X_{t},v(X_{t}))-\Lambda^*]\,\D{t}}\,V^*(X_T)\Bigr]
\qquad\forall\, (T,x)\in\RR_+\times\Rd\,,
\end{equation*}
for any $v\in\Usm$ that satisfies \eqref{EP1.1A}.
\end{enumerate}
\end{proposition}

\begin{proof}
Existence of a solution and parts (i)--(ii) of the proposition
follow by Theorem~\ref{T3.4} in Section~\ref{S3}.
Part (iii)  follows by Theorem~\ref{T1.5} and Lemma~\ref{L3.3},
together with the fact that a diffusion with Lipschitz continuous
diffusion matrix and a drift having at most linear growth is regular.
\qed\end{proof}

%%%%%%%%%%%%%%%%%%%%%%%%%%%%%%%%%%%%%%%%%%%%%%%%%%%%%%%%%%%%%%%%%%%%%%%%%%%%%%%%
\begin{remark}
The hypothesis in \eqref{EA1.2}
of Assumption~\ref{A1.1}
may be replaced by the following.
There exists a $\Cc^{2}$ function $\Lyap_\circ$, satisfying
$\liminf_{\abs{x}\to\infty}\;\frac{\Lyap_\circ(x)}{1+\abs{x}^{2}}>0$,
such that
\begin{equation*}
\frac{\bigl[\Lg^{u}\Lyap_\circ(x)\bigr]^{+}}{\sqrt{\Lyap_\circ(x)}}
\;\xrightarrow[\abs{x}\to\infty]{}\;0\qquad\forall\, u\in\Act\,.
\end{equation*}
It is clear from the proof that the result of
Lemma~\ref{L3.2} in Section~\ref{S3} holds under this assumption.
It is also evident that
\eqref{EA1.2} may be replaced by the more general hypothesis
that $\Exp_{x}^{U}\bigl[\abs{X_{t}}\bigr]\in\sorder(t)$ under any
$U\in\Uadm$, which is the conclusion of Lemma~\ref{L3.2} on which the
proof of Theorem~\ref{T3.4} is based.
Note that when the coefficients $b$ and $\upsigma$ are bounded, it is always the case
that $\Exp_{x}^{U}\bigl[\abs{X_{t}}\bigr]\in\order(t)$.
\end{remark}

%%%%%%%%%%%%%%%%%%%%%%%%%%%%%%%%%%%%%%%%%%%%%%%%%%%%%%%%%%%%%%%%%%%%%%%%%%%%%%%%%
\begin{remark}
As shown in \cite[Proposition~2.6]{Berestycki-15} if $a$ and $b$ are bounded
then in general $\Lambda^*=\infty$ when $c$ is not bounded.
Therefore, the assumption that $c$ is bounded in Proposition~\ref{P1.1}
cannot be relaxed.
\end{remark}

Let
\begin{equation*}
\Usm^* \;\df\;
\bigl\{v\in\Usm \colon \Lambda^v_x(c) = \lm\,,\ \forall\, x\in\Rd\bigr\}\,.
\end{equation*}
Concerning uniqueness of the HJB equation we have the following.

%%%%%%%%%%%%%%%%%%%%%%%%%%%%%%%%%%%%%%%%%%%%%%%%%%%%%%%%%%%%%%%%%%%%%%%%%%%%%%%%%
\begin{theorem}\label{T1.2}
Suppose that in addition to the assumptions of Proposition~\ref{P1.1}
it holds that
\begin{equation}\label{ET1.2A}
\lm(c+h)\;>\;\lm(c)\qquad\forall\,h\in\sCo\,.
\end{equation}
Then there exists a unique pair $(V,\Lambda)\in\mathfrak{V}$ which solves
\begin{equation}\label{ET1.2B}
\min_{u\in\Act}\;
\bigl[\Lg^{u} V(x) + c(x,u)\,V(x)\bigr] \;=\; \Lambda\,V(x)
\qquad\forall\,x\in\Rd\,.
\end{equation}
and $v\in\Usm^*$ if and only if it satisfies \eqref{EP1.1A}.
In addition, the function $V^*$ in Proposition~\ref{P1.1}
has the stochastic representation
\begin{equation}\label{ET1.2C}
V^*(x)\;=\;\Exp_{x}^{v}\Bigl[\E^{\int_{0}^{\uuptau_r}
[c(X_{t},v(X_{t}))-\Lambda^*]\,\D{t}}\,V^*(X_{\uuptau_r})\Bigr]
\qquad\forall\, x\in \Bar{B}_r^c\,,
\end{equation}
for all $r>0$, and $v\in\Usm^*$.
Conversely, if $v\in\Usm^*$ satisfies \eqref{ET1.2C} for some $r>0$,
then $\Lambda^v(c+h) > \Lambda^v(c)$ for all $h\in\sCo$.
\end{theorem}

\begin{proof}
The proof is in Section~\ref{S3}.
\qed\end{proof}

Without imposing any restrictions on the coefficients, we
have the following result.

%%%%%%%%%%%%%%%%%%%%%%%%%%%%%%%%%%%%%%%%%%%%%%%%%%%%%%%%%%%%%%%%%%%%%%%%%%%%%%%%
\begin{proposition}\label{P1.3}
Suppose that $c$ is near-monotone relative to $\lm$.
Then the HJB equation
\begin{equation*}
\min_{u\in\Act}\;
\bigl[\Lg^{u} V^*(x) + c(x,u)\,V^*(x)\bigr] \;=\; \lm\,V^*(x)
\qquad\forall\,x\in\Rd
\end{equation*}
has a solution $V^*\in\Cc^{2}(\RR^{d})$,
satisfying $V^*(0)=1$ and $\inf_{\Rd}\,V^*>0$.
Moreover, any $v\in\Usm$ that satisfies \eqref{EP1.1A}
is stable, and is optimal in the class $\Usm$, i.e., $\Lambda^v_x=\lm$
for all $x\in\Rd$.

Under the additional assumption in \eqref{ET1.2A},
there exists a unique pair $(V,\Lambda)\in\mathfrak{V}_\circ$ which solves
\eqref{ET1.2B}, and
$v\in\Usm^*$ if and only if it satisfies \eqref{EP1.1A}.
The function $V^*$ has the stochastic representation
\begin{equation}\label{EP1.3A}
V^*(x)\;=\;\Exp_{x}^{v}\Bigl[\E^{\int_{0}^{\uuptau_r}
[c(X_{t},v(X_{t}))-\lm]\,\D{t}}\,V^*(X_{\uuptau_r})\Bigr]
\qquad\forall\, x\in \Bar{B}_r^c\,,
\end{equation}
for all $r>0$, and $v\in\Usm^*$.
Conversely, if $v\in\Usm^*$ satisfies \eqref{EP1.3A} for some $r>0$,
then $\Lambda^v(c+h) > \Lambda^v(c)$ for all $h\in\sCo$.
\end{proposition}

The proof of Proposition~\ref{P1.3} is in Section~\ref{S3}.

%%%%%%%%%%%%%%%%%%%%%%%%%%%%%%%%%%%%%%%%%%%%%%%%%%%%%%%%%%%%%%%%%%%%%%%%%%%%%%%%
\begin{remark}
The main reason $\lm$ appears in Proposition~\ref{P1.3}
instead of $\Lambda^*$ has to do with the
way the solution $V^*$ is constructed.
It should be kept in mind that $\Lambda^*\leq\lm$, in general, and
therefore we cannot follow the path of Proposition~\ref{P1.1}
to prove Proposition~\ref{P1.3}.
Instead, a fixed $\epsilon$-optimal stationary Markov control is imposed at
`$\infty$' to guarantee that the solution is bounded away from zero, and then
a limit is taken as $\epsilon\searrow0$.
This has the effect of restricting optimality over the class $\Usm$.
For more details on this see Remark~\ref{R3.2} in Section~\ref{S3}.
\end{remark}

%%%%%%%%%%%%%%%%%%%%%%%%%%%%%%%%%%%%%%%%%%%%%%%%%%%%%%%%%%%%%%%%%%%%%%%%%%%%%%%%
The proofs of these results depend heavily on properties of
the multiplicative Poisson equation (MPE),
which are summarized next.

\subsubsection{Results concerning the multiplicative Poisson equation}

We consider an uncontrolled diffusion
\begin{equation}\label{E1-sde2}
\D{X}_{t} \;=\;b(X_{t})\,\D{t} + \upsigma(X_{t})\,\D{W}_{t}\,,
\end{equation}
where $\upsigma$ and $b$ satisfy (A2)--(A3),  $\upsigma$ is locally  Lipschitz
(as in (A1)), and $b$ is measurable.
We let $\Exp_x$
denote the expectation operator induced by the
strong Markov process with $X_{0}=x$, governed by \eqref{E1-sde2},
and
\begin{equation}\label{E-Lg}
\Lg \;\df\; \tfrac{1}{2}\,a^{ij}(x)\,\partial_{ij}
+ b^{i}(x)\, \partial_{i}\,,
\end{equation}
 with
$a\df \upsigma \upsigma\transp$.
Let $f\,\colon\Rd\to\RR_+$ be measurable and locally bounded,
and define
\begin{equation}\label{E-Lambdaf}
\begin{split}
\Lambda_x(f) &\;\df\; \limsup_{T\to\infty}\;\frac{1}{T}\,\log\,
\Exp_x\Bigl[\E^{\int_{0}^{T} f(X_t)\,\D{t}}\Bigr]\qquad
\forall\, x\in\Rd\,,\\[5pt]
\Lambda(f) &\;\df\; \inf_{x\in\,\Rd}\;\Lambda_x(f)\,.
\end{split}
\end{equation}
We assume $\Lambda(f)<\infty$.
We say that $(\Psi,\Lambda)\in\Sobl^{2,p}(\Rd)\times\RR$, $p>d$, $\Psi>0$,
is a solution of the multiplicative Poisson equation (MPE)
if it satisfies
\begin{equation}\label{E1-Pois}
\Lg \Psi(x) + f(x)\,\Psi(x)
\;=\; \Lambda\,\Psi(x) \qquad\text{a.e.\ }x\in\Rd\,.
\end{equation}
We refer to $\Lambda$ as an \emph{eigenvalue} of the MPE.

Consider the following hypothesis.
%%%%%%%%%%%%%%%%%%%%%%%%%%%%%%%%%%%%%%%%%%%%%%%%%%%%%%%%%%%%%%%%%%%%%%%%%%%%%%%%
\begin{description}
\item[(H1)] The diffusion in \eqref{E1-sde2} is recurrent, and $f\colon\Rd\to\RR_+$
is a locally bounded measurable map which
is near-monotone relative to $\Lambda(f)$.
\end{description}

Implicit in the statement in (H1) is of course the requirement $\Lambda(f)<\infty$.
We have assumed $f\ge0$, for simplicity, but all the results are
valid with $f$ bounded below in $\Rd$.

We compare the definition in \eqref{E-Lambdaf}
with the following definitions for the principal eigenvalue,
commonly used in the pde literature
\cite{Berestycki-15}:

\begin{align*}
\Hat\Lambda(f) &\;\df\;
\inf\;\bigl\{\lambda\in\RR\,\colon \exists\, \varphi\in\Sobl^{2,d}(\Rd)\,,
\ \varphi>0\,,\ \Lg\varphi + (f-\lambda)\varphi \le 0\ \text{a.e. in\ }\Rd\bigr\}\,,
\\[5pt]
\doublehat\Lambda(f) &\;\df\;
\inf\;\bigl\{\lambda\in\RR\,\colon \exists\, \varphi\in\Sobl^{2,d}(\Rd)\,,
\ \inf_{\Rd}\,\varphi>0\,,\ \Lg\varphi
+ (f-\lambda)\varphi \le 0\ \text{a.e. in\ }\Rd\bigr\}\,.
\end{align*}

We have the following theorem.
%%%%%%%%%%%%%%%%%%%%%%%%%%%%%%%%%%%%%%%%%%%%%%%%%%%%%%%%%%%%%%%%%%%%%%%%%%%%%%%%
\begin{theorem}\label{T1.4}
Under {\upshape(H1)}, we have $\Lambda(f) = \Hat\Lambda(f) =\doublehat\Lambda(f)$.
\end{theorem}

\begin{proof}
The proof is in Section~\ref{S2.1}.
\qed
\end{proof}

%%%%%%%%%%%%%%%%%%%%%%%%%%%%%%%%%%%%%%%%%%%%%%%%%%%%%%%%%%%%%%%%%%%%%%%%%%%%%%%%
\begin{definition}
Let $\Psi\in\Sobl^{2,p}(\Rd)$, $p>d$, be a positive
solution of the MPE
\begin{equation}\label{E1-Pois2}
\Lg \Psi(x) + f(x)\,\Psi(x)
\;=\; \Lambda(f)\,\Psi(x) \qquad\text{a.e.\ }x\in\Rd\,,
\end{equation}
and let $\psi\df \log \Psi$.
We introduce the stochastic differential equation (sde)
\begin{equation}\label{E-sde*}
\D{X}^*_{t} \;=\; \bigl(b(X^*_t) + a (X^*_t)
\grad\psi(X^*_t)\bigr)\,\D{t} + \upsigma(X^*_t)\,\D W^*_{t}\,,
\end{equation}
where $W^*$ is, as usual, a standard Wiener process.
We denote by $\Lg^*$ the extended generator of \eqref{E-sde*},
given by
\begin{equation}\label{E-Lg*}
\Lg^* g\;\df\; \tfrac{1}{2}a^{ij}\partial_{ij} g + \langle b,\grad g\rangle + 
\langle a\grad\psi,\grad g\rangle
\end{equation}
for $g\in\Cc^2(\Rd)$.
\end{definition}

The sde in \eqref{E-sde*} is well known.
Recall the Feynman--Kac semigroup corresponding to $\Lg+f$, given by
\begin{equation}\label{E-FK}
P^f_t h(x) \;\df\; \Exp_x\Bigl[\E^{\int_{0}^{t}f(X_s)\,\D{s}}\,h(X_{t})\Bigr]
\qquad\text{for\ \ } h\ge 0\,,\quad h\text{\ measurable.}
\end{equation}
The function $f$ is referred to as the \emph{potential} in the
study of the Feynman--Kac semigroup for symmetric Markov processes, and the
eigenfunction $\Psi$ is called a \emph{ground state}.
The \emph{ground state semigroup}  is given by
\begin{equation}\label{E-cT*}
\cT^{\Psi}_t h(x) \;\df\; \E^{-\Lambda(f)t}\frac{1}{\Psi(x)}\,P^f_t (\Psi h)(x)\,,
\end{equation}
and it turns out that $\Lg^*$ is its generator
\cite{Wu-94,Pinsky}.
In \cite{Kontoyiannis-02,Balaji-00}
$\cT^{\Psi}_t$ called the \emph{twisted kernel}.
In addition, elliptic equations with $\Lg^*$ have been studied extensively
in \cite{Kaise-06}, albeit under smoothness assumptions on the coefficients.

Since the drift of \eqref{E-sde*} does not necessarily satisfy
(A2), existence and uniqueness of a solution for this equation is
not guaranteed.
Diffusions with locally bounded drift have been studied in \cite{Gyongy-96},
and under the assumption that the diffusion matrix $\upsigma$ is locally
Lipschitz and nonsingular, existence of a unique strong solution up to
explosion time has been established.
Recently, similar results have been obtained for locally integrable
drifts \cite{Krylov-05}.
Moreover, if the diffusion in \eqref{E-sde*} is positive recurrent,
it is well known that it has a unique invariant probability measure,
with a positive density \cite{Bogachev-01}.

We say that the diffusion in \eqref{E-sde*} is \emph{regular}, if 
it has a unique strong solution which exists for all $t>0$.

%%%%%%%%%%%%%%%%%%%%%%%%%%%%%%%%%%%%%%%%%%%%%%%%%%%%%%%%%%%%%%%%%%%%%%%%%%%%%%%%
\begin{theorem}\label{T1.5}
Assume {\upshape(H1)}.
Then the diffusion in \eqref{E-sde*} is \emph{regular}
if and only if $P^f_t \Psi(x) = \E^{\Lambda(f) t} \Psi(x)$ for all
$(t,x)\in\RR_+\times\Rd$.
In addition, the following are equivalent:
\begin{enumerate}
\item[{\upshape(}i\/{\upshape)}]
The process $X^*$ in \eqref{E-sde*} is recurrent.

\smallskip
\item[{\upshape(}ii\/{\upshape)}]
For some $r>0$,  we have
\begin{equation*}
\Psi(x) \;=\; \Exp_x\Bigl[\E^{\int_{0}^{\uuptau_r}[f(X_s)-\Lambda(f)]\,\D{s}}\,
\Psi(X_{\uuptau_r})\Bigr]\qquad\forall\,x\in B_r^c\,.
\end{equation*}

\item[{\upshape(}iii\/{\upshape)}]
For some $x\in\Rd$ it holds that
\begin{equation*}
\int_0^\infty\Exp_{x}\,\Bigl[\E^{\int_{0}^{t}
[f(X_s)-\Lambda(f)]\,\D{s}}\Bigr]\,\D{t}\;=\;\infty\,.
\end{equation*}

\item[{\upshape(}iv\/{\upshape)}]
It holds that $\Lambda(f)<\Lambda(f+h)$
for all $h\in\sCo$.

\smallskip
\item[{\upshape(}v\/{\upshape)}]
If $A\subset\Rd$ is a  measurable set of positive
Lebesgue measure, then $\Lambda(f)<\Lambda(f+\epsilon\Ind_A)$
for all $\epsilon>0$.
\end{enumerate}
Moreover, in {\upshape(}ii\/{\upshape)}--{\,\upshape(}iii\/{\upshape)}
 ``some'' may be replaced with ``all''.
\end{theorem}

It is clear by Theorem~\ref{T1.5}\,(ii) that
if the process $X^*$ in \eqref{E-sde*},
corresponding to some positive solution $\Psi$ of \eqref{E1-Pois2},
is recurrent, then there exists
a unique positive solution $\Psi\in\Sobl^{2,p}(\Rd)$, $p>d$, 
to the MPE in \eqref{E1-Pois2}, satisfying $\Psi(0)=1$.
In particular the ground state diffusion in \eqref{E-sde*}
is uniquely determined by $f$.

\begin{proof}
These results are contained in individual lemmas in Section~\ref{S2.2}.
The first assertion follows by Lemmas~\ref{L2.4} and \ref{L2.5}.
That (i)$\;\Leftrightarrow\;$(ii)
and  (i)$\;\Leftrightarrow\;$(v) follow by  Lemmas~\ref{L2.6}
and \ref{L2.11}, respectively.
Theorem~\ref{T2.8} asserts that
(iii)$\;\Rightarrow\;$(ii), while (i)$\;\Rightarrow\;$(iii)
follows by Lemma~\ref{L2.9}.
It is easy to see that (v)$\;\Rightarrow\;$(iv), and by
Lemma~\ref{L2.11},
(iv)$\;\Rightarrow\;$(i).
\end{proof}

%%%%%%%%%%%%%%%%%%%%%%%%%%%%%%%%%%%%%%%%%%%%%%%%%%%%%%%%%%%%%%%%%%%%%%%%%%%%%%%%
\begin{theorem}\label{T1.6}
Under {\upshape(H1)}, the following are equivalent.
\begin{enumerate}
\item[{\upshape(}i\/{\upshape)}]
The process $X^*$ in \eqref{E-sde*} is geometrically ergodic.

\smallskip
\item[{\upshape(}ii\/{\upshape)}]
For some $r>0$, there exists $\delta_r>0$, such that
\begin{equation*}
\Exp_{x}\,\Bigl[\E^{\int_{0}^{\uuptau_r}
[f(X_{s})-\Lambda(f)+\delta_r]\,\D{s}}\Bigr]\;<\;\infty
\qquad\forall\, x\in \Bar{B}_r^c\,.
\end{equation*}

\item[{\upshape(}iii\/{\upshape)}]
It holds that $\Lambda(f)>\Lambda(f-h)$
for some $h\in\sCo$.
\end{enumerate}
Moreover, in {\upshape(}ii\/{\upshape)}--{\,\upshape(}iii\/{\upshape)}
 ``some'' may be replaced with ``all''.
 \end{theorem}
 
\begin{proof}
This follows by Theorem~\ref{T2.12} and Corollary~\ref{C2.14}
in Section~\ref{S2.3}.
\qed
\end{proof}

%%%%%%%%%%%%%%%%%%%%%%%%%%%%%%%%%%%%%%%%%%%%%%%%%%%%%%%%%%%%%%%%%%%%%%%%%%%%%%%%
\begin{remark}
Strict monotonicity of the principal eigenvalue,
i.e., $f'\lneqq f$ implies $\Lambda(f')<\Lambda(f)$,
holds for bounded domains (see also  Lemma~\ref{L2.2}\,(b) in Section~\ref{S2}).
However, this property does not
hold in unbounded domains in general \cite[Remark~2.4]{Berestycki-15}.
Under the near-monotone hypothesis in (H1), Theorems~\ref{T1.5}--\ref{T1.6},
provide the following characterization:
Let $X^*$ be the process in \eqref{E-sde*} corresponding to a solution
of \eqref{E1-Pois2} for $f$.
Then $f\gneqq f'$ implies $\Lambda(f')<\Lambda(f)$
if and only if $X^*$ is geometrically ergodic,
and $X^*$ is recurrent if and only if 
$\Lambda(f')>\Lambda(f)$ for all $f'$ satisfying $f'\gneqq f$.
 
Generally speaking, not much is known for the principal eigensolutions in $\Rd$
of the class of elliptic operators $\Lg$ considered in this paper.
Compare with Sections~8--9 in \cite{Berestycki-15}.
Therefore we feel this characterization will be of interest to
a wider audience.
\end{remark}

\begin{remark}
It is evident by Theorem~\ref{T1.5}\,(iv) and 
(iii)$\;\Rightarrow\;$(i) in Theorem~\ref{T1.6} that if
the ground state diffusion corresponding to $f$ is recurrent,
then the ground state diffusion corresponding to $f+h$
is geometrically ergodic for any $h\in\sCo$.

It is also interesting to note that, under (H1), there always
exists $h\in\sCo$ such that the ground state diffusion in \eqref{E-sde*}
corresponding to $f+h$ is geometrically ergodic.
In fact such an $h$ may be selected in $\Cc_c(\Rd)$ with a given support.
This assertion follows by the proof of Lemma~\ref{L2.11}.
\end{remark}

Consider the following hypothesis.

%%%%%%%%%%%%%%%%%%%%%%%%%%%%%%%%%%%%%%%%%%%%%%%%%%%%%%%%%%%%%%%%%%%%%%%%%%%%%%%%
\begin{description}
\item[(H2)] There exists a smooth function $\psi_0$ such that
\begin{equation}\label{E-H2}
\tfrac{1}{2}\,a^{ij}\partial_{ij} \psi_0 + \langle b,\grad \psi_0\rangle + 
\tfrac{1}{2}\,\langle\grad\psi_0,a\grad \psi_0\rangle +f
\;\xrightarrow\;-\infty\quad\text{as\ \ } \abs{x}\to\infty\,.
\end{equation}
\end{description}

%%%%%%%%%%%%%%%%%%%%%%%%%%%%%%%%%%%%%%%%%%%%%%%%%%%%%%%%%%%%%%%%%%%%%%%%%%%%%%%%
In Theorem~\ref{T1.7} stated below,
we do not assume that the running cost is inf-compact,
or even near-monotone in the sense of \cite[p.~126]{Balaji-00}.
Nor do we assume that $\Lambda(\alpha f)<\infty$ for
some $\alpha>1$, as is common in the literature.
This should be compared with \cite[Theorem~1.2]{Balaji-00},
and \cite[Theorem~2.2]{BorMeyn-02} for
irreducible Markov chains, as well as the more general results
in \cite{Kontoyiannis-02,Kontoyiannis-05,Wu-94}.

%%%%%%%%%%%%%%%%%%%%%%%%%%%%%%%%%%%%%%%%%%%%%%%%%%%%%%%%%%%%%%%%%%%%%%%%%%%%%%%%
\begin{theorem}\label{T1.7}
Under {\upshape(H1)--(H2)},
the sde in \eqref{E-sde*} has a unique strong solution $X^*$ which
exists for all $t>0$ and is a strong Markov process.
Moreover, it is positive recurrent, and its unique
invariant probability measure $\mu^*$ has positive density.
In addition, for any $g\in\Cc_b(\Rd)$, the following hold:
\begin{align*}
\cT^{\Psi}_t g(x)
&\;=\; \Exp^*_x\bigl[g(X^*_t)\bigr]\qquad\forall\,(t,x)\in\RR_+\times\Rd\,,\\[5pt]
\cT^{\Psi}_t g(x)
&\;\xrightarrow[t\to\infty]{}\; \int_{\Rd}g(y)\,\mu^*(\D{y})\qquad
\forall\,x\in\Rd\,,
\end{align*}
where $\Exp_x^*$ denotes the expectation operator associated
with the solution $X^*$ of \eqref{E-sde*}.
In particular, the limit 
\begin{equation*}
\Psi_*(x) \;\df\; \lim_{t\to\infty}\;
\Exp_x\Bigl[\E^{\int_{0}^{t}[f(X_s)-\Lambda(f)]\,\D{s}}\Bigr]
\end{equation*}
is the unique positive solution of the MPE in \eqref{E1-Pois2}
which satisfies $\mu^*(\Psi_*^{-1})=1$.
\end{theorem}

\begin{proof}
These results are contained in
Theorem~\ref{T2.13} and Lemma~\ref{L2.15} in Section~\ref{S2.3}.
\qed\end{proof}

%%%%%%%%%%%%%%%%%%%%%%%%%%%%%%%%%%%%%%%%%%%%%%%%%%%%%%%%%%%%%%%%%%%%%%%%%%%%%%%%
\begin{remark}
The diffusion in \eqref{E-sde*} is studied extensively in \cite{Kaise-06}.
A diffusion of the form
\begin{equation}\label{ER1.6A}
\D{X}^*_{t} \;=\; \bigl(b(X^*_t) + \Hat{a} (X^*_t)
\grad\Breve\psi(X^*_t)\bigr)\,\D{t} + \upsigma(X^*_t)\,\D W^*_{t}
\end{equation}
is treated, where $\Breve\psi$ is a solution of (compare
with \eqref{E-logsde})
\begin{equation}\label{ER1.6B}
\tfrac{1}{2}a^{ij}\partial_{ij} \Breve\psi + \langle b,\grad \Breve\psi\rangle 
+\tfrac{1}{2}\langle\grad\Breve\psi,\Hat{a}\grad\Breve\psi\rangle
+f\;=\; \lambda\,.
\end{equation}
The assumptions imposed are that $a$, $\Hat{a}$, $b$ and $f$ are smooth
and for some constant $C>0$ it holds that
\begin{equation}\label{ER1.6C}
C^{-1}\,\abs{\xi}^2 \;\le\; 
\sum_{i,j=1}^{d} \max\,\{a^{ij}(x)\,,\Hat{a}^{ij}(x)\}\,\xi_{i}\xi_{j}
\;\le\;C^{-1}\,\abs{\xi}^2
\qquad\forall\,\xi\,,\,x\in\Rd\,.
\end{equation}
In addition they assume \eqref{E-H2}.
Under these assumptions, they show that the set of $\lambda$
for which \eqref{ER1.6B} has a solution is of
the form $[\lambda_*,\infty)$,
and that when $\lambda=\lambda_*$ the solution of
\eqref{ER1.6B} is unique and
\eqref{ER1.6A} is positive recurrent.
In addition, for $\lambda>\lambda_*$, \eqref{ER1.6A} is transient.

In comparing their results with ours, it is evident that
\eqref{ER1.6B} is more general than the Poisson equation considered
here and there are no assumptions concerning near-monotonicity for $f$.
On the other hand, we do not require any smoothness of the coefficients,
nor do we require \eqref{ER1.6C}.
\end{remark}

Next, we consider the following hypothesis.

%%%%%%%%%%%%%%%%%%%%%%%%%%%%%%%%%%%%%%%%%%%%%%%%%%%%%%%%%%%%%%%%%%%%%%%%%%%%%%%%
\begin{description}
\item[(H3)] For some $\varepsilon>0$,
$\Lambda\bigl((1+\varepsilon)f\bigr)<\infty$,
and
$f$ is near monotone relative to
$\frac{1}{\varepsilon}\bigl[\Lambda\bigl((1+\varepsilon)f\bigr)-\Lambda(f)\bigr]$.
\end{description}

If $f$ is inf-compact, then (H3) is equivalent to the statement
that $\Lambda\bigl((1+\varepsilon)f\bigr)<\infty$.
On the other hand, since
$\frac{1}{\varepsilon}\bigl[\Lambda\bigl((1+\varepsilon)f\bigr)-\Lambda(f)\bigr]
\ge \Lambda(f)$ by the convexity of $f\mapsto \Lambda(f)$ (see
Lemma~2.2\,(b)), it follows that (H3)\;$\Rightarrow$\;(H1).
It thus follows from Definition~\ref{D1.1}, that under (H3)
$(1+\varepsilon) f$ is near-monotone relative to
$\Lambda\bigl((1+\varepsilon)f\bigr)$.
We have the following theorem, whose proof is at
the end of Section~\ref{S2}.

%%%%%%%%%%%%%%%%%%%%%%%%%%%%%%%%%%%%%%%%%%%%%%%%%%%%%%%%%%%%%%%%%%%%%%%%%%%%%%%%
\begin{theorem}\label{T1.8}
Under {\upshape(H1)} and {\upshape(H3)},
the diffusion in \eqref{E-sde*} 
is geometrically ergodic, and there exist
positive constants $\gamma$, $\kappa^*$ and $\beta^*$, 
depending on $\varepsilon$, such that
if $g\,\colon\Rd\to\RR$ is any locally bounded measurable function
satisfying $\norm{g}_{\Psi^\gamma}<\infty$, it holds that
\begin{equation}\label{ET1.8A}
\babs{\Exp^*_x [g(X_t)] - \mu^*(g)} \;\le\; \kappa^* \E^{-\beta^* t}\,
\norm{g}_{\Psi^\gamma}
\bigl(1+\Psi^\gamma(x)\bigr)\qquad\forall\, t>0\,.
\end{equation}
Moreover, if $f$ is inf-compact then $\gamma$ can be chosen
arbitrarily close to $\varepsilon$.
\end{theorem}

\begin{proof}
The proof of this theorem is in Section~\ref{S2.3}.
\qed
\end{proof}

%%%%%%%%%%%%%%%%%%%%%%%%%%%%%%%%%%%%%%%%%%%%%%%%%%%%%%%%%%%%%%%%%%%%%%%%%%%%%%%%
\begin{remark}
Hypothesis (H3) is often used in the literature \cite{Balaji-00,Wu-94}.
As seen from Theorem~\ref{T1.8}, under (H3) the ground state diffusion
\eqref{E-sde*} is geometrically ergodic with a `storage function' $\Psi^\gamma$
for some $\gamma>0$.
Nevertheless, even under (H2) the ground state diffusion is geometrically ergodic
as can be seen by \eqref{EC2.14A}.
The situation seems to be different for denumerable Markov chains
where unless (H3) holds the twisted kernel cannot be geometrically
ergodic \cite[Theorem~5.1\,(ii)]{Balaji-00}.
\end{remark}

%%%%%%%%%%%%%%%%%%%%%%%%%%%%%%%%%%%%%%%%%%%%%%%%%%%%%%%%%%%%%%%%%%%%%%%%%%%%%%%%
In closing this section, let us mention that
it is a direct consequence of Jensen's inequality
that under (H1), the risk sensitive value $\Lambda(f)$ is not
less that the ergodic value $\mu(f)$, where $\mu$ is the
invariant probability measure of \eqref{E1-sde2}.
The difference $\Lambda(f)-\mu(f)$ can be quantified as the
following lemma shows.
To accomplish this we use the equation arising from \eqref{E1-Pois}
under the transformation $\psi=\log\Psi$, which takes the form 
\begin{equation}\label{E-logsde}
\tfrac{1}{2}a^{ij}\partial_{ij} \psi + \langle b,\grad \psi\rangle 
+\tfrac{1}{2}\langle\grad\psi,a\grad \psi\rangle \;=\; \Lambda(f) -f\,.
\end{equation}

%%%%%%%%%%%%%%%%%%%%%%%%%%%%%%%%%%%%%%%%%%%%%%%%%%%%%%%%%%%%%%%%%%%%%%%%%%%%%%%%
\begin{lemma}%\label{L1.9}
Let $\sG\df \langle\grad\psi,a\grad \psi\rangle$, with $\psi=\log\Psi$.
Under {\upshape(H1)}, it holds that
\begin{equation}\label{EL1.9A}
\tfrac{1}{2}\,\mu(\sG) + \mu(f) \;=\; \Lambda(f)\,,
\end{equation}
where $\mu$ is the
invariant probability measure of \eqref{E1-sde2}.
\end{lemma}

\begin{proof}
By It\^o's formula, we obtain from \eqref{E-logsde}
\begin{equation}\label{EL1.9B}
\Exp_x\bigl[\psi(X_{t\wedge\uptau_n})\bigr] - \psi(x)
+\frac{1}{2}\,\Exp_x\biggl[\int_0^{t\wedge\uptau_n} \sG(X_{s})\,\D{s}\biggr]
+ \Exp_x\biggl[\int_0^{t\wedge\uptau_n} f(X_{s})\,\D{s}\biggr]
\;=\; \Lambda(f)\,\Exp_x[t\wedge\uptau_n]\,.
\end{equation}
Since $f$ is near-monotone relative to $\Lambda(f)$,
\eqref{E2-Pois} takes the form $\Lg\Psi \le k_0 - k_1\Psi$ for
some positive constants $k_0$ and $k_1$.
It then follows by \cite[Lemma~3.7.2]{book}, that any $h\in\sorder(\Psi)$
satisfies
\begin{equation}\label{EL1.9C}
\Exp_x\bigl[h(X_{t\wedge\uptau_n})\bigr]\;
\xrightarrow[n\to\infty]{}\;\Exp_x\bigl[h(X_{t})\bigr]\,,
\quad\text{and}\quad
t^{-1}\,\Exp_x\bigl[h(X_{t})\bigr]
\xrightarrow[n\to\infty]{}\;0\,.
\end{equation}
Since $f$ is bounded below, we may assume without loss of generality
that it is nonnegative.
We first take limits as $n\to\infty$ in \eqref{EL1.9B}, using \eqref{EL1.9C}
with $\psi\in\sorder(\Psi)$
and monotone convergence for the integrals, then divide by $t$
and take limits as $t\to\infty$, using again \eqref{EL1.9B} and
Birkhoff's ergodic theorem to obtain \eqref{EL1.9A}.
\qed\end{proof}

%%%%%%%%%%%%%%%%%%%%%%%%%%%%%%%%%%%%%%%%%%%%%%%%%%%%%%%%%%%%%%%%%%%%%%%%%%%%%%%%
\subsection{Some basic results from the theory of second order elliptic pdes}
\label{S1.4}
%%%%%%%%%%%%%%%%%%%%%%%%%%%%%%%%%%%%%%%%%%%%%%%%%%%%%%%%%%%%%%%%%%%%%%%%%%%%%%%%

In this paper we use some basic properties of elliptic pdes
which we describe next.
The first is Harnack's inequality that plays a central role in
 the study of elliptic equations, and can be stated as follows
\cite[Theorem~9.1]{GilTru}.
Suppose that $\phi\in\Sobl^{2,p}(B_{R+1})$, $p>d$, $R>0$, is a positive function
that solves $\Lg\phi +h \phi=0$ on $B_{R+1}$, with $h\in\Lp^\infty(B_{R+1})$.
Then there exists a constant $C_H$ depending only on $R$, $d$, the constants
$C_{R+1}$ and $C_0$ in (A1)--(A3), and $\norm{h}_\infty$ such
that
\begin{equation}\label{E-Harn}
\phi (x) \;\le\; C_H\,\phi(y)\qquad \forall\, x,y\in B_R\,.
\end{equation}
Relative weak compactness of a family of functions in $\Sobl^{2,p}(B_{R+1})$
can be obtained as a result of the following well-known a priori estimate
\citep[Lemma~5.3]{ChenWu}.
If $\varphi\in\Sobl^{2,p}(B_{R+1})\cap\Lp^{p}(B_{R+1})$, with $p\in(1,\infty)$,
then
\begin{equation}\label{E-apriori}
\bnorm{\varphi}_{\Sob^{2,p}(B_R)}\;\le\;C\,
\Bigl(\bnorm{\varphi}_{\Lp^{p}(B_{R+1})}
+\bnorm{\Lg\varphi}_{\Lp^{p}(B_{R+1})}\Bigr)\,,
\end{equation}
with the constant $C$ depending only on $d$, $R$, $C_{R+1}$, and $C_0$.
This estimate along with the compactness of the embedding
$\Sob^{2,d}(B_R)\hookrightarrow \Cc^{1,r}(\Bar{B}_R)$, for $p>d$
and $r<1-\frac{d}{p}$ (see \citep[Proposition~1.6]{ChenWu})
imply the equicontinuity of any family of functions $\varphi_n$
which satisfies
$\sup_n\,\bigl(\norm{\varphi_n}_{\Lp^{p}(B_{R+1})}
+\norm{\Lg\varphi_n}_{\Lp^{p}(B_{R+1})}\bigr)<\infty$, $p>d$.
We also frequently use the weak and strong maximum principles
in the following form \cite[Theorems~9.5 and 9.6, p.~225]{GilTru}.
The \emph{weak maximum principle} states that if
$\varphi,\psi\in\Sobl^{2,d}(D)\cap\Cc(\Bar{D})$ satisfy
$(\Lg+h)\varphi=(\Lg+h)\psi$ in a bounded domain $D\subset\Rd$,
with $h\in\Lp^d(D)$, $h\leq 0$
and $\varphi=\psi$ on $\partial{D}$, then $\varphi=\psi$ in $D$.
On the other hand, the \emph{strong maximum principle} states that
if $\varphi\in\Sobl^{2,d}(D)$ satisfies
$(\Lg+h)\varphi\le0$ in a bounded domain $D$, with $h=0$ ($h\le0$),
then $\varphi$ cannot attain a minimum (nonpositive minimum)
in $D$ unless it is a constant.
We often use the following variation of the strong maximum principle.
If $\varphi\ge0$ and $(\Lg+h)\varphi\le0$ in a bounded domain $D\subset\Rd$,
then $\varphi$ is either positive on $D$ or identically equal to $0$.
This follows from the general statement by writing $(\Lg+h)\varphi\le0$
as
\begin{equation*}
(\Lg-h^-)\varphi\;\le\;-h^+\varphi\;\le\;0\,.
\end{equation*}

In this paper, we use the theory of elliptic pdes
to obtain limits of sequences of solutions to
\eqref{E1-Pois} as follows:
Suppose $\varphi_n\in\Sobl^{2,p}(\Rd)$, $p>d$, $n\in\NN$, is a sequence
of nonnegative functions satisfying
\begin{equation*}
\Lg \varphi_n + (f_n-\lambda_n) \varphi_n\;=\;-\alpha_n
\quad\text{on\ }\Rd\,,\quad \forall\,n\in\NN\,,
\end{equation*}
where 
$\{\lambda_n\}$ and
$\{\alpha_n\}$ are bounded sequences of nonnegative real numbers,
and $f_n$ is locally bounded
and converges to some $f\in\Lpl^\infty(\Rd)$ uniformly on compact sets.
Suppose also that $\{\varphi_n(0)\}$ is a bounded sequence of positive numbers.
Then by the extension of the Harnack inequality for a class of
superharmonic functions in \cite{AA-Harnack}
(see also \cite[Theorem~A.2.13]{book})
it follows that $\sup_{n\in\NN}\,\norm{\varphi_n}_{\infty,B_R}<\infty$ for every $R>0$.
Thus, by the a priori estimate in \eqref{E-apriori} and
the compactness of the embedding
$\Sob^{2,d}(B_R)\hookrightarrow \Cc^{1,r}(\Bar{B}_R)$,
the sequence $\{\varphi_n\}$ along with its first derivatives
are H\"older equicontinuous when restricted to any ball $B_R$.
Thus, given any diverging sequence $\{k_n\}\subset\NN$ we can extract a subsequence
also denoted as $\{k_n\}$ along which $\varphi_{k_n}$ converges to
some $\varphi\in\Sobl^{2,p}(\Rd)$, $p>d$, and $\lambda_{k_n}$ and
$\alpha_{k_n}$ converge to some constants $\lambda$ and $\alpha$ respectively,
and satisfy $\Lg \varphi + (f-\lambda) \varphi\;=\;-\alpha$ on $\Rd$.
In view of this convergence property, when we refer to a ``limit point'' of
$\{\varphi_n\}$ we mean a limit obtained as in the above procedure.

Completely analogous is the situation with solutions of \eqref{E1-HJB}.
Moreover, since under the current assumptions, the
map $x\mapsto \min_u \bigl[\langle b(x,u), \grad\psi(x)\rangle + c(x,u)\psi(x)\bigr]$,
is locally H\"older continuous, for any $\psi\in \Cc^{1,r}_{\mathrm{loc}}(\Rd)$,
then the compactness of the embedding
$\Sob^{2,d}(B_R)\hookrightarrow \Cc^{1,r}(\Bar{B}_R)$ together with
elliptic regularity \cite[Theorem~9.19]{GilTru} implies that solutions
are in $\Cc^2(\Rd)$, a fact which we use often.

%%%%%%%%%%%%%%%%%%%%%%%%%%%%%%%%%%%%%%%%%%%%%%%%%%%%%%%%%%%%%%%%%%%%%%%%%%%%%%%%
\section{Proofs of the results on the multiplicative Poisson equation}\label{S2}
%%%%%%%%%%%%%%%%%%%%%%%%%%%%%%%%%%%%%%%%%%%%%%%%%%%%%%%%%%%%%%%%%%%%%%%%%%%%%%%%

In this section we establish basic properties of the MPE through
lemmas that lead to the proof of Theorems~\ref{T1.5}--\ref{T1.8}.
Throughout the rest of the section $f\colon\Rd\to\RR$ is a locally bounded measurable
function, and $\Lg$ is as defined in \eqref{E-Lg}.
Also, as mentioned earlier, we assume that for the sde in \eqref{E1-sde2},
$\upsigma$ and $b$ satisfy (A2)--(A3),  $\upsigma$ is locally  Lipschitz
(as in (A1)), and $b$ is measurable.

%%%%%%%%%%%%%%%%%%%%%%%%%%%%%%%%%%%%%%%%%%%%%%%%%%%%%%%%%%%%%%%%%%%%%%%%%%%%%%%%
\subsection{Some basic results on the MPE}\label{S2.1}
%%%%%%%%%%%%%%%%%%%%%%%%%%%%%%%%%%%%%%%%%%%%%%%%%%%%%%%%%%%%%%%%%%%%%%%%%%%%%%%%

Some well known properties of the MPE are summarized in the following lemma.

%%%%%%%%%%%%%%%%%%%%%%%%%%%%%%%%%%%%%%%%%%%%%%%%%%%%%%%%%%%%%%%%%%%%%%%%
\begin{lemma}\label{L2.1}
Let $f$ be near-monotone relative to $\Lambda$,
and $(\Psi,\Lambda)\in\Sobl^{2,p}(\Rd)\times\RR$, $p\ge1$,
be a nonnegative solution to \eqref{E1-Pois} satisfying $\Psi(0)>0$.
Then the following are equivalent:
\begin{itemize}
\item[{\upshape(}a\/{\upshape)}]
The diffusion \eqref{E1-sde2} is recurrent.

\item[{\upshape(}b\/{\upshape)}]
$\inf_{\Rd}\,\Psi>0$.

\item[{\upshape(}c\/{\upshape)}]
$\min_{\Rd}\,\Psi =\min_{\sK} \,\Psi$, with $\sK$ as in Definition~\ref{D1.1}.

\item[{\upshape(}d\/{\upshape)}]
The function $\Psi$ is inf-compact.

\item[{\upshape(}e\/{\upshape)}]
The diffusion \eqref{E1-sde2} is geometrically ergodic, i.e.,
it is positive recurrent with invariant probability measure $\mu$,
and there exist positive constants $\kappa$ and $\beta$, such that
if $g\,\colon\Rd\to\RR$ is any locally bounded measurable function
satisfying $\norm{g}_{\Psi}<\infty$, it holds that
\begin{equation}\label{E-erg}
\babs{\Exp_x [g(X_t)] - \mu(g)} \;\le\; \kappa \E^{-\beta t}\,\norm{g}_{\Psi}
\bigl(1+\Psi(x)\bigr)\qquad\forall\, t>0\,.
\end{equation}

\item[{\upshape(}f\/{\upshape)}]
$\Lambda_x(f)\le\Lambda$ for all $x\in\Rd$.

\item[{\upshape(}g\/{\upshape)}]
$\Lambda_x(f)\le\Lambda$ for some $x\in\Rd$.
\end{itemize}
\end{lemma}

\begin{proof}
Let $\sB$ be a bounded ball and $\delta>0$ a constant,
such that $f-\Lambda>\delta$ in $\sB^c$.
Then with $\uuptau\equiv\uptau(\sB^c)$ we have
\begin{align*}
\Psi(x) &\;\ge\; \Exp_{x}\,
\bigl[\E^{\delta\uuptau}\,\Psi(X_{\uuptau})\,\Ind\{\uuptau<\infty\}\bigr]
\nonumber\\[5pt]
&\;\ge\; \Bigl(\inf_{\partial\sB}\; \Psi\Bigr)\,
\Exp_{x}\bigl[\E^{\delta\uuptau}\,\Ind\{\uuptau<\infty\}\bigr]
\qquad\forall\, x\in\Bar{\sB}^c\,.
\end{align*}
If (a) holds, then since $\inf_{\sB}\, \Psi>0$ by the Harnack inequality,
and $x\mapsto \Exp_{x}\bigl[\E^{\delta\uuptau}\,\Ind\{\uuptau<\infty\}\bigr]$
is inf-compact,
by Assumption~(A2),
then both (c) and (d) follow, and of course either (c) or (d)
imply (b).
Thus to complete the proof it suffices to
show that (b) implies (e) and (f) and that (g) implies (a).

Suppose $\inf_{\Rd}\,\Psi>0$.
Then since $\Lg \Psi < -\delta\Psi$ on $\sB^c$,
(e) follows (see \cite{Down-95,Fort-05}).
Also, by \eqref{E1-Pois} and Fatou's lemma we have
\begin{align*}
\Psi(x)&\;\ge\; \Exp_{x}\,\Bigl[\E^{\int_{0}^{T}
[f(X_{t})-\Lambda]\,\D{t}}\,\Psi(X_{T})\Bigr]\\[5pt]
&\;\ge\; \Bigl(\inf_{\Rd}\; \Psi\Bigr)\,\Exp_{x}\,
\Bigl[\E^{\int_{0}^{T}[f(X_{t})-\Lambda]\,\D{t}}\Bigr]\,,
\end{align*}
and (f) follows by taking log, dividing by $T$ and letting $T\to\infty$.

It is obvious that (f)\;$\Rightarrow$\;(g).
We next show that (g)\;$\Rightarrow$\;(a), and this completes the proof.
Using Jensen's inequality we have
\begin{equation*}
\limsup_{T\to\infty}\;\frac{1}{T}\,\int_{0}^{T} 
\Exp_x[f(X_t)]\,\D{t}\;\le\;
\Lambda_x(f) \;\le\;\Lambda\,.
\end{equation*}
A standard argument then shows that
any limit point of ergodic occupation measures
along some sequence $T_n\to\infty$ in $\cP(\Rd\cup\{\infty\})$,
i.e., the set of probability measures on the one point compactification
of $\Rd$, takes the form $\rho \mu + (1-\rho)\delta_{\infty}$,
with $\rho\in(0,1)$ and that $\mu$ is an invariant probability measure of
\eqref{E1-sde2}
(see Lemma~3.4.6 and Theorem~3.4.7 in \cite{book}).
This of course implies that \eqref{E1-sde2} is positive recurrent.
\qed\end{proof}

The following lemma summarizes some results from \cite{Berestycki-94,Berestycki-15,
Quaas-08a} on eigenvalues of
the Dirichlet problem for the operator $\Lg$.
Recall that subsets $A$ and $B$ of $\Rd$ we write $A\Subset B$
if $\Bar{A}\subset B$. 

%%%%%%%%%%%%%%%%%%%%%%%%%%%%%%%%%%%%%%%%%%%%%%%%%%%%%%%%%%%%%%%%%%%%%%%%%%%%%%%%
\begin{lemma}\label{L2.2}
For each $r\in(0,\infty)$ there exists a unique pair
$(\widehat\Psi_{r},\Hat\lambda_{r})
\in\bigl(\Sob^{2,p}(B_{r})\cap\Cc(\Bar{B}_{r})\bigr)\times\RR$,
for any $p>d$, satisfying
$\widehat\Psi_{r}>0$ on $B_{r}$, $\widehat\Psi_{r}=0$ on
$\partial B_{r}$, and $\widehat\Psi_{r}(0)=1$,
which solves
\begin{equation}\label{L2.2A}
\Lg \widehat\Psi_{r}(x) + f(x)\,\widehat\Psi_{r}(x)
\;=\; \Hat\lambda_{r}\,\widehat\Psi_{r}(x)
\qquad\text{a.e.\ }x\in B_{r}\,.
\end{equation}
Moreover, solution has the following properties:
\begin{itemize}
\item[{\upshape(}a{\upshape)}]
The map $r\mapsto\Hat\lambda_{r}$ is continuous and strictly increasing.
\item[{\upshape(}b{\upshape)}]
In its dependence on the function $f$,
 $\Hat\lambda_r$ is nondecreasing, convex, and Lipschitz continuous
 (with respect to the $\Lp^{\infty}$ norm) with Lipschitz constant $1$.
In addition, if $f\lneqq f'$ then $\Hat\lambda_r(f)<\Hat\lambda_r(f')$.
\end{itemize}
\end{lemma}

\begin{proof}
Existence and uniqueness of the solution follow
by \cite[Theorem~1.1]{Quaas-08a} (see also \cite{Berestycki-94}).
Part (a) follows by \cite[Theorem~1.10]{Berestycki-15},
and (iii)--(iv) of \cite[Proposition~2.3]{Berestycki-15}
while part [(b)] follows by \cite[Proposition~2.1]{Berestycki-94}.
\end{proof}

We refer to the pair
$(\widehat\Psi_{r},\Hat\lambda_{r})$ in Lemma~\ref{L2.2}
as the Dirichlet eigensolution
of the MPE on $B_r$.
We also call $\widehat\Psi_{r}$, and $\Hat\lambda_{r}$, the
Dirichlet
eigenfunction and eigenvalue on $B_r$, respectively.

%%%%%%%%%%%%%%%%%%%%%%%%%%%%%%%%%%%%%%%%%%%%%%%%%%%%%%%%%%%%%%%%%%%%%%%%%%%%%%%%
\begin{lemma}\label{L2.3}
Let $(\widehat\Psi_{r},\Hat\lambda_{r})$ be as in Lemma~\ref{L2.2}.
Then
\begin{itemize}
\item[{\upshape(}i\/{\upshape)}]
$\Hat\lambda_{r}< \inf_{x\in B_r}\,\Lambda_x(f)$ for all $r\in(0,\infty)$.

\item[{\upshape(}ii\/{\upshape)}]
If $\Lambda(f)<\infty$, then
every sequence $\{r_n\}$, with $r_n\to\infty$ contains a
subsequence also denoted as $\{r_n\}$, along which
$(\widehat\Psi_{r_n},\Hat\lambda_{r_n})$ in \eqref{L2.2A}
converge to some $(\widehat\Psi,\Hat\lambda)
\in\Sobl^{2,p}(\Rd)\times\RR$, $p>d$,
which solves the MPE \eqref{E1-Pois}.

\item[{\upshape(}iii\/{\upshape)}]
If the diffusion in \eqref{E1-sde2} is recurrent and
$f$ is near-monotone relative to $\Lambda(f)$, then
$\Lambda_x(f)=\Lambda(f)$ for all $x\in\Rd$, and
$\Hat\lambda_{r}\nearrow \Lambda(f)$ as $r\to\infty$.
\end{itemize}
\end{lemma}

\begin{proof}
Since $\widehat\Psi_{r}=0$ on $\uptau_r=\uptau(B_r^c)$, applying
It\^o's formula to \eqref{L2.2A} we obtain
\begin{align*}
\widehat\Psi_r(x) &\;=\;
\Exp_{x}\,\Bigl[\E^{\int_{0}^{T} [f(X_{t})-\Hat\lambda_r]\,\D{t}}\,
\widehat\Psi_r(X_{T})\,\Ind_{\{T\le\uptau_r\}}\Bigr]\\[5pt]
&\;\le\; \norm{\widehat\Psi_r}_{\infty,B_r}\,
\Exp_{x}\,\Bigl[\E^{\int_{0}^{T}[f(X_{t})-\Hat\lambda_r]\,\D{t}}\Bigr]
\qquad\forall\,(T,x)\in\RR_+\times B_r\,.
\end{align*}
Therefore, we have
\begin{equation*}
\Lambda_x(f)\;=\;\limsup_{T\to\infty}\;\frac{1}{T}\;\log\,
\Exp_{x}\,\Bigl[\E^{\int_{0}^{T}f(X_{t})\,\D{t}}\Bigr]
\;\ge\; \Hat\lambda_{r}
\qquad\forall\,x\in B_r\,,
\end{equation*}
which proves part (i).
Part (ii) is as in \cite[Lemma~2.1]{Biswas-11a}.

Suppose $\Hat\lambda_{r_n}\to\Hat\lambda$ along some sequence $r_n\nearrow\infty$.
By part (ii) there exists a subsequence also denoted as $\{r_n\}$ along which
$\widehat\Psi_{r_n}\to \widehat\Psi\in\Sobl^{2,p}(\Rd)\times\RR$, $p>d$,
and the pair $(\widehat\Psi,\Hat\lambda)$ solves the MPE \eqref{E1-Pois}.
By Lemma~\ref{L2.1}\,(f) we have
$\Hat\lambda\ge\Lambda_x(f)$ for all $x\in\Rd$, which
combined with part (i) results in equality.
\qed\end{proof}

By Lemma~\ref{L2.1}, (H1) implies that \eqref{E1-sde2} is positive recurrent.
Also without loss of generality we may assume that $\Lambda(f)>0$, otherwise
we translate $f$ by a constant to attain this.
Lemma~\ref{L2.3} shows  that there exists
a positive solution  $\Psi\in\Sobl^{2,p}(\Rd)$, $p>d$, to
\begin{equation}\label{E2-Pois}
\Lg \Psi(x) + f(x)\,\Psi(x)
\;=\; \Lambda(f)\,\Psi(x) \qquad\text{a.e.\ }x\in\Rd\,.
\end{equation}
Then, necessarily $\inf_{\Rd}\,\Psi>0$ by Lemma~\ref{L2.1}.
Unless stated otherwise, we use the symbol
$\Psi$ to denote a positive solution of \eqref{E2-Pois}.

We continue with the proof of Theorem~\ref{T1.4}.

\begin{pT1.4}
Let $\epsilon>0$ and $\sK_\epsilon$ a compact set in $\Rd$ such that
$\essinf_{\sK^c_\epsilon} (f-\Lambda(f)-\epsilon)\ge0$.
If a positive function $\varphi\in\Sobl^{2,d}(\Rd)$, satisfies
$\Lg\varphi + (f-\lambda)\varphi\le 0$ for some $\lambda\le \Lambda(f)+\epsilon$,
then by the proof of Lemma~\ref{L2.1}\,(d), we have
$\inf_{\Rd}\,\varphi>0$.
On the other hand, if $\Lg\varphi + (f-\lambda)\varphi\le 0$
and $\inf_{\Rd}\,\varphi>0$, then as in the proof of Lemma~\ref{L2.1}\,(f)
we obtain $\Lambda_x(f)\le\lambda$ for all $x\in\Rd$.
Thus $\Lambda(f)\le\hat\Lambda(f)$, which implies that
$\Lambda(f)\le \Hat\Lambda(f)\le\doublehat\Lambda(f)$.
However, $\Psi$ satisfies
$\Lg\Psi + \bigl(f-\Lambda(f)\bigr)\Psi\le 0$ and therefore
$\Lambda(f)\ge\doublehat\Lambda(f)$, which results in equality.
The proof is complete.
\qed
\end{pT1.4}

%%%%%%%%%%%%%%%%%%%%%%%%%%%%%%%%%%%%%%%%%%%%%%%%%%%%%%%%%%%%%%%%%%%%%%%%%%%%%%%%
\subsection{Results concerning Theorem~\ref{T1.5}}\label{S2.2}
%%%%%%%%%%%%%%%%%%%%%%%%%%%%%%%%%%%%%%%%%%%%%%%%%%%%%%%%%%%%%%%%%%%%%%%%%%%%%%%%

Recall the definitions of the ground state diffusion in \eqref{E-sde*},
the operator $
$ in \eqref{E-Lg*}, and the ground state
semigroup in \eqref{E-cT*}.
The operators $\Lg$ and $\Lg^*$ are linked via
the following useful identity.
If  $\Psi$ is a positive solution of \eqref{E2-Pois}
and $\Tilde{\Psi}$ is a positive solution of
\begin{equation*}
\Lg \Tilde\Psi(x) + \Tilde{f}(x)\,\Tilde\Psi(x)
\;=\; \Lambda(\Tilde{f})\,\Tilde\Psi(x) \qquad\text{a.e.\ }x\in\Rd\,,
\end{equation*}
then
\begin{equation}\label{E-Lg*Id}
\Lg^* \bigl(\Tilde\Psi\,\Psi^{-1}\bigr) \;=\; \bigl(\Lambda(\Tilde{f}) - \Lambda(f)
+ f - \Tilde{f}\bigr)\,\Tilde\Psi\,\Psi^{-1}\,.
\end{equation}

Throughout this section we assume (H1).
Recall the definition in \eqref{E-FK}.
We start with the following lemma.

%%%%%%%%%%%%%%%%%%%%%%%%%%%%%%%%%%%%%%%%%%%%%%%%%%%%%%%%%%%%%%%%%%%%%%%%%%%%%%%%
\begin{lemma}\label{L2.4}
Let 
\begin{equation*}
\sM_t\;\df\;\exp\biggl(\int_0^t \bigl\langle
\upsigma\transp(X_s)\grad\psi(X_s),\D W_s\bigr\rangle\,\D{s} - \frac{1}{2}
\int_0^t
\langle\grad\psi,a\grad \psi\rangle(X_s)\,\D{s}\biggr)\,.
\end{equation*}
Then $\bigl(\sM_t,\sF^X_t\bigr)$ is a martingale if and only if
\begin{equation}\label{EL2.4A}
\E^{-\Lambda(f)t} P^f_t\Psi (x) = \Psi(x)\qquad\forall\,(t,x)\in\RR_+\times\Rd\,.
\end{equation}
\end{lemma}

\begin{proof}
Recall that $\psi=\log\Psi$.
We use the first exit times from $B_n$, i.e., $\uptau_n$
as localization times.
Since the drift of \eqref{E1-sde2} satisfies (A2),
it is well-known that $\uptau_n\to \infty$ as $n\to\infty$ $\Prob_x$-a.s.
Applying the It\^{o}-Krylov formula \cite[p.~122]{Krylov} to \eqref{E1-sde2}
we obtain
\begin{align}\label{EL2.4B}
\psi(X_{t\wedge\uptau_n}) -\psi(x) &\;=\;
\int_0^{t\wedge\uptau_n}\mathcal{L}(X_{s})\,\D{s}
+ \int_0^{t\wedge\uptau_n} \langle \grad\psi(X_s), \upsigma(X_s)\,\D{W_s}\rangle
\nonumber\\[3pt]
&\;=\; \int_0^{t\wedge\uptau_n}
\Bigl( \Lambda(f)- f(X_s)
- \tfrac{1}{2}\langle\grad\psi,a\grad \psi\rangle(X_s)\Bigr)\,\D{s}
\nonumber\\[3pt]
&\mspace{300mu}
+ \int_0^{t\wedge\uptau_n} \langle \grad\psi(X_s), \upsigma(X_s) \D{W_s}\rangle\,.
\end{align}
By \eqref{EL2.4B} we have
\begin{align}\label{EL2.4D}
\Exp_x&\left[\E^{\int_0^{t\wedge \uptau_n}
[f(X_s)-\Lambda(f)]\,\D{s}}\,g(X_{t\wedge\uptau_n})\right]\nonumber\\[3pt]
&\;=\; \Exp_x\biggl[g(X_{t\wedge\uptau_n}) \exp\biggl(
-\psi(X_{t\wedge\uptau_n}) + \psi(x) \nonumber\\[3pt]
&\mspace{30mu}
+\int_0^{t\wedge\uptau_n} \langle \grad\psi(X_s), \upsigma(X_s) \D{W_s}\rangle
-\int_0^{t\wedge\uptau_n} \tfrac{1}{2}\langle\grad\psi,a\grad \psi\rangle(X_s)\,
\D{s}\biggr)\biggr]\nonumber\\[5pt]
&\;=\; \Psi(x) \Exp_x\bigl[g(X_{t\wedge\uptau_n}) \,
\Psi^{-1}(X_{t\wedge\uptau_n}) \,\sM_{t\wedge\uptau_n}\bigr]\,.
\end{align}
Thus choosing an increasing sequence $g_m\in\Cc_b(\Rd)$ which converges
to $\Psi$, uniformly on compact sets, we obtain
\begin{equation*}
\Exp_x\left[\E^{\int_0^{t\wedge \uptau_n}
[f(X_s)-\Lambda(f)]\,\D{s}}\,\Psi(X_{t\wedge\uptau_n})\right]
\;=\;\Psi(x) \Exp_x\bigl[\sM_{t\wedge\uptau_n}\bigr]
\end{equation*}
by monotone convergence.
It follows that $\Exp_x\bigl[\sM_{t\wedge\uptau_n}\bigr]=1$,
so that $\bigl(\sM_t,\sF^X_t\bigr)$ is a local martingale.

Next consider $g\in\Cc_c(\Rd)$.
By \eqref{EL2.4D} we obtain
\begin{equation*}
\Exp_x\left[\E^{\int_0^{t}
[f(X_s)-\Lambda(f)]\,\D{s}}\,g(X_t)\,\Psi(X_t)\,\Ind_{\{t\le\uptau_n\}}\right]
\;=\;\Psi(x) \Exp_x\bigl[g(X_t)\,\sM_{t}\,\Ind_{\{t\le\uptau_n\}}\bigr]
\end{equation*}
for all sufficiently large $n$, and therefore also
\begin{equation}\label{EL2.4E}
\Exp_x\left[\E^{\int_0^{t}
[f(X_s)-\Lambda(f)]\,\D{s}}\,g(X_t)\,\Psi(X_t)\right]
\;=\;\Psi(x) \Exp_x\bigl[g(X_t)\,\sM_{t}\bigr]
\end{equation}
by monotone convergence.
Thus evaluating \eqref{EL2.4E}
on some increasing sequence $g_n\in\Cc_c(\Rd)$ which converges
to $1$, uniformly on compact sets, and taking limits as $n\to\infty$,
using again monotone convergence and \eqref{EL2.4A}, we obtain
\begin{equation}\label{EL2.4F}
\Psi(x)\;=\;\Exp_x\left[\E^{\int_0^{t}
[f(X_s)-\Lambda(f)]\,\D{s}}\,\Psi(X_{t})\right]
\;=\;\Psi(x) \Exp_x\bigl[\sM_{t}\bigr]\,.
\end{equation}
If \eqref{EL2.4A} holds then $\Exp[\sM_{t}]=1$ by \eqref{EL2.4F}.
Thus $\bigl(\sM_t,\sF^X_t\bigr)$ is a martingale.
Conversely, if $\Exp[\sM_{t}]=1$, then \eqref{EL2.4F} implies \eqref{EL2.4A}.
This completes the proof.
\qed
\end{proof}

%%%%%%%%%%%%%%%%%%%%%%%%%%%%%%%%%%%%%%%%%%%%%%%%%%%%%%%%%%%%%%%%%%%%%%%%%%%%%%%%
\begin{lemma}\label{L2.5}
The sde in \eqref{E-sde*} has a unique strong solution $X^*$ which
exists for all $t>0$ if and only if $\bigl(\sM_t,\sF^X_t\bigr)$ is a martingale.
\end{lemma}

\begin{proof}
By \eqref{EL2.4D} we have
\begin{align}\label{EL2.5A}
\Exp_x\left[\E^{\int_0^{t\wedge \uptau_n}
[f(X_s)-\Lambda(f)]\,\D{s}}\,g(X_{t\wedge\uptau_n})\right]
&\;=\; \Psi(x) \Exp_x\bigl[g(X_{t\wedge\uptau_n}) \,
\Psi^{-1}(X_{t\wedge\uptau_n}) \,\sM_{t\wedge\uptau_n}\bigr]\nonumber\\[5pt]
&\;=\; \Exp_x^*\bigl[g(X^*_{t\wedge\uptau_n}) \,
\exp\bigl(-\psi(X^*_{t\wedge\uptau_n}) + \psi(x)\bigr)\bigr]\,,
\end{align}
where in the last line we use Cameron--Martin--Girsanov theorem \citep[p.~225]{LiSh-I}.

By the second equality in \eqref{EL2.5A} we have
\begin{equation}\label{EL2.5B}
\Exp_x\bigl[g(X_{t\wedge\uptau_n}) 
\,\Psi^{-1}(X_{t\wedge\uptau_n}) \sM_{t\wedge\uptau_n}\bigr]\;=\;
\Exp_x^*\bigl[g(X^*_{t\wedge\uptau_n})
\,\Psi^{-1}(X^*_{t\wedge\uptau_n})\bigr]
\end{equation}
for any $g\in\Cc_b(\Rd)$.
Let $h\colon\Rd\to[0,1]$ be a continuous function which is equal to $1$
on $B_{n-1}$, and vanishes on $B_n^c$.
Evaluating \eqref{EL2.5B} on $g= h\Psi$
we obtain $\Prob_x^* (t<\uptau_{n})
\ge \Exp_x\bigl[\sM_{t}\,\Ind_{\{t\le\uptau_{n-1}\}}\bigr]$
by \eqref{EL2.5B}.
Therefore, if $\bigl(\sM_t,\sF^X_t\bigr)$ is a martingale,
then with  $\uptau_\infty\df \lim_{n\to\infty}\,\uptau_n$,
we have $\Prob_x^* (t<\uptau_\infty)=\Exp_{x} [\sM_t]=1$.
Thus
\begin{equation*}
\Prob_x^* (\uptau_\infty=\infty) \;=\; \lim_{t\to\infty}\; 
\Prob_x^* (t<\uptau_\infty)\;=\;1\,,
\end{equation*}
which shows that the diffusion is regular.

The same argument shows that
$\Prob_x^* (t<\uptau_{n-1})
\le \Exp_x\bigl[\sM_{t}\,\Ind_{\{t\le\uptau_{n}\}}\bigr]$,
from which it follows that if the $X^*$ process is regular, then
$\bigl(\sM_t,\sF^X_t\bigr)$ is a martingale.
This completes the proof.
\qed
\end{proof}

%%%%%%%%%%%%%%%%%%%%%%%%%%%%%%%%%%%%%%%%%%%%%%%%%%%%%%%%%%%%%%%%%%%%%%%%%%%%%%%%
\begin{lemma}\label{L2.6}
If process $X^*$ in \eqref{E-sde*} is recurrent, then, for any $r>0$, we have
\begin{equation}\label{EL2.6A}
\Psi(x) \;=\; \Exp_x\Bigl[\E^{\int_{0}^{\uuptau_r}[f(X_s)-\Lambda(f)]\,\D{s}}\,
\Psi(X_{\uuptau_r})\Bigr]\,,\quad\text{for\ \ }x\in B_r^c\,.
\end{equation}
Conversely, if \eqref{EL2.6A} holds for some $r>0$ then
the process $X^*$ is recurrent.
\end{lemma}

\begin{proof}
Let $r>\epsilon>0$, and $h\colon\Rd\to[0,1]$ be a continuous function which equals
$1$ on $\bar{B}_{r}$ and vanishes on $B_{r+\epsilon}^c$.
Using \eqref{EL2.5A} with $g=h\Psi$ and localized at $\uuptau_r\wedge t\wedge \uptau_n$,
and then taking limits as $\epsilon\searrow0$, we obtain
\begin{equation}\label{EL2.6X}
\Exp_x\Bigl[\E^{\int_{0}^{t\wedge\uuptau_r}[f(X_s)-\Lambda(f)]\,\D{s}}\,
\Psi(X_{t\wedge\uuptau_r})\,\Ind_{\Bar{B}_r}(X_{t\wedge\uuptau_r})\,
\Ind_{\{t\wedge\uuptau_r<\uptau_n\}}\Bigr]
\;=\; \Psi(x)\,\Prob_x^* (\uuptau_r \le t\wedge\uptau_n)\,.
\end{equation}
By Fatou's lemma we have
\begin{equation*}
\Exp_x\Bigl[\E^{\int_{0}^{\uuptau_r}[f(X_s)-\Lambda(f)]\,\D{s}}\,
\Psi(X_{\uuptau_r})\Bigr] \;\le\; \Psi(x)\,,
\end{equation*}
and therefore, taking limits
in \eqref{EL2.6X} first as $t\to\infty$, and then as $n\to\infty$, we obtain
\begin{equation}\label{EL2.6B}
\Exp_x\Bigl[\E^{\int_{0}^{\uuptau_r}[f(X_s)-\Lambda(f)]\,\D{s}}\,
\Psi(X_{\uuptau_r})\Bigr] \;=\; \Psi(x)\,\Prob_x^* (\uuptau_r <\uptau_\infty)\,.
\end{equation}
Note that for fixed $n$ one can justify the limit as $t\to\infty$
of the term on the left hand side of \eqref{EL2.6X}
by dominated convergence.
If $X^*$ is recurrent (and regular) then $\Prob_x^* (\uptau_\infty=\infty)=1$,
and $\Prob_x^* (\uuptau_r <\infty)=1$ for any $r>0$, and thus
\eqref{EL2.6A} follows by \eqref{EL2.6B}.

Now suppose that \eqref{EL2.6A} holds for some $r>0$.
Then $\Prob_x^* (\uuptau_r <\uptau_\infty)=1$.
We claim that this implies that the process is regular,
i.e., $\Prob_x^* (\uptau_\infty=\infty)=1$.
To prove the claim,
let $\mathfrak{s}_{0}=0$, and for $k=0,1,\dotsc$ define inductively
an increasing sequence of stopping times by
\begin{equation*}
\begin{split}
\mathfrak{s}_{2k+1}
&\df\inf\;\{t>\mathfrak{s}_{2k}: X_{t}^*\in B^{c}_{2r}\}\,,\\[3pt]
\mathfrak{s}_{2k+2}
&\df \inf\; \{t>\mathfrak{s}_{2k+1}: X_{t}^*\in B_{r}\}\,.
\end{split}
\end{equation*}
It is clear that $\Prob_x^* (\uuptau_r <\uptau_\infty)=1$ implies that
$\Prob_x^*(\mathfrak{s}_{k}<\uptau_\infty)=1$.
It follows that $\mathfrak{s}_{n}\nearrow\infty$, $\Prob_x^*$-a.s.,
by a standard argument used in the
proof of \cite[Lemma~2.6.6]{book}.
This proves the claim.
Since the coefficients of $\Lg^*$ are locally bounded, Harnack's theorem
applies, and thus
a classical argument due to Hasminskii
shows that $\Prob_x^* (\uuptau_r <\infty)=1$ for some $r>0$ implies
that the same holds for all $r>0$ \cite[Lemma~2.1, p.~111]{Hasminskii}.
This  completes the proof.
\qed
\end{proof}

We have the following continuity property (compare with
\cite[Proposition~9.2]{Berestycki-15}), even though it is
not used in the proof of the main results.

%%%%%%%%%%%%%%%%%%%%%%%%%%%%%%%%%%%%%%%%%%%%%%%%%%%%%%%%%%%%%%%%%%%%%%%%%%%%%%%%
\begin{lemma}%\label{L2.7}
Suppose $f_n\in\Lpl^\infty(\Rd)$ is a sequence of nonnegative
functions
converging to $f\in\Lpl^\infty(\Rd)$,
such  that
\begin{enumerate}
\item
$\Lambda(f_n)\le \Lambda(f)$ for each $n\in\NN$.
\item
For some $\epsilon>0$
the set $D_\epsilon\;\df\;\cup_{n\in\NN} \{x\colon f_n(x)< \Lambda(f)+\epsilon\}$
is bounded.
\end{enumerate}
Then $\Lambda(f_n)\to\Lambda(f)$ as $n\to\infty$.
\end{lemma}

\begin{proof}
Suppose not.
Then  we have $\liminf_{n\to\infty}\,\Lambda(f_n)=\Bar\Lambda<\Lambda(f)$.
Let $\Psi_n\in\Sobl^{2,p}$,  $p>d$, solutions to the Poisson equation
\begin{equation*}
\Lg\Psi_n + f_n \Psi_n \;=\;\Lambda(f_n) \Psi_n\,,
\end{equation*}
satisfying $\Psi_n(0)=1$.
Then $\inf_{\Rd}\,\Psi_n = \inf_{\Bar{D}_\epsilon}\,\Psi_n$ by (2).
Let $\overline\Psi$ be any limit point of $\Psi_n$ as $n\to\infty$,
along some subsequence along which $\Lambda(f_n)$ converges to $\Bar\Lambda$.
Then we obtain
\begin{equation*}
\Lg\overline\Psi + f \overline\Psi \;=\;\Bar\Lambda \overline\Psi\,,
\end{equation*}
and $\overline\Psi$ satisfies
 $\inf_{\Rd}\,\overline\Psi= \inf_{\Bar{K}_\epsilon}\,\overline\Psi$.
Then by Lemma~\ref{L2.1}\,(f) we have $\Lambda(f)\le\Bar\Lambda$ which
contradicts the original hypothesis.
\qed
\end{proof}

%%%%%%%%%%%%%%%%%%%%%%%%%%%%%%%%%%%%%%%%%%%%%%%%%%%%%%%%%%%%%%%%%%%%%%%%%%%%%%%%
\begin{theorem}\label{T2.8}
Suppose that for some $\Hat{x}\in\Rd$ we have 
\begin{equation}\label{ET2.8A}
\Gamma(\Hat{x})\;\df\;\int_0^\infty\Exp_{\Hat{x}}\Bigl[\E^{\int_{0}^{t}
[f(X_s)-\Lambda(f)]\,\D{s}}\Bigr]\,\D{t}\;=\;\infty\,.
\end{equation}
Then \eqref{E2-Pois} has a unique positive solution
$\Psi\in\Sobl^{2,d}(\Rd)$, satisfying $\Psi(0)=1$.
In addition, $\Psi$ satisfies
\begin{equation}\label{ET2.8B}
\Psi(x) \;=\; \Exp_x\Bigl[\E^{\int_{0}^{\uptau(\sB^c)}[f(X_s)-\Lambda(f)]\,\D{s}}\,
\Psi(X_{\uptau(\sB^c)})\Bigr]
\qquad\forall\, x\in \sB^c
\end{equation}
for some ball $\sB$.
\end{theorem}

\begin{proof}
Let $\theta>0$ be any positive constant such
that $f$ is near monotone relative to $\Lambda(f)+2\theta$,
and fix some $\alpha\in(0,\theta)$.
In order to simplify the notation we define
\begin{equation*}
F_\alpha(x)\;\df\; f(x) - \Lambda(f) - \alpha\,,\qquad\text{and\ \ }
F_0(x)\;\df\; f(x) - \Lambda(f)\,,
\end{equation*}
and
\begin{equation*}
\Gamma_\alpha(x) \;\df\; \int_{0}^{\infty}
\Exp_x\Bigl[ \E^{\int_{0}^{t}
F_\alpha(X_s)\,\D{s}}\Bigr]\,\D{t}\,.
\end{equation*}
Without loss of generality we assume $\Hat{x}=0$.
Moreover, $\Gamma_{\alpha}(0)$ is finite follows
from the fact that 
\begin{align*}
\Exp_x\Bigl[ \E^{\int_{0}^{t}F_\alpha(X_s)\,\D{s}}\Bigr] &\;\le\;
\E^{-\alpha t}\, \Bigl(\inf_{\Rd}\;\Psi\Bigr)^{-1}  \,
\Exp_x\Bigl[ \E^{\int_{0}^{t}F_0(X_s)\, \D{s}}\Psi(X_t)\,\Bigr]\\
&\;\le\;  \E^{-\alpha t}\, \Bigl(\inf_{\Rd}\;\Psi\Bigr)^{-1} \, \Psi(x)
\quad \forall \; t>0\,.
\end{align*}

Recall the eigenvalues $\{\Hat\lambda_n\}$ in Lemma~\ref{L2.2}.
Since $\Hat\lambda_n<\Lambda(f)$, the principal eigenvalue of the operator
$-\Lg-F_\alpha$ on every $B_n$ is positive.
Thus by Proposition~6.2 and Theorem~6.1 in \cite{Berestycki-94},
for any $n\in\NN$, the Dirichlet problem
\begin{equation}\label{ET2.8C}
\Lg\varphi_{\alpha,n}(x)+ F_\alpha(x)\,\varphi_{\alpha,n}(x)
\;=\; -\Gamma_\alpha^{-1}(0)\qquad\text{a.e.\ }x\in B_n\,,\qquad
\varphi_{\alpha,n}=0\text{\ \ on\ \ }\partial B_n\,,
\end{equation}
has a unique solution $\varphi_{\alpha,n}\in\Sobl^{2,p}(B_n)\cap\Cc(\Bar{B}_n)$,
for any $p\ge1$.
In addition, by the \emph{refined maximum principle} in
\cite[Theorem~1.1]{Berestycki-94} $\varphi_{\alpha,n}$ is nonnegative.
Since $\Gamma^{-1}_{\alpha}(0)>0$, then $\varphi_{\alpha,n}$
cannot be identically equal to $0$.
Thus writing \eqref{ET2.8C} as
\begin{equation*}
\Lg\varphi_{\alpha,n}- F_\alpha^-\,\varphi_{\alpha,n}
\;=\; - F_\alpha^+\,\varphi_{\alpha,n} - \Gamma_\alpha^{-1}(0)\,,
\end{equation*}
it follows by the strong maximum principle
that $\varphi_{\alpha,n}>0$ in $B_n$.

By It\^o's formula we obtain from \eqref{ET2.8C} that
\begin{align}\label{ET2.8D}
\varphi_{\alpha,n}(x) &\;=\; \Exp_x\Bigl[\E^{\int_0^{T\wedge\uptau_n}
F_\alpha(X_s)\,\D{s}}\,\varphi_{\alpha,n}(X_{T\wedge\uptau_n})\Bigr]
+\Gamma_\alpha^{-1}(0)\Exp_x\biggl[\int_0^{T\wedge\uptau_n}
\E^{\int_0^{t}F_\alpha(X_s)\,\D{s}}\,\D{t}\biggr]\nonumber\\[5pt]
&\;=\; \Exp_x\Bigl[\E^{\int_0^{T}
F_\alpha(X_s)\,\D{s}}\,\varphi_{\alpha,n}(X_{T})\,\Ind_{\{T\le\uptau_n\}}\Bigr]
+\Gamma_\alpha^{-1}(0)\Exp_x\biggl[\int_0^{T\wedge\uptau_n}
\E^{\int_0^{t}F_\alpha(X_s)\,\D{s}}\,\D{t}\biggr]
\end{align}
for all $(T,x)\in\RR_+\times B_n$, where we use
the property that $\varphi_{\alpha,n}=0$ on $\partial B_n$.
Let $\Psi$ be any positive solution of \eqref{E2-Pois}, satisfying
$\Psi(0)=1$.
Writing \eqref{E2-Pois} as
\begin{equation*}
\Lg \Psi(x) + F_\alpha(x)\,\Psi(x)
\;=\; -\alpha\,\Psi(x)\qquad\text{a.e.\ }x\in\Rd\,,
\end{equation*}
and using It\^o's formula and Fatou's lemma we obtain
\begin{equation}\label{ET2.8F}
\Psi(x) \;\ge\; \Exp_x\Bigl[\E^{\int_0^{\uptau_n}
F_\alpha(X_s)\,\D{s}}\,\Psi(X_{\uptau_n})\Bigr]
+ \alpha\,\Exp_x\biggl[\int_0^{\uptau_n}
\E^{\int_0^{t}F_\alpha(X_s)\,\D{s}}\,\Psi(X_t)\,\D{t}\biggr]\,,
\end{equation}
for  any $\alpha\ge0$.
By It\^o's formula and Fatou's lemma, we have from \eqref{E2-Pois} that
\begin{equation*}
\Exp_{x}\,\Bigl[\E^{\int_{0}^{t}F_0(X_s)\,\D{s}}\,\Bigr]
\;\le\; \Bigl(\inf_{\Rd}\,\Psi\Bigr)^{-1}\,
\Psi(x)\qquad\forall\,t>0\,,\quad\forall\,x\in \Rd\,.
\end{equation*}
Therefore,
\begin{equation*}
\Exp_x\Bigl[\E^{\int_0^{T}
F_\alpha(X_s)\,\D{s}}\,\varphi_{\alpha,n}(X_{T})\,\Ind_{\{T\le\uptau_n\}}\Bigr]
\;\le\;
\E^{-\alpha T}\,\Bigl(\sup_{B_n}\,\varphi_{\alpha,n}\Bigr)
\Exp_x\Bigl[\E^{\int_0^{T}F_0(X_s)\,\D{s}}\,\Ind_{\{T\le\uptau_n\}}\Bigr]
\;\xrightarrow[T\to\infty]{}\;0\,.
\end{equation*}
Thus taking limits in \eqref{ET2.8D} as $T\to\infty$, using monotone convergence
for the second integral, we obtain
\begin{equation*}
\varphi_{\alpha,n}(x)\;=\; 
\Gamma_\alpha^{-1}(0) \Exp_x\biggl[\int_0^{\uptau_n}
\E^{\int_0^{t}F_\alpha(X_s)\,\D{s}}\,\D{t}\biggr]\,,
\end{equation*}
which implies by \eqref{ET2.8F}
that
\begin{equation*}
\varphi_{\alpha,n}\;\le\;
\frac{1}{\alpha\Gamma_\alpha(0)}\,
\Bigl(\inf_{\Rd}\,\Psi\Bigr)\Psi \qquad\forall\,n\in\NN\,.
\end{equation*}
It therefore follows by \eqref{E-apriori} that $\{\varphi_{\alpha,n}\}$
is relatively weakly compact in $\Sob^{2,p}(B_n)$, for any $p\ge1$ and $n\in\NN$,
and as explained in Section~\ref{S1.4},
$\varphi_{\alpha,n}$ converges uniformly on compact sets
along some sequence $n\to\infty$
to a nonnegative
$\Phi_\alpha\in\Sobl^{2,p}(\Rd)$, for any $p\ge1$, which solves
\begin{equation}\label{ET2.8G}
\Lg \Phi_\alpha(x) + F_\alpha(x)\,\Phi_\alpha(x)
\;=\; -\Gamma_\alpha^{-1}(0)\qquad\text{a.e.\ }x\in\Rd\,.
\end{equation}
It is clear by the strong maximum principle and since also
$\Gamma_\alpha^{-1}(0)>0$ that
$\Phi_\alpha>0$.
By \eqref{ET2.8D} and dominated and monotone convergence, we obtain
\begin{equation}\label{ET2.8I}
\Phi_{\alpha}(x)\;=\; \Exp_x\Bigl[\E^{\int_0^{T}
F_\alpha(X_s)\,\D{s}}\,\Phi_{\alpha}(X_{T})\Bigr]
+\Gamma_\alpha^{-1}(0)\Exp_x\biggl[\int_0^{T}
\E^{\int_0^{t}F_\alpha(X_s)\,\D{s}}\,\D{t}\biggr]
\end{equation}
for all $T>0$ and $x\in\Rd$, and
since \eqref{ET2.8D} holds with $T$ replaced by $\uuptau_r$, we also have
\begin{equation}\label{ET2.8J}
\Phi_\alpha(x)\;=\;
\Exp_x\Bigl[\E^{\int_{0}^{\uuptau_r}
F_\alpha(X_s)\,\D{s}}\,\Phi_\alpha(X_{\uuptau_r})\Bigr]
+\Gamma_\alpha^{-1}(0)
\Exp_x\biggl[\int_{0}^{\uuptau_r} \E^{\int_{0}^{t}
F_\alpha(X_s)\,\D{s}}\,\D{t}\biggr]
\end{equation}
for all $x\in B_r^c$ and $r>0$.
It also follows by letting $T\to\infty$ in \eqref{ET2.8I} that $\Phi_\alpha(0)=1$.

Recall the constant $\theta>0$ in the beginning of the proof,
and the set $\sK_\theta$ in Definition~\ref{D1.1}.
Let $\sB$ be a ball containing $\sK_\theta$.
With $\uuptau=\uptau(\sB^c)$, and integration by parts, we obtain
\begin{align}\label{ET2.8K}
\Exp_x\biggl[\int_{0}^{\uuptau} \E^{\int_{0}^{t}
F_0(X_s)\,\D{s}}\,\D{t}\biggr]  &\;\le\;
\frac{1}{\theta}\,\Exp_x\biggl[\int_{0}^{\uuptau}[f(X_t)-\Lambda(f)]\,
\E^{\int_{0}^{t}[f(X_s)-\Lambda(f)]\,\D{s}}\,\D{t}\biggr]\,\nonumber\\[5pt]
&\;\le\;
\frac{1}{\theta}\,
\Exp_x\biggl[\E^{\int_{0}^{\uuptau}[f(X_s)-\Lambda(f)]\,\D{s}}\biggr]
\nonumber\\[5pt]
&\;\le\;
\Bigr(\theta\,\inf_{\partial \sB}\;\Psi\Bigr)^{-1} \Psi(x)
\qquad\forall\,x\in \sB^c\,.
\end{align}
It follows by \eqref{ET2.8A} that $\Gamma_\alpha^{-1}(0)$ is bounded
uniformly over $\alpha\in(0,1)$.
Thus, the assumptions are met for the Harnack inequality for a class of
superharmonic functions in \cite{AA-Harnack}, and we obtain by
\eqref{ET2.8G} that
\begin{equation}\label{ET2.8L}
\Phi_\alpha(x) \;\le\; C_{H}\,,\qquad\forall\, x\in \Bar\sB\,,
\quad\forall\,\alpha\in(0,1)\,,
\end{equation}
for some constant $C_{H}$.
It then follows by \eqref{ET2.8J}, \eqref{ET2.8K}, and \eqref{ET2.8L}
that the collection $\{\Phi_\alpha\,,\;\alpha\in(0,1)\}$ is
locally bounded, and thus also
relatively
weakly compact in $\Sob^{2,p}(B_R)$ for any $p\ge1$ and $R>0$, by \eqref{ET2.8G}.
Taking limits along some sequence $\alpha\searrow0$, we obtain a positive
function $\Phi\in\Sobl^{2,d}(\Rd)$ which satisfies
\begin{equation*}
\Lg \Phi(x) + F_0(x)\,\Phi(x)
\;=\; 0\qquad\text{a.e.\ }x\in\Rd\,,
\end{equation*}
and  $\Phi(0)=1$.
By \eqref{ET2.8A}, \eqref{ET2.8J}, and \eqref{ET2.8K} we obtain
\begin{equation*}
\Phi(x)\;=\;\Exp_x\Bigl[
\E^{\int_0^{\uuptau}F_0(X_s)\,\D{s}}\,\Phi(X_{\uuptau})\Bigr]\,.
\end{equation*}
Also by Fatou's lemma we have
\begin{equation*}
\Exp_x\Bigl[\E^{\int_{0}^{\uuptau_r}F_0(X_s)\,\D{s}}\Psi(X_{\uuptau})\Bigr]
\;\le\; \Psi(x)
\qquad\forall\,x\in B_r^c\,,\quad\forall\,r>0\,.
\end{equation*}
This shows by the comparison principle (see the same argument
detailed below) that $\Psi=\Phi$
and hence \eqref{ET2.8B} holds.

We next prove uniqueness.
Let $\sB$ be a ball that contains $\sK$ in Definition~\ref{D1.1}
and let $\uuptau=\uptau(\sB^c)$.
If $\Tilde\Psi$ is any positive solution of \eqref{E2-Pois} with $\Tilde\Psi(0)=1$,
then
\begin{equation*}
\Tilde\Psi(x) \;\ge\; \Exp_x\Bigl[\E^{\int_{0}^{\uuptau}F_0(X_s)\,\D{s}}\,
\Tilde\Psi(X_{\uuptau})\Bigr] \qquad\forall\,x\in \sB^c\,,\quad\forall\,r>0\,,
\end{equation*}
by It\^o's formula and Fatou's lemma.
This together with \eqref{ET2.8B} implies that
\begin{align*}
\Tilde\Psi(x)-\Psi(x)&\;\ge\;
\Exp_{x}\Bigl[\E^{\int_{0}^{\uuptau}F_0(X_s)\,\D{s}}\,
\bigl(\Tilde\Psi(X_{\uuptau}) - \Psi(X_{\uuptau}) \bigr)\Bigr]
\nonumber\\[5pt]
&\;\ge\; \min_{\partial\sB}\,\bigr(\Tilde\Psi-\Psi\bigr)
\qquad\forall\, x\in \sB^c\,.
\end{align*}
It follows that $\Tilde\Psi-\Psi$ attains a global minimum on $\Bar\sB$.
Which means that we can scale $\Psi$ 
until it touches $\Tilde\Psi$ from below in at least one point in $\Bar\sB$.
Since
\begin{equation*}
\Lg(\Tilde\Psi-\Psi) - \bigl(f-\Lambda(f)\bigr)^-
\bigl(\Tilde\Psi-\Psi\bigr)\;=\;
-\bigl(f-\Lambda(f)\bigr)^+\bigl(\Tilde\Psi-\Psi\bigr)\;\le\; 0
\quad\text{on\ \ }\Rd\,,
\end{equation*}
it follows by the strong maximum principle
that $\Tilde\Psi=\Psi$ on $\Rd$.
This completes the proof.
\qed
\end{proof}

We next prove the converse statement.

%%%%%%%%%%%%%%%%%%%%%%%%%%%%%%%%%%%%%%%%%%%%%%%%%%%%%%%%%%%%%%%%%%%%%%%%%%%%%%%%
\begin{lemma}\label{L2.9}
Let $\Gamma$ be as in \eqref{ET2.8A}.
If process $X^*$ in \eqref{E-sde*} is recurrent, then
$\Gamma(x)=\infty$ for all $x\in\Rd$.
\end{lemma}

\begin{proof}
We argue by contradiction.
Suppose, without loss of generality, that $\Gamma(0)<\infty$.
Following the steps in the proof of Theorem~\ref{T2.8} we obtain
a positive function $\Phi\in\Sobl^{2,d}(\Rd)$ which satisfies
\begin{equation*}
\Lg \Phi(x) + F_0(x)\,\Phi(x)
\;=\; -\Gamma^{-1}(0)\qquad\text{a.e.\ }x\in\Rd\,,
\end{equation*}
and  $\Phi(0)=1$.

By \eqref{ET2.8J} and \eqref{ET2.8K} we have
\begin{equation*}
\Phi(x)\;\le\;
\Exp_x\Bigl[\E^{\int_{0}^{\uuptau}
F_0(X_s)\,\D{s}}\,\Phi(X_{\uuptau})\Bigr]
+\Gamma^{-1}(0)
\Bigr(\theta\,\inf_{\partial \sB}\;\Psi\Bigr)^{-1} \Psi(x)
\qquad\forall\,x\in \sB^c\,.
\end{equation*}
Combining this with \eqref{ET2.8L} and Lemma~\ref{L2.6} we
deduce that $\frac{\Phi}{\Psi}$ is bounded above in $\Rd$.
We also have
\begin{equation*}
\Lg^*\frac{\Phi}{\Psi}=-\Gamma^{-1}(0)\,\Psi^{-1}\quad\text{on\ }\Rd\,.
\end{equation*}
Therefore $\frac{\Phi}{\Psi}(X^*_t)$ is an
$\sF^{X^*}_t$-supermartingale, and thus converges a.s.
Since the $X^*$ process is recurrent the limit must be a constant,
which implies that $\Psi=\Phi$.
This of course is only possible if $\Gamma^{-1}(0)=0$, which is
a contradiction.
This completes the proof.
\qed
\end{proof}

%%%%%%%%%%%%%%%%%%%%%%%%%%%%%%%%%%%%%%%%%%%%%%%%%%%%%%%%%%%%%%%%%%%%%%%%%%%%%%%%
\begin{lemma}\label{L2.10}
Under {\upshape(H1)}, the following hold:
\begin{enumerate}
\item[{\upshape(}i\/{\upshape)}]
For any $r>0$
the Dirichlet eigensolutions
$(\widehat\Psi_{n},\Hat\lambda_{n})$ in \eqref{L2.2A} have the stochastic
representation
\begin{equation*}
\widehat\Psi_{n}(x)\;=\; \Exp_{x} \Bigl[\E^{\int_{0}^{\uuptau_{r}}
[f(X_{t})-\Hat\lambda_{n}]\,\D{t}}\,
\widehat\Psi_{n}(X_{\uuptau_{r}})\,
\Ind\{\uuptau_{r}<\uptau_{n}\}\Bigr]\qquad\forall\,
x\in B_{n}\setminus\Bar{B}_r\,,
\end{equation*}
for all large enough $n\in\NN$.
\item[{\upshape(}ii\/{\upshape)}]
Suppose that for some ball $\sB\subset\Rd$ and a constant $\delta>0$, we have
\begin{equation*}
\Exp_{x}\,\Bigl[\E^{\int_{0}^{\uptau(\sB^c)}
[f(X_{s})-\Lambda(f)+\delta]\,\D{s}}\Bigr]\;<\;\infty
\qquad\forall\, x\in \Bar{\sB}^c\,.
\end{equation*}
Then the solution $\Psi$ of \eqref{E2-Pois} satisfies
\begin{equation}\label{EL2.10A}
\Psi(x) \;=\; \Exp_x\Bigl[\E^{\int_{0}^{\uuptau_r}[f(X_s)-\Lambda(f)]\,\D{s}}\,
\Psi(X_{\uuptau_r})\Bigr]
\qquad\forall\,r>0\,,\quad\forall\, x\in B_r^c\,.
\end{equation}
\end{enumerate}
\end{lemma}

\begin{proof}
Let $\Tilde{f}\df f-\delta\Ind_{B_r}$, and $\Hat\lambda_n(\Tilde{f})$,
$\Hat\lambda_n(f)$ denote the Dirichlet eigenvalues
on $B_n$ corresponding to $\Tilde{f}$, $f$, respectively.
By the strict monotonicity of the Dirichlet eigenvalue in Lemma~\ref{L2.2}\,(a),
we have $\Hat\lambda_n(\Tilde{f})<\Hat\lambda(f)$ for all $n\in\NN$.
It then follows by the continuity of the Dirichlet eigenvalue with respect to the
radius of the ball, again in Lemma~\ref{L2.2}\,(a), that there exists a ball
$B_{r_n}\Supset B_n$ such that
 $\Hat\lambda_{r_n}(\Tilde{f})<\Hat\lambda_n(f)$.
Define $\delta_n\df\Hat\lambda_n(f)-\Hat\lambda_{r_n}(\Tilde{f})>0$.
Let $\bigl(\widehat\Psi_{r_n}^{\Tilde{f}},\Hat\lambda_{r_n}(\Tilde{f})\bigr)$,
and $\bigl(\widehat\Psi_{n},\Hat\lambda_{n}(f)\bigr)$,
denote the Dirichlet eigensolutions on $B_{r_n}$ for $\Tilde{f}$,
and on $B_n$ for $f$, respectively.
In the interest of simplifying the equations,
we drop the explicit dependence on $n$, and adopt the following
simplified notation:
\begin{equation*}
\begin{gathered}
\widehat\Psi\;=\;\widehat\Psi_{n}\,,\quad
\Tilde\Psi\;=\;\widehat\Psi_{r_n}^{\Tilde{f}}\,,\quad
\Hat\lambda\;=\;\Hat\lambda_{n}(f)\,,\quad
\Tilde\lambda\;=\;\Hat\lambda_{r_n}(\Tilde{f})\,,\quad
\delta\;\df\;\Hat\lambda-\Tilde\lambda\,>\,0\,,\\[3pt]
D\;=\;B_n\,,\quad
\Tilde{D} \;=\;B_{r_n}\,,\quad
\uptau\;=\;\uptau(B_n)\,.
\end{gathered}
\end{equation*}
Then
\begin{equation}\label{EL2.10B}
\begin{aligned}
\Lg \Tilde\Psi(x) + \Tilde{f}(x)\,\Tilde\Psi(x)
&\;=\; (\Hat\lambda-\delta)\,\Tilde\Psi(x)
\qquad\text{a.e.\ }x\in \Tilde{D}\,,&&\quad\text{and \ \ }
\Tilde\Psi(0)=1\,,~\Tilde\Psi\bigr|_{\partial\Tilde{D}}=0
\\[5pt]
\Lg \widehat\Psi(x) + f(x)\,\widehat\Psi(x)
&\;=\; \Hat\lambda\,\widehat\Psi(x)
\qquad\text{a.e.\ }x\in D\,,&&\quad\text{and \ \ }
\widehat\Psi(0)=1\,,~\widehat\Psi\bigr|_{\partial{D}}=0\,.
\end{aligned}
\end{equation}
Applying Dynkin's formula to the first equation in \eqref{EL2.10B}, for the
stopping time $T\wedge\uptau\wedge\uuptau_r$, we obtain
\begin{align*}
\Tilde\Psi(x) &\;=\;
\Exp_{x}\Bigl[\E^{\int_{0}^{T\wedge\uptau\wedge\uuptau_r}
[\Tilde{f}(X_{t})-\Hat\lambda+\delta]\,\D{t}}\,
\Tilde\Psi(X_{T\wedge\uptau\wedge\uuptau_r})\Bigr]
\nonumber\\[5pt]
&\;\ge\;
\Exp_{x}\Bigl[\E^{\int_{0}^{T}
[\Tilde{f}(X_{t})-\hat\lambda+\delta]\,\D{t}}\,\Tilde\Psi(X_{T})\,
\Ind\{T<\uptau\wedge\uuptau_r\}\Bigr]\,.
\end{align*}
Therefore we have
\begin{multline}\label{EL2.10D}
\Exp_{x}\Bigl[\E^{\int_{0}^{T}
[\Tilde{f}(X_{t})-\Hat\lambda]\,\D{t}}\,\widehat\Psi(X_{T})\,
\Ind\{T<\uptau\wedge\uuptau_r\}\Bigr]\\
\;\le\;
\E^{-\delta\,T}\,(\sup_{D} \widehat\Psi)\,\Bigr(\inf_{D}\,\Tilde\Psi\Bigr)^{-1}
\Tilde\Psi(x)\qquad\forall\,T>0\,,\quad\forall\, x\in D\,.
\end{multline}
Thus, applying Dynkin's formula to the second equation in \eqref{EL2.10B},
we obtain
\begin{align}\label{EL2.10E}
\widehat\Psi(x)&\;=\;
\lim_{T\to\infty}\;
\Exp_{x}\Bigl[\E^{\int_{0}^{T}[f(X_{t})-\Hat\lambda]\,\D{t}}\,
\widehat\Psi(X_{T})\,\Ind\{T<\uptau\wedge\uuptau_r\}\Bigr]
\nonumber\\[5pt]
&\mspace{100mu}
+ \lim_{T\to\infty}\;
\Exp_{x}\Bigl[\E^{\int_{0}^{\uptau\wedge\uuptau_r}[f(X_{t})-\Hat\lambda]\,\D{t}}\,
\widehat\Psi(X_{\uptau\wedge\uuptau_r})\,\Ind\{T\ge\uptau\wedge\uuptau_r\}\Bigr]
\nonumber\\[5pt]
&\;=\;
\Exp_{x}\Bigl[\E^{\int_{0}^{\uptau\wedge\uuptau_r}[f(X_{t})-\Hat\lambda]\,\D{t}}\,
\widehat\Psi(X_{\uptau\wedge\uuptau_r})\Bigr]
\nonumber\\[5pt]
&\;=\;
\Exp_{x}\Bigl[\E^{\int_{0}^{\uuptau_r}[f(X_{t})-\Hat\lambda]\,\D{t}}\,
\widehat\Psi(X_{\uuptau_r})\,\Ind\{\uuptau_r<\uptau\}\Bigr]\,,
\end{align}
where for the first limit we use \eqref{EL2.10D}, for the second
limit we use monotone convergence,
while for the last equality we use the property that $\widehat\Psi=0$
on $\partial{D}$.
This proves part (i).

Now we prove \eqref{EL2.10A}. Applying Fatou's lemma to \eqref{EL2.10E} we obtain,
with $\uuptau=\uptau(\sB^c)$ that
\begin{equation}\label{EL2.10F}
\Psi(x)\;\geq \;\Exp_{x}\Bigl[\E^{\int_{0}^{\uuptau}
[f(X_{t})-\Lambda(f)]\,\D{t}}\,\Psi(X_{\uuptau})\Bigr].
\end{equation}
We write \eqref{EL2.10E} as
\begin{align}\label{EL2.10G}
\widehat\Psi_{n}(x)&\;=\;
\Exp_{x}\Bigl[\E^{\int_{0}^{\uuptau}
[f(X_{t})-\Hat\lambda_{n}]\,\D{t}}\,
\Psi(X_{\uuptau})\,\Ind\{\uuptau<\uptau_{n}\}\Bigr]
\nonumber\\[5pt]
&\mspace{150mu}\;+\; \biggl(\sup_{B_r}\,\babs{\Psi-\widehat\Psi_n}\biggr)\;
\Exp_{x}\Bigl[\E^{\int_{0}^{\uuptau}
[f(X_{t})-\Hat\lambda_{n}]\,\D{t}}\,\Ind\{\uuptau<\uptau_{n}\}\Bigr]
\nonumber\\[5pt]
&\;\le\;
\Exp_{x}\Bigl[\E^{\int_{0}^{\uuptau}
[f(X_{t})-\Hat\lambda_{n}]\,\D{t}}\,\Psi(X_{\uuptau})\Bigr]
+ \biggl(\sup_{B_r}\,\babs{\Psi-\widehat\Psi_n}\biggr)\;
\Exp_{x}\Bigl[\E^{\int_{0}^{\uuptau}
[f(X_{t})-\Hat\lambda_{n}]\,\D{t}}\,\Ind\{\uuptau<\uptau_{n}\}\Bigr]\,.
\end{align}
Note that since $\Hat\lambda_{n}\nearrow\Lambda(f)$, the first
term on the right hand side of \eqref{EL2.10G} is finite for
all large enough $n$ by Lemma~\ref{L2.3}(iii).
Let
\begin{equation*}
\kappa_n\;\df\; \biggl(\inf_{B_r}\,\widehat\Psi_{n}\biggr)^{-1}
\biggl(\sup_{B_r}\,\babs{\Psi-\widehat\Psi_n}\biggr)\,.
\end{equation*}
The second term on the right hand side of \eqref{EL2.10G} has the bound
\begin{align*}
\biggl(\sup_{B_r}\,\babs{\Psi-\widehat\Psi_n}\biggr)\;
\Exp_{x}\Bigl[\E^{\int_{0}^{\uuptau}
[f(X_{t})-\Hat\lambda_{n}]\,\D{t}} \Ind\{\uuptau<\uptau_{n}\}\Bigr]&\;\le\;
\kappa_n\;
\Exp_{x}\Bigl[\E^{\int_{0}^{\uuptau}
[f(X_{t})-\Hat\lambda_{n}]\,\D{t}}\,
\widehat\Psi_{n}(X_{\uuptau})\,\Ind\{\uuptau<\uptau_{n}\}\Bigr]
\nonumber\\[5pt]
&\;=\; \kappa_n\;\widehat\Psi_{n}(x)\,.
\end{align*}
By the convergence of $\widehat\Psi_{n}\to\Psi$
as $n\to\infty$, uniformly on compact sets,
and since $\widehat\Psi_{n}$ is bounded away from $0$ in $B_r$,
uniformly in $n\in\NN$ by the Harnack inequality,
we have $\kappa_n\to0$ as $n\to\infty$.
Therefore, the second term on the right hand side of \eqref{EL2.10G}
vanishes as $n\to\infty$. 
Also, since $\Hat\lambda_{n}$ is nondecreasing in $n$, and
$\Hat\lambda_{n}\nearrow\Lambda(f)$, we have
\begin{equation}\label{EL2.10H}
\Exp_{x}\Bigl[\E^{\int_{0}^{\uuptau}
[f(X_{t})-\Hat\lambda_{n}]\,\D{t}}\,\Psi(X_{\uuptau})\Bigr]
\;\xrightarrow[n\to\infty]{}\;
\Exp_{x}\Bigl[\E^{\int_{0}^{\uuptau}
[f(X_{t})-\Lambda(f)]\,\D{t}}\,\Psi(X_{\uuptau})\Bigr]
\end{equation}
by monotone convergence.
Thus taking limits in \eqref{EL2.10G} as $n\to\infty$, and using 
\eqref{EL2.10F} and \eqref{EL2.10H}
we obtain \eqref{ET2.8B}.
This implies that \eqref{EL2.10A} holds for any $r>0$ by Lemma~\ref{L2.6}.
This completes the proof.
\qed\end{proof}

%%%%%%%%%%%%%%%%%%%%%%%%%%%%%%%%%%%%%%%%%%%%%%%%%%%%%%%%%%%%%%%%%%%%%%%%%%%%%%%%
\begin{lemma}\label{L2.11}
If the process $X^*$ in \eqref{E-sde*} is not recurrent,
then for any bounded measurable set $A\in\Rd$ of positive
Lebesgue measure there exists $\epsilon>0$ such that
$\Lambda(f+\epsilon\Ind_{A})= \Lambda(f)$.
On the other hand, if $X^*$ is recurrent then
$\Lambda(f)<\Lambda(f+\epsilon\Ind_{A})$ for all
measurable set $A\in\Rd$ of positive
Lebesgue measure and $\epsilon>0$.
\end{lemma}

\begin{proof}
Let $A\in\Rd$ be a bounded measurable set, and $\sB$ and open ball containing $A$.
We set $\uuptau=\uptau(\sB^c)$.
Recall the eigenvalues $\{\Hat\lambda_n\}$ in Lemma~\ref{L2.2}.
Since $\Hat\lambda_n<\Lambda(f)$, the principal eigenvalue of the operator
$-\Lg-(f-\Lambda(f))$ on every $B_n$ is positive.
Thus by Proposition~6.2 and Theorem~6.1 in \cite{Berestycki-94},
for any $n\in\NN$, and $\alpha_n\ge0$, the Dirichlet problem
\begin{equation}\label{EL2.11A}
\Lg\varphi_n + (f-\Lambda(f))\varphi_n \;=\; - \alpha_n \Ind_A\quad\text{in~}B_n\,,
\end{equation}
with $\varphi_n=0$ on $\partial B_n$,
has a unique solution $\varphi_{n}\in\Sobl^{2,p}(B_n)\cap\Cc(\Bar{B}_n)$,
for any $p\ge1$.
In addition, provided $\alpha_n\ge0$, then by the \emph{refined maximum principle} in
\cite[Theorem~1.1]{Berestycki-94} $\varphi_{n}$ is nonnegative.
Let $\Tilde\alpha_n>0$ be such that the Dirichlet problem
$\Lg\Tilde\varphi_n + (f-\Lambda(f))\Tilde\varphi_n 
= - \Tilde\alpha_n \Ind_A$ in $B_n$,
with $\Tilde\varphi_n=0$ on $\partial B_n$, satisfies $\Tilde\varphi_n(0)=1$.
We set $\alpha_n= \min\,(1,\Tilde\alpha_n)$ in \eqref{EL2.11A}.
Thus, the assumptions are met for the Harnack inequality for a class of
superharmonic functions in \cite{AA-Harnack}, and we have
\begin{equation*}
\varphi_n(x) \;\le\; C_{H}\qquad\forall\, x\in B_n\,,
\quad\forall\,n\in\NN\,,
\end{equation*}
for some constant $C_{H}$ depending on $n$.
Thus the collection $\{\varphi_n\,,\;n\in\NN\}$ is relatively
weakly compact in $\Sob^{2,p}(B_R)$ for any $p\ge1$ and $R>0$,
and taking limits along some sequence $n\to\infty$, we obtain a positive function
$\Phi\in\Sobl^{2,d}(\Rd)$,
which solves
\begin{equation}\label{EL2.11A0}
\Lg\Phi + (f-\Lambda(f))\Phi \;=\; - \alpha \Ind_A
\end{equation}
on $\Rd$, for some nonnegative constant $\alpha$.
If $\alpha>0$ then $\Phi$ is positive on $\Rd$ by the strong maximum principle.
On the other hand, if $\alpha_n\searrow0$, then the definition of $\alpha$,
$\varphi_n(0)=1$ along this sequence, except for
a finite number of terms, since $\varphi_n(0)=1$ whenever $\alpha_n<1$.
Thus, in either case, the solution $\Phi$ is positive.

By It\^o's formula applied to \eqref{EL2.11A}, with $\uuptau=\uptau(\sB^c)$,
we obtain
\begin{multline}\label{EL2.11B}
\varphi_n(x) \;=\; \Exp_x\Bigl[\E^{\int_{0}^{\uuptau}[f(X_s)-\Lambda(f)]\,\D{s}}\,
\varphi_n(X_{\uuptau})\,\Ind_{\{\uuptau<T\wedge\uptau_n\}}\Bigr]\\[5pt]
+\Exp_x\Bigl[\E^{\int_{0}^{T}[f(X_s)-\Lambda(f)]\,\D{s}}\,
\varphi_n(X_{T})\,
\Ind_{\{T<\uuptau\wedge\uptau_n\}}\Bigr]
\qquad\forall\,x\in B_n\setminus\sB\,,\ \forall\,T>0\,.
\end{multline}
With 
$(\widehat\Psi_{n},\Hat\lambda_{n})$
denoting the Dirichlet eigensolutions in \eqref{L2.2A},
and It\^o's formula, we obtain
\begin{equation*}
\Exp_x\Bigl[\E^{\int_{0}^{T}[f(X_s)-\Hat\lambda_{n+1}]\,\D{s}}\,
\Ind_{\{T<\uuptau\wedge\uptau_n\}}\Bigr]
\;\le\; \Bigl(\inf_{B_n\setminus\sB}\,\widehat\Psi_{n+1}\Bigr)^{-1}\,
\widehat\Psi_{n+1}(x)\qquad\forall\,T>0\,,\ 
\forall\,x\in B_n\setminus\sB\,.
\end{equation*}
Thus, we have
\begin{align*}
\Exp_x &\Bigl[\E^{\int_{0}^{T}[f(X_s)-\Lambda(f)]\,\D{s}}\,
\varphi_n(X_{T})\,\Ind_{\{T<\uuptau\wedge\uptau_n\}}\Bigr]\\[5pt]
&\;\le\; 
\norm{\varphi_n}_\infty\;\E^{(\Hat\lambda_{n+1}-\Lambda(f))T}\,
\Exp_x\Bigl[\E^{\int_{0}^{T}[f(X_s)-\Hat\lambda_{n+1}]\,\D{s}}\,
\Ind_{\{T<\uuptau\wedge\uptau_n\}}\Bigr]\\[5pt]
&\;\le\;
\norm{\varphi_n}_\infty\;\E^{(\Hat\lambda_{n+1}-\Lambda(f))T}\,
\Bigl(\inf_{B_n\setminus\sB}\,\widehat\Psi_{n+1}\Bigr)^{-1}\,
\widehat\Psi_{n+1}(x)\quad\forall\,T>0\,,
\end{align*}
and all $x\in B_n$,
and it follows that the second term on the right hand side of \eqref{EL2.11B}
vanishes as $T\to\infty$.
Therefore, taking limits as $T\to\infty$ in \eqref{EL2.11B}, we have
\begin{equation}\label{EL2.11C}
\varphi_n(x) \;=\; \Exp_x\Bigl[\E^{\int_{0}^{\uuptau}[f(X_s)-\Lambda(f)]\,\D{s}}\,
\varphi_n(X_{\uuptau})\,\Ind_{\{\uuptau<\uptau_n\}}\Bigr]
\qquad\forall\,x\in B_n\setminus\sB\,,
\end{equation}
and taking again limits as $n\to\infty$ in \eqref{EL2.11C}, we obtain
\begin{equation}\label{EL2.11D}
\Phi(x) \;=\; \Exp_x\Bigl[\E^{\int_{0}^{\uuptau}[f(X_s)-\Lambda(f)]\,\D{s}}\,
\Phi(X_{\uuptau})\,\Ind_{\{\uuptau<\infty\}}\Bigr]
\qquad\forall\,x\in \sB^c\,.
\end{equation}

Suppose that $X^*$ in \eqref{E-sde*} is not recurrent.
Then we claim that $\alpha>0$.
Indeed, if $\alpha=0$, then $\Phi$ is a ground state corresponding to
$f$, and Lemma~\ref{L2.6} together with \eqref{EL2.11D} imply that
$X^*$ is recurrent, which contradicts the hypothesis.
Therefore, writing \eqref{EL2.11A0} as
\begin{equation*}
\Lg\Phi + \bigl(f+ \alpha\,\Phi^{-1}\Ind_A-\Lambda(f)\bigr)\Phi \;=\; 0\,,
\end{equation*}
and letting $\epsilon\df \inf_{\sB}\;\alpha\Phi^{-1}$, it follows
that $\Lambda(f+\epsilon\Ind_{A})= \Lambda(f)$.

Next assume $X^*$ in \eqref{E-sde*} is recurrent, and suppose that
$\Lambda(f+\epsilon\Ind_{A})= \Lambda(f)$
for some $\epsilon>0$.
We repeat the construction in the first part of the proof, starting
from the Dirichlet problem
\begin{equation*}
\Lg\varphi_n + \bigl(f+\epsilon\Ind_{A}-\Lambda(f)\bigr)\varphi_n
\;=\; - \alpha_n \Ind_A\quad\text{in~}B_n\,,
\end{equation*}
with $\varphi_n=0$ on $\partial B_n$, to obtain
positive function
$\Phi\in\Sobl^{2,d}(\Rd)$,
which solves
\begin{equation*}
\Lg\Phi + \bigl(f+ (\epsilon+\alpha\Phi^{-1})\Ind_A-\Lambda(f)\bigr)\Phi \;=\; 0
\end{equation*}
for some $\alpha\ge0$,
and satisfies \eqref{EL2.11D}.
Since $X^*$ is recurrent, the ground state $\Psi$ corresponding to
$f$ satisfies \eqref{EL2.11D} by Lemma~\ref{L2.6}.
Thus
\begin{equation*}
\Phi(x)-\Psi(x) \;\ge\; \min_{\sB}\,\bigr(\Phi-\Psi\bigr)
\qquad\forall\, x\in\sB^c\,,
\end{equation*}
and by scaling $\Phi$ so that it touches $\Psi$ at one point from above,
and using the strong maximum principle we deduce that $\Psi=\Phi$.
This of course is impossible unless $\epsilon=0$ and $\alpha=0$,
hence we reach a contradiction.
The proof is complete.
\qed
\end{proof}

%%%%%%%%%%%%%%%%%%%%%%%%%%%%%%%%%%%%%%%%%%%%%%%%%%%%%%%%%%%%%%%%%%%%%%%%%%%%%%%%
\subsection{Results concerning Theorems~\ref{T1.6}--\ref{T1.8}}\label{S2.3}
%%%%%%%%%%%%%%%%%%%%%%%%%%%%%%%%%%%%%%%%%%%%%%%%%%%%%%%%%%%%%%%%%%%%%%%%%%%%%%%%

We start with the following theorem.

%%%%%%%%%%%%%%%%%%%%%%%%%%%%%%%%%%%%%%%%%%%%%%%%%%%%%%%%%%%%%%%%%%%%%%%%%%%%%%%%
\begin{theorem}\label{T2.12}
Under {\upshape(H1)} the following are equivalent.
\begin{itemize}
\item[{\upshape(}a{\upshape)}]
For some ball $\sB\subset\Rd$ and a constant $\delta>0$, we have
\begin{equation}\label{ET2.12A}
\Exp_{x}\,\Bigl[\E^{\int_{0}^{\uptau(\sB^c)}
[f(X_{s})-\Lambda(f)+\delta]\,\D{s}}\Bigr]\;<\;\infty
\qquad\forall\, x\in \Bar{\sB}^c\,.
\end{equation}
\item[{\upshape(}b{\upshape)}]
If $h\in\Cc_c(\Rd)$ is non-negative function, $h\not\equiv0$,
then $\Lambda(f)>\Lambda(f-h)$.
\smallskip
\item[{\upshape(}c{\upshape)}]
For every $r>0$, there exists $\delta_r>0$, such that
\begin{equation}\label{ET2.12B}
\Exp_{x}\,\Bigl[\E^{\int_{0}^{\uuptau_r}
[f(X_{s})-\Lambda(f)+\delta_r]\,\D{s}}\Bigr]\;<\;\infty
\qquad\forall\, x\in \Bar{B}_r^c\,.
\end{equation}
\end{itemize}
\end{theorem}

\begin{proof}
The proof (a)$\;\Rightarrow\;$(b) is by contradiction.
Without loss of generality  suppose that $\sB$ contains the support of $h$,
and define $\Tilde{f}=f-h$.
It is clear that $\Lambda(f)\ge\Lambda(\Tilde{f})$,
so suppose that $\Lambda(f)=\Lambda(\Tilde{f})$.
Let $\Tilde\Psi$ be a positive solution to
\begin{equation*}
\Lg \Tilde\Psi(x) + \Tilde{f}(x)\Tilde\Psi(x)
\;=\; \Lambda(f) \Tilde\Psi(x)\qquad\text{a.e.\ }x\in\Rd\,.
\end{equation*}
If $\Psi$ is a positive solution to
\eqref{E2-Pois}, then we have
\begin{equation*}
\Lg \Psi(x) + \bigl(\Tilde{f}(x)-\Lambda(f)\bigr)\Psi(x)
\;=\; -h(x) \Psi(x)\qquad\text{a.e.\ }x\in\Rd\,.
\end{equation*}
By (a) and Lemma~\ref{L2.10}, we have
\begin{equation*}
\Tilde\Psi(x) \;=\; \Exp_x\Bigl[\E^{\int_{0}^{\uptau(\sB^c)}
[f(X_s)-\Lambda(f)]\,\D{s}}\,
\Tilde\Psi(X_{\uptau(\sB^c)})\Bigr]
\qquad\forall\, x\in \Bar{\sB}^c\,.
\end{equation*}
Therefore, as in the proof of uniqueness in Theorem~\ref{T2.8}, we have
\begin{equation*}
\Psi(x)-\Tilde\Psi(x) \;\ge\; \min_{\sB}\,\bigr(\Psi-\Tilde\Psi\bigr)
\qquad\forall\, x\in\sB^c\,,
\end{equation*}
which means that we can scale $\Psi$ 
until it touches $\Tilde\Psi$ from above in at least one point in $\sB$.
Since
\begin{equation*}
\Lg(\Psi-\Tilde\Psi) - \bigl(f-\Lambda(f)\bigr)^-
\bigl(\Psi-\Tilde\Psi\bigr)\;=\;
-\bigl(f-\Lambda(f)\bigr)^+\bigl(\Psi-\Tilde\Psi\bigr) - h \Tilde\Psi \;\le\; 0
\quad\text{on\ \ }\Rd\,,
\end{equation*}
it follows by the strong maximum principle that $\Psi=\Tilde\Psi$ on $\Rd$.
However, $\Psi$ solves $\Lg \Psi \;=\; (\Lambda(f)- f)\Psi$
on $\Rd$, and this implies that $h\Tilde\Psi=0$ a.e.
This, of course, is not possible since $\Tilde\Psi$ is positive.
Thus we reached a contradiction, and the proof of part (a) is complete.

Let $r>0$, and $\Tilde\Psi$ be a positive solution of \eqref{E2-Pois}
corresponding to $\Tilde{f}=f-\Ind_{B_r}$.
By part (b), we have $\delta_r\df \Lambda(f) - \Lambda(\Tilde{f})>0$.
Applying It\^o's formula and Fatou's lemma to $\Tilde\Psi$,
we obtain \eqref{ET2.12B}.
This shows (b)$\;\Rightarrow\;$(c), and since
it is clear that (c)$\;\Rightarrow\;$(a)
the proof is complete.
\qed
\end{proof}

Recall that $\Exp_x^*$ denotes the expectation operator
associated with $X^*$ in \eqref{E-sde*},
while $\Exp_x$ and $\mu$ denote the expectation operator
and invariant probability measure
associated with the solution of \eqref{E1-sde2}, respectively.

%%%%%%%%%%%%%%%%%%%%%%%%%%%%%%%%%%%%%%%%%%%%%%%%%%%%%%%%%%%%%%%%%%%%%%%%%%%%%%%%
\begin{theorem}\label{T2.13}
Assume {\upshape(H1)},
and suppose that the sde in \eqref{E-sde*} has a unique strong solution $X^*$ which
exists for all $t>0$.
The following hold for any function $g\in\Cc_{b}(\Rd)$, and all
$(t,x)\in\RR_+\times\Rd$:
\begin{align}
\Psi(x)\,\Exp^*_x\bigl[g(X^*_t)\,\Psi^{-1}(X^*_{t})\bigr] &\;=\;
\Exp_x\Bigl[\E^{\int_{0}^{t}[f(X_s)-\Lambda(f)]\,\D{s}}\,g(X_{t})\Bigr]
\,,\label{ET2.13A}\\[5pt]
\Psi(x)\,\Exp^*_x\bigl[g(X^*_t)\bigr]
&\;=\;\Exp_x\Bigl[\E^{\int_{0}^{t}[f(X_s)-\Lambda(f)]\,\D{s}}\,
\Psi(X_{t})\,g(X_{t})\Bigr]\label{ET2.13B}\,,
\end{align}
where $\Psi$ is a positive solution of \eqref{E2-Pois}.
In particular, we have
\begin{equation}\label{ET2.13C}
\Psi(x) \;=\;\Exp_{x}\,\Bigl[\E^{\int_{0}^{t}
[f(X_s)-\Lambda(f)]\,\D{s}}\,\Psi(X_{t})\Bigr]\qquad\forall\,(t,x)\in\RR_+\times\Rd\,.
\end{equation}
Suppose that $X^*$ is positive recurrent and let $\mu^*$
denote its invariant probability measure.
Then
\begin{equation}\label{ET2.13D}
\Exp_x\bigl[\E^{\int_{0}^{t}[f(X_s)-\Lambda(f)]\,\D{s}}\,\Psi(X_{t})\,g(X_{t})\bigr]
\;\xrightarrow[t\to\infty]{}\; \Psi(x)\,\int_{\Rd}
g(y)\,\mu^*(\D{y})
\end{equation}
for all $g\in\Cc_b(\Rd)$.
In particular, $\Psi_*$ defined by
\begin{equation*}
\Psi_*(x) \;\df\; \lim_{T\to\infty}\;\Exp_{x}\,\Bigl[\E^{\int_{0}^{T}
[f(X_{t})-\Lambda(f)]\,\D{t}}\Bigr]\,,
\end{equation*}
is a positive solution of the MPE and satisfies 
$\mu^*(\Psi_*^{-1})=1$.
\end{theorem}

\begin{proof}
Recall \eqref{EL2.5A}, which we repeat here as
\begin{align}\label{ET2.13F}
\Exp_x\left[\E^{\int_0^{t\wedge \uptau_n}
[f(X_s)-\Lambda(f)]\,\D{s}}\,g(X_{t\wedge\uptau_n})\right]
&\;=\; \Psi(x) \Exp_x\bigl[g(X_{t\wedge\uptau_n}) \,
\Psi^{-1}(X_{t\wedge\uptau_n}) \,\sM_{t\wedge\uptau_n}\bigr]\nonumber\\[5pt]
&\;=\; \Psi(x) \Exp_x^*\bigl[g(X^*_{t\wedge\uptau_n}) \,
\Psi^{-1}(X^*_{t\wedge\uptau_n})\bigr]\,.
\end{align}
Let $g\in\Cc_c(\Rd)$, and note that since $\Psi$ is inf-compact,
the term inside the expectation in the right hand side of \eqref{ET2.13F} is bounded
uniformly in $n$.
Thus letting $n\to\infty$ in \eqref{ET2.13F}, we obtain
\begin{align}\label{ET2.13H}
\Exp_x\Bigl[\E^{\int_0^{t} [f(X_s)-\Lambda(f)]\,\D{s}}\,g(X_{t})\Bigr]
&\;=\; \Psi(x) \Exp_x\bigl[g(X_{t}) \,
\Psi^{-1}(X_{t}) \,\sM_{t}\bigr]\nonumber \\[5pt]
&\;=\; \Psi(x)\Exp_x^*\bigl[g(X^*_{t})
\Psi^{-1}(X^*_{t})\bigr]\qquad\forall\,t>0\,.
\end{align}
This proves \eqref{ET2.13A} for $g\in\Cc_c(\Rd)$, and also
for $g\in\Cc_b(\Rd)$ by monotone convergence over an increasing sequence
$g_n\in\Cc_c(\Rd)$.

Applying again monotone convergence
to \eqref{ET2.13H} we obtain \eqref{ET2.13B}.
By Lemma~\ref{L2.5}, $\bigl(\sM_t,\sF^X_t\bigr)$ is a martingale,
and thus \eqref{ET2.13C} follows by Lemma~\ref{L2.4}.

Since $X^*$ is ergodic, we obtain from \eqref{ET2.13B} that
\begin{align*}
\lim_{T\to\infty}\;\frac{1}{T}\,\int_0^T\,\Exp_x
\Bigl[\E^{\int_0^{T} (f(X_s)-\Lambda(f))\,\D{s}}\,\Psi(X_{T})\,g(X_{T})\Bigr]
&\;=\; \Psi(x) \;\lim_{T\to\infty}\;
\frac{1}{T}\,\int_0^T \Exp^*_x\bigl[g(X^*_t)\bigr]\\[5pt]
&\;=\; \Psi(x)\,\int_{\Rd} g(y)\,\mu^*(\D{y})
\end{align*}
by Birkhoff's ergodic theorem.
Then \eqref{ET2.13D} follows
by an application of \cite[Theorem~4.12]{Ichihara-12}.
This completes the proof.
\qed
\end{proof}

%%%%%%%%%%%%%%%%%%%%%%%%%%%%%%%%%%%%%%%%%%%%%%%%%%%%%%%%%%%%%%%%%%%%%%%%%%%%%%%%
\begin{corollary}\label{C2.14}
Assume {\upshape(H1)}.
Then there exist an open ball $\sB\subset\Rd$ and a constant $\delta>0$
such that \eqref{ET2.12A} holds, if and only if the ground state diffusion
$X^*$ is geometrically ergodic.
\end{corollary}

\begin{proof}
We first show necessity.
By Theorem~\ref{T2.12},
for any $r>0$, we have $\delta\df \Lambda(f)-\Lambda(f-\gamma\Ind_{B_r})>0$.
Let $\Tilde\Psi$ be the positive solution
of the MPE
\begin{equation*}
\Lg\Tilde\Psi \;=\; \bigl(\Lambda(f-\gamma\Ind_{B_r})-f
+\gamma\Ind_{B_r})\,\widetilde\Psi\,.
\end{equation*}
With $\Tilde\Lyap=\Tilde\Psi\Psi^{-1}$, we have
\begin{align}\label{EC2.14A}
\Lg^* \Tilde\Lyap &\;=\; \bigl(\Lambda(f-\gamma\Ind_{B_r})-f
+\gamma\Ind_{B_r}-\Lambda(f)+f)\,\Tilde\Lyap\nonumber\\[3pt]
&\;=\; \bigl(\gamma\Ind_{B_r}-\delta)\,\Tilde\Lyap\,.
\end{align}
It follows by Lemma~\ref{L2.10}\,(ii) that
$\inf_{\Rd}\,\Tilde\Lyap>0$.
Thus, the Foster--Lyapunov equation in \eqref{EC2.14A} implies
that \eqref{E-sde*} is geometrically ergodic.

Next suppose that $X^*$ is geometrically ergodic.
This implies that for some open ball $\sB\subset\Rd$ and a constant $\delta>0$
we have
$\Exp^*_x\bigl[\E^{\delta\uuptau}\bigr]<\infty$ for all $x\in\sB^c$,
where $\uuptau=\uptau(\sB^c)$.
We may select $\sB$ so that $\sB\supset\sK$ in Definition~\ref{D1.1}.
Similarly to \eqref{ET2.13F} we obtain
\begin{equation}\label{EC2.14B}
\Exp_x\left[\E^{\int_0^{\uuptau\wedge \uptau_n}
[f(X_s)-\Lambda(f)+\delta]\,\D{s}}\,\Psi(X_{\uuptau\wedge\uptau_n})\right]
\;=\; \Psi(x) \Exp_x^*\bigl[\E^{\delta(\uuptau\wedge \uptau_n)}\bigr]\,.
\end{equation}
Thus \eqref{ET2.12A} follows by letting $n\to\infty$ in
\eqref{EC2.14B}, and using the fact that $\inf_{\Rd}\,\Psi>0$.
This completes the proof.
\qed
\end{proof}

%%%%%%%%%%%%%%%%%%%%%%%%%%%%%%%%%%%%%%%%%%%%%%%%%%%%%%%%%%%%%%%%%%%%%%%%%%%%%%%%
\begin{lemma}\label{L2.15}
Under {\upshape(H1)--(H2)} the ground state diffusion
in \eqref{E-sde*} is geometrically ergodic.
\end{lemma}

\begin{proof}
By (H2) there exist a nonnegative constant $\kappa_0$ and an
inf-compact function $g_0\colon\Rd\to\RR_+$ such that
\begin{equation*}
\tfrac{1}{2}\,a^{ij}\partial_{ij} \psi_0 + \langle b,\grad \psi_0\rangle + 
\tfrac{1}{2}\,\langle\grad\psi_0,a\grad \psi_0\rangle +f
\;=\;\kappa_0 - g_0\,.
\end{equation*}
Thus, with $\Psi_0=\E^{\psi_0}$ we obtain
\begin{equation*}
\Lg \Psi_0 \;=\; (\kappa_0 - g_0 - f) \Psi_0\,.
\end{equation*}
The argument used in the proof of Lemma~\ref{L2.1} shows that
$\Psi_0$ is inf-compact, and therefore $\inf_{\Rd}\,\Psi_0>0$.
Let $\sB$ be a ball such that $g_0 \ge \kappa_0-\Lambda(f)+\delta$ on
$\sB^c$, for some $\delta>0$.
By It\^o's formula and Fatou's lemma we obtain, with $\uuptau=\uptau(\sB^c)$
that
\begin{equation*}
\Psi_0(x) \;\ge\; \Bigl(\inf_{\Rd}\;\Psi_0\Bigr)\,
\Exp_{x}\,\Bigl[\E^{\int_{0}^{\uptau}
[f(X_{s})-\Lambda(f)+\delta]\,\D{s}}\Bigr]
\qquad\forall\, x\in \Bar{\sB}^c\,.
\end{equation*}
The result then follows by Corollary~\ref{C2.14}.
\qed
\end{proof}

We conclude this section with the proof of Theorem~\ref{T1.8}.

%%%%%%%%%%%%%%%%%%%%%%%%%%%%%%%%%%%%%%%%%%%%%%%%%%%%%%%%%%%%%%%%%%%%%%%%%%%%%%%%
\begin{pT1.8}
We claim that there exist a bounded open ball $\sB_\circ$
and a positive constant $\delta_\circ$ such that,
with $\uuptau_\circ=\uptau(\sB_\circ^c)$
we have
\begin{equation*}
\Exp_{x}\Bigl[\E^{\int_{0}^{\uuptau_\circ}
[f(X_{t})-\Lambda(f)+\delta_\circ]\,\D{t}}\Bigr]\;<\;\infty
\qquad\forall\,x\in \sB_\circ^c\,.
\end{equation*}

Under the hypothesis for some $\varepsilon>0$,
$\varepsilon f$ is near monotone relative to
$\Lambda\bigl((1+\varepsilon)f\bigr)-\Lambda(f)$ and so $\varepsilon(f-\Lambda(f))$ is
near-monotone relative to
$\Lambda\bigl((1+\varepsilon)f\bigr)-(1+\varepsilon)\Lambda(f)$.
Thus, we can select some $\delta_\circ$ such
that $f-\Lambda(f)$ is near monotone relative to
$\theta(\varepsilon)$ which is defined as
\begin{equation*}
\theta(\varepsilon)\;\df\;\frac{1}{\varepsilon}\Bigl(
(1+\varepsilon)\delta_\circ +
\Lambda\bigl((1+\varepsilon)f\bigr)-(1+\varepsilon)\Lambda(f)\bigr)\,.
\end{equation*}
Let $\sB_\circ$ be a bounded
ball centered at the origin
such that $f-\Lambda\ge\theta(\varepsilon)$ on $\sB_\circ^c$.
Let $F(t,x) \df \E^{\theta(\varepsilon)t}\Psi$.
Then $\nicefrac{\partial}{\partial t} F(t,x)+ \Lg F(t,x)\le0$
for all $x\in\sB_\circ^c$, and
a straightforward application of the It\^o formula shows that
\begin{equation}\label{PT1.8A}
\Exp_x \bigl[\E^{\theta(\varepsilon)\uuptau_\circ}\bigr]\;\le\;
\Psi(x)\,\Bigl(\inf_{\partial \sB_\circ}\;\Psi\Bigr)^{-1}
\qquad\forall\,x\in \sB_\circ^c\,.
\end{equation}
We write
\begin{equation*}
f-\Lambda(f)+\delta_\circ \;=\; \frac{\varepsilon}{1+\varepsilon}\theta(\varepsilon)
+ \biggl(f - \frac{\Lambda\bigl((1+\varepsilon)f\bigr)}{1+\varepsilon}\biggr)\,,
\end{equation*}
and apply H\"older's inequality to obtain
\begin{equation}\label{PT1.8B}
\Exp_{x}\Bigl[\E^{\int_{0}^{\uuptau_\circ}
[f(X_{t})-\Lambda(f)+\delta_\circ]\,\D{t}}\Bigr]\;\le\;
\Bigl(\Exp_x \bigl[\E^{\theta(\varepsilon)\uuptau_\circ}\bigr]
\Bigr)^{\frac{\varepsilon}{1+\varepsilon}}
\,\biggr(\Exp_{x}\Bigl[\E^{\int_{0}^{\uuptau_\circ}
[(1+\varepsilon)f(X_{t})-\Lambda((1+\varepsilon)f)]\,\D{t}}\Bigr]
\biggr)^{\frac{1}{1+\varepsilon}}\,.
\end{equation}
Let $\Psi_\epsilon$ denote a positive solution of \eqref{E2-Pois}
for $(1+\varepsilon)f$ and $\Lambda((1+\varepsilon)f)$.
Then by Fatou's lemma we have
\begin{equation}\label{PT1.8C}
\Exp_{x}\Bigl[\E^{\int_{0}^{\uuptau_\circ}
[(1+\varepsilon)f(X_{t})-\Lambda((1+\varepsilon)f)]\,\D{t}}\Bigr]
\;\le\; \Psi_\varepsilon(x)
\Bigl(\inf_{\partial \sB_\circ}\;\Psi_\varepsilon\Bigr)^{-1}
\qquad\forall\,x\in \sB_\circ^c\,.
\end{equation}
The claim then follows by \eqref{PT1.8A}, \eqref{PT1.8B}, and \eqref{PT1.8C}.

Let $\varepsilon>0$ be as in (H3) and $\Psi_\varepsilon$
denote the solution of the MPE for $(1+\varepsilon)f$.
It is straightforward to show using H\"older's inequality
together with the stochastic representation equation in \eqref{ET2.13C}
that $F\df\Psi_\varepsilon \Psi^{-1}$ is inf-compact.
Indeed if $\varepsilon f$ is near monotone relative to
$\Lambda\bigl((1+\varepsilon) f\bigr)-\Lambda(f)$,
this implies that for some ball $\sB$ and constants $\delta>0$
and $\gamma>0$ we have
\begin{equation}\label{PT1.8D}
\begin{split}
 f-\Lambda(f) &\;\ge\;\delta\,,\\[5pt]
(1+\varepsilon)f -\Lambda\bigl((1+\varepsilon) f\bigr)&\;\ge\;
(1+\gamma) \bigl(f-\Lambda(f)\bigr)
\end{split}
\end{equation}
on $\sB^c$,
where we also use the property that $f$ is near monotone relative
to $\Lambda(f)$.
By \eqref{PT1.8D} and Jensen's inequality, and letting
$\uuptau=\uuptau(\sB^{c})$, we obtain
\begin{equation}\label{PT1.8E}
\biggl(\Exp_{x}\Bigl[\E^{\int_{0}^{\uuptau}
[f(X_{t})-\Lambda(f)]\,\D{t}}\Bigr]\biggr)^{1+\gamma}\;\le\;
\Exp_{x}\Bigl[\E^{\int_{0}^{\uuptau}
[(1+\varepsilon)f(X_{t})-\Lambda((1+\varepsilon)f)]\,\D{t}}\Bigr]\qquad
\forall\,x\in\sB^c\,.
\end{equation}
Therefore, by \eqref{ET2.13C} and \eqref{PT1.8E} we have
\begin{equation}\label{PT1.8F}
\frac{\Psi_\varepsilon(x)}{\Psi(x)}\;\ge\;
M_\sB\,\bigl(\Psi(x)\bigr)^\gamma \qquad\forall\,x\in\sB^c\,,
\end{equation}
with
\begin{equation*}
M_\sB\;\df\;\biggl(\inf_{\partial\sB}\,\Psi_\varepsilon\biggr)
\biggl(\sup_{\partial\sB}\,\Psi\biggr)^{-1-\gamma}\,.
\end{equation*}
By \eqref{E-Lg*Id} and \eqref{PT1.8D} we obtain 
\begin{align}\label{PT1.8G}
\Lg^*F(x) &\;=\;
\bigl(\Lambda\bigl((1+\varepsilon) f\bigr)-\Lambda(f)-\varepsilon f \bigr)\,F(x)
\nonumber\\[5pt]
&\;\le\; -\gamma \bigl(f-\Lambda(f)\bigr)\,F(x)\nonumber\\[5pt]
&\;\le\; -\gamma \delta\,F(x)
\qquad\forall\,x\in\sB^c\,.
\end{align}
The Foster--Lyapunov equation in \eqref{PT1.8G}
implies of course that
the diffusion is geometrically ergodic.
The estimate in \eqref{ET1.8A} is  obtained as the one in
\eqref{E-erg} using \eqref{PT1.8F} and \eqref{PT1.8G}.
\qed\end{pT1.8}

%%%%%%%%%%%%%%%%%%%%%%%%%%%%%%%%%%%%%%%%%%%%%%%%%%%%%%%%%%%%%%%%%%%%%%%%%%%%%%%%
\section{Proofs of the results on the control problem}\label{S3}
%%%%%%%%%%%%%%%%%%%%%%%%%%%%%%%%%%%%%%%%%%%%%%%%%%%%%%%%%%%%%%%%%%%%%%%%%%%%%%%%

%%%%%%%%%%%%%%%%%%%%%%%%%%%%%%%%%%%%%%%%%%%%%%%%%%%%%%%%%%%%%%%%%%%%%%%%%%%%%%%%
For the proof of Proposition~\ref{P1.1} we need some auxiliary lemmas.
The lemma which follows is the nonlinear Dirichlet eigenvalue
problem studied in \cite{Quaas-08a}, combined with a result
from \cite[Lemma~2.1]{Biswas-11a}.

%%%%%%%%%%%%%%%%%%%%%%%%%%%%%%%%%%%%%%%%%%%%%%%%%%%%%%%%%%%%%%%%%%%%%%%%%%%%%%%%
\begin{lemma}\label{L3.1}
For each $n\in\NN$,
there exists a unique pair
$(\widehat{V}_{n},\Hat\lambda_n)
\in\bigl(C^2(B_{n})\cap\Cc(\Bar{B}_{n})\bigr)\times\RR$, satisfying
$\widehat{V}_{n}>0$ on $B_{n}$,
$\widehat{V}_{n}=0$ on $\partial B_{n}$, and $\widehat{V}_{n}(0)=1$,
which solves
\begin{equation*}%\label{EL3.1A}
\min_{u\in\Act}\; \bigl[\Lg^{u} \widehat{V}_{n}(x) + c(x,u)\,\widehat{V}_{n}(x)\bigr]
\;=\; \Hat\lambda_n\,\widehat{V}_{n}(x)\,,\qquad x\in B_{n}\,.
\end{equation*}
Moreover, $\Hat\lambda_n<\Hat\lambda_{n+1}< \Lambda^*$ for all $n\in\NN$.
\end{lemma}
Let us add here that, as shown in \cite{Quaas-08a}, the non-linear
elliptic operator in Lemma~\ref{L3.1} has two principal eigenvalues.
In the setting of \cite{Quaas-08a},
$-\hat\lambda_n$ is the principal eigenvalue in $B_n$
with negative principal eigenfunction $-\widehat{V}_n$.
Then strict monotonicity of $\hat\lambda_n$ follows from 
\cite[Remark~3]{Quaas-08a} (or by the strong maximum principle).
%%%%%%%%%%%%%%%%%%%%%%%%%%%%%%%%%%%%%%%%%%%%%%%%%%%%%%%%%%%%%%%%%%%%%%%%%%%%%%%%
\begin{lemma}\label{L3.2}
Suppose that $\upsigma$ is bounded and
\begin{equation*}
\max_{u\in\Act}\;\frac{\bigl\langle b(x,u),\, x\bigr\rangle^{+}}{\abs{x}}
\;\xrightarrow[\abs{x}\to\infty]{}\;0\,.
\end{equation*}
Then,
\begin{equation*}
\limsup_{t\to\infty}\; \frac{1}{t}\;\Exp_{x}^{U}\bigl[\abs{X_{t}}\bigr]\;=\;0
\qquad\forall\,U\in\Uadm\,.
\end{equation*}
\end{lemma}

\begin{proof}
We claim that for each $\varepsilon>0$ there exists
a positive constant $C_{\varepsilon}$ such that
$\varepsilon C_{\varepsilon}\to 0$, as $\varepsilon\searrow0$, and
\begin{equation*}
\max_{u\in\Act}\;\bigl\langle b(x,u),\, x\bigr\rangle^{+}\;\le\; C_{\varepsilon}
\bigl(1+\varepsilon\,\abs{x}\bigr)\qquad \forall\,x\in\Rd\,.
\end{equation*}
Indeed, if $f$ is nonnegative
and $f(x)\in\sorder(\abs{x})$, we
write
\begin{align*}
f(x) &\;\le\; \sup_{\abs{x}<R}\,f(x)
+ \biggl(\sup_{\abs{x}\ge R}\,\frac{f(x)}{\abs{x}}\biggr) \abs{x}\\[5pt]
&\;=\; M_{R} + \varepsilon_{R} \abs{x}\\[5pt]
&\;<\; 1+M_{R} + \varepsilon_{R} \abs{x}\\[5pt]
&\;=\; (1+M_{R})\Bigl( 1 + \tfrac{\varepsilon_{R}}{1+M_{R}} \abs{x}\Bigr)\,,
\end{align*}
which proves the claim since $\varepsilon_{R}\searrow0$ as $R\nearrow\infty$.

Applying It\^o's formula in \eqref{E2.2}, under any control $U\in\Uadm$, we have
\begin{align}\label{EL3.2A}
\Exp_{x}^{U}\bigl[\abs{X_{t}}^{2}\bigr]&\;\le\; \abs{x}^{2}
+ \int_{0}^{t} \Exp_{x}^{U}\bigl[ 2\, \langle b(X_{s},U_{s})\,, X_{s}\rangle^{+}
+ \trace\bigl(a(X_{s})\bigr)\bigr]\,\D{s}
\nonumber\\[5pt]
&\;\le\; \abs{x}^{2} + C_{\varepsilon}'\int_{0}^{t}
\bigl(1+\varepsilon\,\Exp_{x}^{U}[\abs{X_{s}}]\bigr)\,\D{s}\,,
\end{align}
where $C_{\varepsilon}'$ also satisfies
$\varepsilon C_{\varepsilon}'\to 0$, as $\varepsilon\searrow0$.
Let $\varphi(t)$ denote the right hand side of \eqref{EL3.2A}.
Then
\begin{equation*}
\Dot\varphi(t)\;\le\;
C_{\varepsilon}'\bigl(1 + \varepsilon \sqrt{\varphi(t)}\bigr)\,,
\end{equation*}
which implies that
\begin{equation}\label{EL3.2B}
\frac{\Dot\varphi(t)}{\sqrt{\varepsilon^{-2} + \varphi(t)}}
\;\le\; 2\,\varepsilon C_{\varepsilon}'\,.
\end{equation}
Integrating \eqref{EL3.2B}, we have
$\sqrt{\varepsilon^{-2} + \varphi(t)}\le \varepsilon C_{\varepsilon}'\, t
+ \sqrt{\varepsilon^{-2} + \abs{x}^2}$
and using \eqref{EL3.2A}, we obtain
\begin{align}\label{EL3.2C}
\Exp_{x}^{U}\bigl[\abs{X_{t}}\bigr]&\;\le\;\sqrt{ \varphi(t)}\nonumber\\[5pt]
&\;\le\;
\sqrt{\varepsilon^{-2} + \varphi(t)}\;\le\;\varepsilon C_{\varepsilon}'\, t
+\sqrt{\varepsilon^{-2} + \abs{x}^2}\,.
\end{align}
Since \eqref{EL3.2C} holds for all $\varepsilon>0$, the result follows.
\qed\end{proof}

We need the following lemma.

%%%%%%%%%%%%%%%%%%%%%%%%%%%%%%%%%%%%%%%%%%%%%%%%%%%%%%%%%%%%%%%%%%%%%%%%%%%%%%%%
\begin{lemma}\label{L3.3}
Let $\varphi\in \Sobl^{2,p}$, $p> d$, be a strong, positive solution of 
\begin{equation*}
\Lg \varphi + h \varphi \;=\; 0\,,
\end{equation*}
where $h\in L^\infty(\Rd)$.
Suppose that $b$ and $\upsigma$ are bounded and $\upsigma$ is
Lipschitz continuous. Then there exists a constant $\tilde{C}$ such that
\begin{equation*}
\sup_{x\in\Rd}\, \frac{\abs{\grad\varphi(x)}}{\varphi(x)}\; <\; \tilde{C}\,.
\end{equation*}
\end{lemma}

\begin{proof}
Let $x\in\Rd$. By \eqref{E-Harn} we obtain
\begin{equation}\label{EL3.3A}
\varphi(y) \;\le\; C_H\, \varphi(z) \quad \text{for all\ } z, y\in B_2(x)\,.
\end{equation}
Since the coefficients $b$, $\upsigma$ and $h$ are bounded,
the Harnack constant $C_H$ in \eqref{EL3.3A} does not depend on $x$.
Again, applying \eqref{E-apriori} and the Sobolev embedding theorem
mentioned in the beginning of Section~\ref{S3}, we obtain
\begin{equation*}
\sup_{y\in B_1(x)}\,\abs{\grad\varphi(y)} \;\le\;
\tilde{C} \sup_{z\in B_2(x)}\,\varphi(z) \;\le\;
\tilde{C}\, C_H \inf_{z\in B_2(x)}\, \varphi(z) \;\le\;
 \tilde{C}\, C_H\varphi(x)\,,
\end{equation*}
where the constant $\tilde{C}$ does not depend on $x$.
Hence the result follows.
\qed\end{proof}

%%%%%%%%%%%%%%%%%%%%%%%%%%%%%%%%%%%%%%%%%%%%%%%%%%%%%%%%%%%%%%%%%%%%%%%%%%%%%%%%
\begin{theorem}\label{T3.4}
Suppose Assumption~\ref{A1.1} and \eqref{EA1.2} hold, and that
$c$ is near-monotone relative to $\Lambda^*$ and bounded.
Then
\begin{itemize}
\item[{\upshape(}i\/{\upshape)}]
There exists a solution $(V^*,\Lambda^*)\in\Cc^{2}(\RR^{d})\times\RR$,
satisfying $V^*(0)=1$ and $\inf_{\Rd}\,V^*>0$, to the
HJB equation
\begin{equation*}
\min_{u\in\Act}\;
\bigl[\Lg^{u} V^*(x) + c(x,u)\,V^*(x)\bigr] \;=\; \Lambda^*\,V^*(x)
\qquad\forall\,x\in\Rd\,.
\end{equation*}
\item[{\upshape(}ii\/{\upshape)}]
If $(V,\Lambda)\in\mathfrak{V}$ satisfies
\begin{equation}\label{ET3.4A}
\min_{u\in\Act}\;
\bigl[\Lg^{u} V(x) + c(x,u)\,V(x)\bigr] \;=\; \Lambda\,V(x)
\qquad\forall\,x\in\Rd\,,
\end{equation}
then $\Lambda=\Lambda^*$, and $\inf_{\Rd}\,V>0$.
In addition, any measurable selector from the minimizer of \eqref{ET3.4A}
belongs to $\Ussm$ and is optimal, i.e., $\Lambda^v_x=\Lambda^*$
for all $x\in\Rd$.
\end{itemize}
\end{theorem}

\begin{proof}
As shown in \cite[Lemma~2.1]{Biswas-11a},
any limit point, 
$(V^*,\lambda^*)\in\Cc^{2}(\RR^{d})\times\RR$
of the eigensolutions $\bigl(\widehat{V}_{n},\Hat\lambda_{n}\bigr)$
of the Dirichlet problem on $B_{n}$ in Lemma~\ref{L3.1}, as $n\to\infty$,
satisfies
\begin{equation}\label{ET3.4B}
\min_{u\in\Act}\;
\bigl[\Lg^{u} V^*(x) + c(x,u)\,V^*(x)\bigr] \;=\; \lambda^*\,V^*(x)
\qquad\forall\,x\in\Rd\,.
\end{equation}
Clearly then,
$V^*>0$ on $\RR^{d}$, $V^*(0)=1$, and $\lambda^*\le\Lambda^*$
by Lemma~\ref{L3.1}.
Since $b$, $\upsigma$, and $c$ are bounded,
and $\grad(\log V^*) = \frac{\grad V^*}{V^*}$,
then by Lemma~\ref{L3.3} there exists
a constant $\kappa>0$ such that
\begin{equation}\label{ET3.4C}
\E^{-\kappa(1+\abs{x})}\;\le\;V^*(x)\;\le\;
\E^{\kappa(1+\abs{x})}\qquad\forall\,x\in\Rd\,.
\end{equation}
Let $v$ be a measurable selector from the minimizer of \eqref{ET3.4B}.
A straightforward application of Fatou's lemma on the
stochastic representation of the solution $V^*$ of \eqref{ET3.4B}
shows that
\begin{equation}\label{ET3.4D}
V^*(x)\;\ge\; \Exp_{x}^{v}
\Bigl[\E^{\int_{0}^{T} [c(X_{t},v(X_{t}))-\lambda^*]\,\D{t}}\,
V^*(X_{T})\Bigr] \qquad \forall\,T>0\,.
\end{equation}
Evaluating \eqref{ET3.4D} at $x=0$, taking the logarithm on both sides,
applying Jensen's inequality,
dividing by $T$, and rearranging terms, we obtain
\begin{equation}\label{ET3.4E}
\frac{1}{T}\;\Exp_{x}^{v}\biggl[\int_{0}^{T}
c\bigl(X_{t},v(X_{t})\bigr)\,\D{t}\biggr]
+\frac{1}{T}\;\Exp_{x}^{v}\bigl[\log{V^*(X_{T})}\bigr] \;\le\;
\lambda^* + \frac{1}{T} \log V^*(x)\,.
\end{equation}
Hence, since $\babs{\log{V^*(X_{T})}}\le \kappa\,\bigl(1+\abs{X_{T}}\bigr)$
by \eqref{ET3.4C},
it follows by Lemma~\ref{L3.2} that
\begin{equation*}
\limsup_{T\to\infty}\;\frac{1}{T}\;
\Exp_{x}^{v}\bigl[\babs{\log{V^*(X_{T})}}\bigr]\;=\;0\,.
\end{equation*}
Therefore, by \eqref{ET3.4E} we obtain
\begin{equation*}
\limsup_{T\to\infty}\;
\frac{1}{T}\;\Exp_{x}^{v}\biggl[\int_{0}^{T}
c\bigl(X_{t},v(X_{t})\bigr)\,\D{t}\biggr]\;\le\;\lambda^*\,.
\end{equation*}
Since $\lambda^*\;\le\;\Lambda^*$, it follows that $v$ is stable
by the  argument used in the proof of (g)\;$\Rightarrow$\;(a) in Lemma~\ref{L2.1}.
Also $V^*$ is inf-compact, and $\Lambda^{v}(c)\le\lambda^*$
by Lemma~\ref{L2.1}\,(d), and (f), respectively.
Thus we have shown that
\begin{equation*}
\Lambda_x^{v}(c)\;\le\;\lambda^*\;\le\;\Lambda^*\qquad\forall\,x\in\Rd\,.
\end{equation*}
Therefore, we have $\lambda^*=\Lambda^*=\Lambda_x^{v}(c)$
for all $x\in\Rd$ by the definition of $\Lambda^*$.
This proves part (i).

The proof of parts (ii) follows by the same arguments as in
the preceding paragraph.
\qed\end{proof}

In order to simplify the notation, we define
\begin{equation*}
\Bar{c}_{v}(x)\;\df\; c\bigl(x,v(x)\bigr)\qquad\forall\,x\in\Rd\,,\quad
v\in\Usm\,.
\end{equation*}
Also recall that
\begin{equation*}
\Usm^* \;\df\;
\bigl\{v\in\Usm\, \colon \Lambda^v_x(c) = \lm\,,\ \forall\, x\in\Rd\bigr\}\,.
\end{equation*}
The proof in Lemma~\ref{L2.1}, (f)\;$\Rightarrow$\;(a),
shows that any $v\in\Usm^*$ is a stable Markov control.

In the two lemmas which follow we do not impose Assumption~\ref{A1.1}.

%%%%%%%%%%%%%%%%%%%%%%%%%%%%%%%%%%%%%%%%%%%%%%%%%%%%%%%%%%%%%%%%%%%%%%%%%%%%%%%%
\begin{lemma}\label{L3.5}
Suppose that $c$ is near-monotone relative to $\lm$,
and that $V\in\Cc^{2}(\RR^{d})$ is a positive solution of
\begin{equation}\label{EL3.5A}
\min_{u\in\Act}\;
\bigl[\Lg^{u} V(x) + c(x,u)\,V(x)\bigr] \;=\; \lm\,V(x)
\qquad\forall\,x\in\Rd\,.
\end{equation}
Let $v\in\Usm^*$, and $V_v\in\Sobl^{2,p}(\Rd)$, $p>d$,
be a positive solution of
\begin{equation}\label{EL3.5B}
\Lg_v V_v(x)
+ \Bar{c}_{v}(x)\,V_v(x)
\;=\; \lm\,V_v(x)\qquad \text{a.e.\ in\ }\Rd\,.
\end{equation}
If for some ball $\sB\in\Rd$ it holds that
\begin{equation*}
V(x) \;\le\; \Exp_x^{v}\Bigl[\E^{\int_{0}^{\uuptau}
[\Bar{c}_{v}(X_s)-\lm]\,\D{s}}\,V(X_{\uuptau})\Bigr]
\qquad\forall\,x\in \sB^c\,,
\end{equation*}
with $\uuptau=\uptau(\sB^c)$, then $\frac{V}{V_v}$ is constant on $\Rd$.
In particular,
\begin{equation*}
\Lg_{v} V(x) + \Bar{c}_{v}(x)\,V(x) \;=\; \lm\,V(x)
\qquad \text{a.e.\ in\ }\Rd\,.
\end{equation*}
\end{lemma}

\begin{proof}
By It\^o's formula and the Fatou lemma, we have
\begin{equation*}
V_v(x) \;\ge\; \Exp_x^{v}\Bigl[\E^{\int_{0}^{\uuptau}
[\Bar{c}_{v}(X_s)-\lm]\,\D{s}}\,V_v(X_{\uuptau})\Bigr]
\qquad\forall\,x\in \sB^c\,.
\end{equation*}
Therefore we can scale $V_v$, by multiplying with a positive constant,
so that it touches $V$ from above at some point in $\sB$.
Since \eqref{EL3.5A} implies that
\begin{equation*}
\Lg_{v} V(x) + \Bar{c}_{v}(x)\,V(x)\;\ge\; \lm\,V(x)
\qquad \text{a.e.\ in\ }\Rd\,,
\end{equation*}
and thus,
\begin{equation*}
\Lg_{v} (V_v-V)(x) - (\Bar{c}_{v}(x)-\lm)^{-}\,(V_v-V)(x)\;\le\; 0\,
\qquad \text{a.e.\ in\ }\Rd\,,
\end{equation*}
the result follows by the strong maximum principle.
\qed
\end{proof}

%%%%%%%%%%%%%%%%%%%%%%%%%%%%%%%%%%%%%%%%%%%%%%%%%%%%%%%%%%%%%%%%%%%%%%%%%%%%%%%%
\begin{lemma}\label{L3.6}
Let $c$ be near-monotone relative to $\lm$, and
suppose that $\lm(c+h)>\lm(c)$ for all $h\in\sCo$.
Then for any positive solution $V_v\in\Sobl^{2,p}(\Rd)$, $p>d$,
of \eqref{EL3.5B},  with $v\in\Usm^*$,
it holds that
\begin{equation}\label{EL3.6A}
\min_{u\in\Act}\;
\bigl[\Lg^{u} V_v(x) + c(x,u)\,V_v(x)\bigr] \;=\; \lm\,V_v(x)
\qquad\forall\,x\in\Rd\,.
\end{equation}
\end{lemma}

\begin{proof}
Let $\Tilde{v}$ be a minimizing selector of
\begin{equation*}
x\;\mapsto\;\argmin_{u\in\Act}\;
\bigl\{\langle b(x, u), \grad V_v(x)\rangle + c_v(x)\, V_v(x)\bigr\}\,.
\end{equation*}
Then $\Lg_{\Tilde{v}} V_v(x) + c_{\tilde{v}}(x)\,V_v(x)\le \lm\,V_v$ which implies
that $\Lambda^{\Tilde{v}}(c)\leq\lm$, and thus $\Tilde{v}\in\Usm^*$
and $\Lambda^{\Tilde{v}}(c)=\lm$.
If \eqref{EL3.6A} is not an equality, then for some measurable
set $A$ of positive Lebesgue measure we have
\begin{equation*}
\Lg_{\Tilde{v}} V_v
+ \bigl(\Bar{c}_{\Tilde{v}}+\Ind_{A}\bigr)\,V_v
\;\leq\; \lm\,V_v\qquad \text{a.e.\ in\ }\Rd\,.
\end{equation*}
Thus $\Lambda^{\Tilde{v}}(c+\Ind_{A})\le\lm$ by Lemma~\ref{L2.1}\,(f).
However, if \eqref{ET1.2A} holds, then by Theorem~\ref{T1.5} we have
\begin{equation*}
\Lambda^{\Tilde{v}}(c+\Ind_{A})\;>\;\Lambda^{\Tilde{v}}(c)\;=\;\lm\,,
\end{equation*}
 which is a contradiction.
Thus \eqref{EL3.6A} holds.
\qed
\end{proof}

\begin{remark}\label{R3.1}
Let $\mathscr{V}\,\df\,\bigl\{V\in\Cc^{2}(\RR^{d})\colon\,
V(0)=1\,,~\inf_{\Rd}\,V_v>0\bigr\}$, and
\begin{align*}
\overline{\mathfrak{G}}&\;\df\; \bigl\{V\in\mathscr{V}\colon\,
V\text{\ solves\ \eqref{EL3.5A}}\bigr\}\,,\\[5pt]
\mathfrak{G}&\;\df\; \bigl\{V_v\in\mathscr{V}\colon\,
V_v\text{\ solves\ \eqref{EL3.5B}\ for some\ }v\in\Usm^*\bigr\}\,.
\end{align*}
It follows by Lemma~\ref{L3.6} that, if
$\lm(c+h)>\lm(c)$ for all $h\in\sCo$, then $\overline{\mathfrak{G}}=\mathfrak{G}$.
\end{remark}

We continue with the proof of Theorem~\ref{T1.2}.

%%%%%%%%%%%%%%%%%%%%%%%%%%%%%%%%%%%%%%%%%%%%%%%%%%%%%%%%%%%%%%%%%%%%%%%%%%%%%%%%
\begin{pT1.2}
In view of Theorem~\ref{T3.4}\,(ii) it suffices to prove uniqueness
in the class $\mathfrak{V}_{\circ}$.  Also recall that $\Lambda^*=\lm$.
Since the Dirichlet eigenvalues in Lemma~\ref{L3.1} satisfy
$\Hat\lambda_n<\Lambda^*$ for all $n\in\NN$,
 the Dirichlet problem
\begin{equation}\label{PT1.2A}
\min_{u\in\Act}\;
\bigl[\Lg^u\varphi_{n}(x)+ \bigl(c(x,u)-\Lambda^*\bigr)\,\varphi_{n}(x)\bigr]
\;=\; -\alpha_n\,\Ind_{\sB}(x)\qquad\text{a.e.\ }x\in B_n\,,\qquad
\varphi_{n}=0\text{\ \ on\ \ }\partial B_n\,,
\end{equation}
with $\alpha_n>0$,
has a unique solution $\varphi_{n}\in\Sobl^{2,p}(B_n)\cap\Cc(\Bar{B}_n)$,
for any $p\ge1$ \cite[Theorem~1.9]{Quaas-08a}
(see also \cite[Theorem~1.1\,(ii)]{Yoshimura-06}).
We choose $\alpha_n$ as in Lemma~\ref{L2.11}.
Namely if $\Tilde\alpha_n>0$ is such that the solution
$\varphi_n$ of \eqref{PT1.2A} with $\alpha_n=\Tilde\alpha_n$
satisfies $\varphi_n(0)=1$, we set $\alpha_n=\min(1,\Tilde\alpha_n)$.
Passing to the limit to obtain a positive solution
$\Phi\in\Sobl^{2,p}(\Rd)$ of
\begin{equation}\label{PT1.2B}
\min_{u\in\Act}\;
\bigl[\Lg^u\Phi(x)+ \bigl(c(x,u)-\Lambda^*\bigr)\,\Phi(x)\bigr]
\;=\; -\alpha\,\Ind_{\sB}(x)\,,\qquad x\in\Rd\,.
\end{equation}

We claim that $\inf_{\Rd}\,\Phi>0$.
Indeed, let $\Hat{v}$ be a measurable selector from
the minimizer of \eqref{PT1.2B}.
Then
\begin{equation}\label{PT1.2Bb}
\Lg_{\Hat{v}}\Phi(x)+ \bigl(\Bar{c}_{\Hat{v}}(x)-\Lambda^*\bigr)\,\Phi(x)
\;=\; -\alpha\,\Ind_{\sB}(x)\,,\qquad \text{a.e.\ in\ }\Rd\,.
\end{equation}
Consider a ball $\Breve\sB$ which is concentric with $\sB$ and of twice the
radius.
By the proof of Lemma~\ref{L3.3} we obtain
\begin{equation*}
\sup_{x\in\Breve\sB^c}\, \frac{\abs{\grad\Phi(x)}}{\Phi(x)}\; <\; \Breve{C}\,,
\end{equation*}
for some constant $\Breve{C}$.
Since $\inf_{\Breve\sB}\,\Phi>0$,
it follows that
$\frac{\abs{\grad\Phi(x)}}{\Phi(x)}$ is bounded on $\Rd$.
Hence a bound as in \eqref{ET3.4C} holds for $\Phi$.
By It\^o's formula and Fatou's lemma we obtain from \eqref{PT1.2Bb} that
\begin{equation*}
\Phi(x)\;\ge\; \Exp_{x}^{v}
\Bigl[\E^{\int_{0}^{T} [\Bar{c}_{\Hat{v}}(X_{t})-\Lambda^*]\,\D{t}}\,
\Phi(X_{T})\Bigr] \qquad \forall\,T>0\,,
\end{equation*}
and then proceed exactly as in the proof of Theorem~\ref{T3.4} to show
that $\Phi$ is inf-compact.
This proves the claim.

Let $v$ be any control in $\Usm^*$.
By It\^o's formula applied to \eqref{PT1.2A}, with $\uuptau=\uptau(\sB^c)$,
we obtain
\begin{multline*}
\varphi_n(x) \;\le\; \Exp_x\Bigl[\E^{\int_{0}^{\uuptau}
[\Bar{c}_{v}(X_s)-\lm]\,\D{s}}\,
\varphi_n(X_{\uuptau})\,\Ind_{\{\uuptau<T\wedge\uptau_n\}}\Bigr]\\[5pt]
+\Exp_x\Bigl[\E^{\int_{0}^{T}[\Bar{c}_{v}(X_s)-\lm]\,\D{s}}\,
\varphi_n(X_{T})\,
\Ind_{\{T<\uuptau\wedge\uptau_n\}}\Bigr]
\qquad\forall\,x\in B_n\setminus\sB\,,\ \forall\,T>0\,.
\end{multline*}
Then we use the argument in the proof of Lemma~\ref{L2.11},
by replacing $f$ with $\Bar{c}_{v}$, and $\Lambda(f)$ with
$\lm$ in \eqref{EL2.11B}--\eqref{EL2.11D}
to obtain
\begin{equation}\label{PT1.2C}
\Phi(x) \;\le\; \Exp_x\Bigl[\E^{\int_{0}^{\uuptau}[\Bar{c}_v(X_s)-\Lambda^*]\,\D{s}}\,
\Phi(X_{\uuptau})\Bigr]
\qquad\forall\,x\in \sB^c\,,
\end{equation}
for all $v\in\Usm^*$.

Using \eqref{PT1.2Bb} and the
assumption $\lm(c+h)>\lm(c)$ for all $h\in\sCo$, we conclude as
in the proof of Lemma~\ref{L2.11} that $\alpha=0$.
Recall the definition of $\overline{\mathfrak{G}}$ and $\mathfrak{G}$ in
Remark~\ref{R3.1}.
Thus by Lemma~\ref{L3.5} we obtain $\mathfrak{G}=\{\Phi\}$,
and since $\overline{\mathfrak{G}}=\mathfrak{G}$ by Remark~\ref{R3.1},
uniqueness of the solution to \eqref{ET1.2B} and the verification
part of the result follow.

Since by Fatou's lemma the converse inequality to \eqref{PT1.2C} holds,
the representation in \eqref{ET1.2C} follows by Theorem~\ref{T1.5}.

For the converse, note that $\lm(c+h)=\lm(c)$ implies the
existence of $v\in\Usm$ such that
\begin{equation*}
\Lambda^v_x(c+h)\;=\;\lm(c+h)\;=\;\lm(c)\;\le\;\Lambda^v_x(c)
\end{equation*}
for all $x\in\Rd$. But $\Lambda^v_x(c+h)\geq \Lambda^v_x(c)$.
Therefore, $v\in\Usm^*$ and 
$\Lambda^v_x(c+h)= \Lambda^v_x(c)$.
If this is the case, then \eqref{ET1.2C} contradicts Theorem~\ref{T1.5}.
This completes the proof.
\qed
\end{pT1.2}

%%%%%%%%%%%%%%%%%%%%%%%%%%%%%%%%%%%%%%%%%%%%%%%%%%%%%%%%%%%%%%%%%%%%%%%%%%%%%%%%
\begin{pP1.3}
Let $\varepsilon_0>0$ be small enough so that
$c$ is near-monotone relative to $\lm+\varepsilon_0$,
and for $\varepsilon\in(0,\varepsilon_0)$,
let $v_{\varepsilon}\in\Usm$ be a $\varepsilon$-optimal control relative
to $\lm$. 
In other words, $v_{\varepsilon}$ satisfies
$\Lambda^{v_{\varepsilon}}\le \lm+\varepsilon$.
An analogous argument as in the proof
(g)\;$\Rightarrow$\;(a) of Lemma~\ref{L2.1} shows that
$v_\varepsilon$ is a stable Markov control.
Define
\begin{equation*}
b^{v_{\varepsilon}}_{n}(x,u)\;\df\;\begin{cases}
b(x,u)&\text{for}~x\in B_{n}\,,\quad u\in\Act\,,\\[5pt]
b\bigl(x,v_{\varepsilon}(x)\bigr)&\text{for}~x\in B_{n}^{c}\,,\quad u\in\Act\,,
\end{cases}
\end{equation*}
and $c^{v_{\varepsilon}}_{n}(x,u)$ in an exactly analogous manner.
Let $\Usm^{n,v_{\varepsilon}}$ denote the class of stationary Markov controls
that agree with $v_{\varepsilon}$ on $B_{n}^{c}$.
By \cite[Theorem~1.1]{Quaas-08a}, there exists a unique pair
$(V^{v_{\varepsilon}}_{n,k},\lambda^{v_{\varepsilon}}_{n,k})
\in\bigl(\Sob^{2,p}(B_{k})\cap\Cc(\Bar{B}_{k})\bigr)\times\RR$,
for any $p>d$, satisfying
$V^{v_{\varepsilon}}_{n,k}>0$ on $B_{k}$ and $V^{v_{\varepsilon}}_{n,k}(0)=1$,
which solves
\begin{equation}\label{PP1.3A}
\min_{u\in\Act}\; \bigl[\Lg^{u} V^{v_{\varepsilon}}_{n,k}(x)
+ c^{v_{\varepsilon}}_{n}(x,u)\,V^{v_{\varepsilon}}_{n,k}(x)\bigr]
\;=\; \lambda^{v_{\varepsilon}}_{n,k}\,V^{v_{\varepsilon}}_{n,k}(x)
\qquad\text{a.e.\ }x\in B_{k}\,,
\end{equation}
and $V^{v_{\varepsilon}}_{n,k}=0$ on $\partial B_{k}$
(compare with Lemma~\ref{L3.1}).
Following the proof of \cite[Lemma~2.1]{Biswas-11a}, we deduce that
$\lambda^{v_{\varepsilon}}_{n,k}\le\Lambda^{v_{\varepsilon}}_x$ for all $x\in B_k$.
Since $V^{v_{\varepsilon}}_{n,k}$ is locally bounded, uniformly in $k\in\NN$,
by Harnack's inequality,
we have 
$\sup_k\,\norm{V^{v_{\varepsilon}}_{n,k}}_{\Sob^{2,p}(B_R)}<\infty$
for  all $R>0$, and any $p>d$.
Thus taking limits as $k\to\infty$ along some subsequence, we obtain
by \eqref{PP1.3A} a pair
$(V^{v_{\varepsilon}}_{n},\lambda^{v_{\varepsilon}}_{n})
\in\Sobl^{2,p}(\Rd)\times\RR$,
for any $p>d$, satisfying $\lambda^{v_{\varepsilon}}_{n}\le\Lambda^{v_{\varepsilon}}$,
$V^{v_{\varepsilon}}_{n}>0$ on $\Rd$, and $V^{v_{\varepsilon}}_{n}(0)=1$,
which solves
\begin{equation}\label{PP1.3B}
\min_{u\in\Act}\; \bigl[\Lg^{u} V^{v_{\varepsilon}}_{n}(x)
+ c^{v_{\varepsilon}}_{n}(x,u)\,V^{v_{\varepsilon}}_{n}(x)\bigr]
\;=\; \lambda^{v_{\varepsilon}}_{n}\,V^{v_{\varepsilon}}_{n}(x)\qquad
\text{a.e.\ }x\in\Rd\,.
\end{equation}
It also follows by elliptic regularity that the restriction of
$V^{v_{\varepsilon}}_{n}$ in $B_{n}$ is in $\Cc^{2}(\Rd)$.

Let $\Hat{v}_{n}\in\Usm^{n,v_{\varepsilon}}$
be a  measurable selector from the minimizer of \eqref{PP1.3B}.
Since $\Hat{v}_{n}$ agrees with $v_{\varepsilon}$ on $B_n^c$,
any Lyapunov function for the diffusion
under the control $v_{\varepsilon}$ is also a Lyapunov function
under the control $\Hat{v}_{n}$.
It follows that $\Hat{v}_{n}\in\Ussm$.
Let $\sB$ be a bounded open ball such that
$c(x,u)>\lm+\varepsilon_0$ for all $(x,u)\in\sB^c\times\Act$.
Applying It\^o's formula to \eqref{PP1.3B}, and using Fatou's lemma,
we obtain, with $\uuptau\equiv\uptau\bigl(\sB^{c}\bigr)$, that
\begin{align}\label{PP1.3C}
V^{v_{\varepsilon}}_{n}(x)&\;\ge\; \Exp_{x}^{\Hat{v}_{n}}
\Bigl[\E^{\int_{0}^{\uuptau}
[c^{v_{\varepsilon}}_{n}(X_{t},\Hat{v}_{n}(X_{t}))
-\lambda^{v_{\varepsilon}}_{n}]\,\D{t}}\,
V^{v_{\varepsilon}}_{n}(X_{\uuptau})\Bigr]\nonumber\\[5pt]
&\;\ge\; \biggl(\min_{\partial\sB(\varepsilon_0)}\;V^{v_{\varepsilon}}_{n}\biggr)\,
\Exp_{x}^{\Hat{v}_{n}}\Bigl[\exp\Bigl(\tfrac{1}{2}(\varepsilon_0-\varepsilon)\,
\uuptau\Bigr)\Bigr]\qquad \forall\, x\in\sB^{c}\,.
\end{align}
Since $V^{v_{\varepsilon}}_{n}(0)=1$, it follows by
the Harnack inequality that
$\min_{\partial\sB(\varepsilon_0)}\;V^{v_{\varepsilon}}_{n}\ge C_H^{-1}$.
Therefore, by \eqref{PP1.3C}, we have
\begin{equation}\label{PP1.3D}
\inf_{n\in\NN}\;\inf_{\Rd}\;V^{v_{\varepsilon}}_{n}\;>\;0\,.
\end{equation}
Again $\sup_n\,\norm{V^{v_{\varepsilon}}_{n}}_{\Sob^{2,p}(B_R)}<\infty$
for  all $R>0$, and any $p>d$.
Thus, taking limits in \eqref{PP1.3C} as $n\to\infty$, along some subsequence,
we obtain a pair
$(V^{v_{\varepsilon}},\lambda^{v_{\varepsilon}})
\in\Cc^{2}(\Rd)\times\RR$,
satisfying $\lambda^{v_{\varepsilon}}\le\Lambda^{v_{\varepsilon}}$,
$\inf_{\Rd}\;V^{v_{\varepsilon}}>0$, and $V^{v_{\varepsilon}}(0)=1$,
which solves
\begin{equation}\label{PP1.3E}
\min_{u\in\Act}\;\bigl[\Lg^{u} V^{v_{\varepsilon}}(x)
+ c(x,u)\,V^{v_{\varepsilon}}(x)\bigr]
\;=\; \lambda^{v_{\varepsilon}}\,V^{v_{\varepsilon}}(x)\qquad
\text{a.e.\ }x\in\Rd\,.
\end{equation}
Once more, taking any limit of \eqref{PP1.3E} as $\varepsilon\searrow0$ along
some subsequence, we obtain a function $V^*\in\Cc^{2}(\RR^{d})$, satisfying
\begin{equation}\label{PP1.3F}
\min_{u\in\Act}\;
\bigl[\Lg^{u} V^*(x) + c(x,u)\,V^*(x)\bigr] \;=\; \lm\,V^*(x)
\qquad\forall\,x\in\Rd\,.
\end{equation}
It holds that $\inf_{\Rd}\;V^*>0$ by \eqref{PP1.3D}, and $V^*(0)=1$
by construction.

Let $v^{*}$ be a measurable selector from the minimizer of \eqref{PP1.3F}.
Then $v^*\in\Ussm$, and $\Lambda_x^{v^{*}} = \lm$ for all $x\in\Rd$
by Lemma~\ref{L2.1}.

Suppose that $\lm(c+h)>\lm(c)$ for all $h\in\sCo$.
We follow a variation of the construction in the proof of Theorem~\ref{T2.8},
using \cite[Theorem~1.9]{Quaas-08a}.
Let $\theta>0$ be any positive constant such
that $c$ is near monotone relative to $\lm+2\theta$,
and fix some $\alpha\in(0,\theta)$.
Let $\zeta_{\alpha}>0$.
As argued in the  proof of Theorem~\ref{T1.2}, the Dirichlet problem
\begin{equation}\label{PP1.3G}
\min_{u\in\Act}\,\bigl[\Lg^u \varphi_{\alpha,n}(x)+ 
\bigl(c(x,u) - \lm - \alpha\bigr)\,\varphi_{\alpha,n}(x)\bigr]
\;=\; -\zeta_{\alpha}\qquad\text{a.e.\ }x\in B_n\,,\qquad
\varphi_{\alpha}=0\text{\ \ on\ \ }\partial B_n\,,
\end{equation}
has a unique nonnegative solution $\varphi_{\alpha,n}\in\Cc^2(B_n)\cap\Cc(\Bar{B}_n)$.
It is clear by the strong maximum principle that $\varphi_{\alpha,n}$
is positive in $B_n$.
For $v\in\Usm^*$, let
$V_v\in\Sobl^{2,p}(\Rd)$, $p>d$, $V_v(0)=1$, be a positive solution of
\begin{equation}\label{PP1.3H}
\Lg_{v} V_v(x) + \Bar{c}_{v}(x)\,V_v(x)
\;=\; \lm\,V_v(x)\qquad \text{a.e.\ in\ }\Rd\,.
\end{equation}
Writing \eqref{PP1.3H} as
\begin{equation*}
\Lg_v V_v(x) + (\Bar{c}_{v}(x)-\lm-\alpha)\,V_v(x)
\;=\; -\alpha\,V_v(x)\qquad\text{a.e.\ }x\in\Rd\,,
\end{equation*}
and using It\^o's formula and Fatou's lemma we obtain
\begin{equation*}
V_v(x) \;\ge\; \Exp_x\Bigl[\E^{\int_0^{\uptau_n}
[\Bar{c}_v(X_s) - \lm - \alpha]\,\D{s}}\,V_v(X_{\uptau_n})\Bigr]
+ \alpha\,\Exp_x\biggl[\int_0^{\uptau_n}
\E^{\int_0^{t}[\Bar{c}_v(X_s) - \lm - \alpha]\,\D{s}}\,V_v(X_t)\,\D{t}\biggr]
\end{equation*}
for any $\alpha\ge0$,
and we use this as in the proof of Theorem~\ref{T2.8} to deduce that
\begin{equation*}
\varphi_{\alpha,n}\;\le\;
\frac{\zeta_{\alpha}}{\alpha}\,
\Bigl(\inf_{\Rd}\,V_v\Bigr)V_v \qquad\forall\,n\in\NN\,,
\end{equation*}
and for all $v\in\Usm^*$.
As explained in Theorem~\ref{T2.8}, this allows us to pass to a limit along
some sequence
$n\to\infty$, to obtain a positive function $\varPhi_\alpha\in\Cc^2(\Rd)$ which solves
\begin{equation}\label{PP1.3I}
\min_{u\in\Act}\,\bigl[\Lg^u \varPhi_\alpha(x)+ 
\bigl(c(x,u) - \lm - \alpha\bigr)\,\varPhi_\alpha(x)\bigr]
\;=\; -\zeta_{\alpha}\qquad\forall\,x\in\Rd\,.
\end{equation}

By It\^o's formula we obtain from \eqref{PP1.3G} that
\begin{multline}\label{PP1.3J}
\varphi_{\alpha,n}(x) \;\le\;  \Exp_x\Bigl[\E^{\int_0^{\uuptau_r\wedge T}
[\Bar{c}_v(X_s) - \lm - \alpha]\,\D{s}}\,\varphi_{\alpha,n}(X_{\uuptau_r\wedge T})\,
\Ind_{\{\uuptau_r\wedge T<\uptau_n\}}\Bigr]\\[5pt]
+\zeta_\alpha\Exp_x\biggl[\int_0^{\uuptau_r\wedge T\wedge\uptau_n}
\E^{\int_0^{t}[\Bar{c}_v(X_s) - \lm - \alpha]\,\D{s}}\,\D{t}\biggr]
\qquad\forall\,(T,x)\in\RR_+\times (B_n\setminus B_r)\,,
\end{multline}
and for all $v\in\Usm^*$, where we use
the property that $\varphi_{\alpha,n}=0$ on $\partial B_n$.
Since
\begin{equation*}
\Exp_x\Bigl[\E^{\int_{0}^{\uptau_r\wedge T}[\Bar{c}_v(X_s) - \lm]\,\D{s}}\Bigr]
\;\le\; \Bigl(\inf_{\Rd}\,V_v\Bigr)^{-1}\,
V_v(x)\qquad\forall\,x\in B_n\setminus B_r\,,\ \forall\,r, T\,>0\,,
\end{equation*}
using dominated convergence for the first term on the right hand side
of \eqref{PP1.3J}, and monotone convergence for the second term,
we first take limits as $T\to\infty$, and then as $n\to\infty$ to
obtain
\begin{equation}\label{PP1.3K}
\varPhi_\alpha(x)\;\le\;
\Exp_x^v\Bigl[\E^{\int_{0}^{\uuptau_r}
[\Bar{c}_{v}(X_s) - \lm - \alpha]\,\D{s}}\,\varPhi_\alpha(X_{\uuptau_r})\Bigr]
+\zeta_\alpha\,
\Exp_x^v\biggl[\int_{0}^{\uuptau_r} \E^{\int_{0}^{t}
[\Bar{c}_{v}(X_s) - \lm - \alpha]\,\D{s}}\,\D{t}\biggr]
\end{equation}
for all $x\in B_r^c$, $r>0$, and $v\in\Usm^*$.

We next show that $\inf_{\Rd}\,\varPhi_\alpha\ge M$ for some constant
$M>0$ and all $\alpha>0$ sufficiently small.
Let $\theta>0$ be such that
$c$ is near monotone relative to $\Lambda(f)+2\theta$,
and let $\sK_\theta$ be as in in Definition~\ref{D1.1}.
If $\Hat{v}\in\Usm$ is a selector from the minimizer of \eqref{PP1.3I},
then by It\^o's formula and Fatou's lemma we obtain
\begin{equation}\label{PP1.3Xa}
\varPhi_\alpha(x)\;\ge\; \zeta_\alpha\,
\Exp_x^{\Hat{v}}\biggl[\int_0^{\infty}
\E^{\int_0^{t}[\Bar{c}_{\Hat{v}}(X_s) - \lm - \alpha]\,\D{s}}\,\D{t}\biggr]
\;\ge\;  \frac{\zeta_\alpha}{\lm +\alpha}\qquad\forall\,x\in\Rd\,,
\end{equation}
since the running cost $c$ is nonnegative.
From now on, we select $\zeta_\alpha$ so that $\varPhi_\alpha(0)=1$.
Thus $\zeta_\alpha\le \lm +\alpha$ by \eqref{PP1.3Xa}.
Since $\zeta_\alpha$ is bounded, the pde in \eqref{PP1.3I} satisfies
the assumptions for the Harnack inequality for a class of
superharmonic functions in \cite{AA-Harnack}.
This implies that if $\sB$ is some fixed ball containing $\sK_\theta$, then
\begin{equation}\label{PP1.3L}
C_{H}^{-1} \;\le\; \varPhi_\alpha(x) \;\le\; C_{H}\,,\qquad\forall\, x\in \Bar\sB\,,
\quad\forall\,\alpha\in(0,1)\,,
\end{equation}
for some constant $C_{H}$.
We leave it to the reader to verify that the assertions in Lemma~\ref{L2.1}
are true for any supersolution of \eqref{E1-Pois}.
Therefore, since $\inf_{\Rd}\, \varPhi_\alpha>0$, it follows from \eqref{PP1.3I} that
$\inf_\Rd\, \varPhi_\alpha=\inf_\sK\, \varPhi_\alpha\ge C_{H}^{-1}$
for all $\alpha\in(0,\theta)$.

On the other hand, as in \eqref{ET2.8K}, for any $v\in\Usm^*$, we obtain
\begin{equation}\label{PP1.3M}
\Exp_x^{v}\biggl[\int_{0}^{\uuptau} \E^{\int_{0}^{t}
[\Bar{c}_{v}(X_s) - \lm]\,\D{s}}\,\D{t}\biggr]
\;\le\; \Bigr(\theta\,\inf_{\partial \sB}\;V_v\Bigr)^{-1} V_v(x)
\qquad\forall\,x\in \sB^c\,,
\end{equation}
where $\sB$ is the ball in the preceding paragraph,
and $V_v$ is as in \eqref{PP1.3H}.
It then follows by \eqref{PP1.3K}, \eqref{PP1.3L}, and \eqref{PP1.3M}
that the collection $\{\varPhi_\alpha\,,\;\alpha\in(0,1)\}$ is
locally bounded, and thus also
relatively
weakly compact in $\Sob^{2,p}(B_R)$ for any $p\ge1$ and $R>0$, by \eqref{ET2.8G}.
Thus passing to the limit in
\eqref{PP1.3I} as $\alpha\searrow0$ along some sequence, we obtain
a positive $\overline{V}\in\Cc^2(\Rd)$, and a constant $\Bar\zeta$
solving
\begin{equation*}
\min_{u\in\Act}\,\bigl[\Lg^u \overline{V}(x)+ 
\bigl(c(x,u) - \lm \bigr)\,\overline{V}(x)\bigr]
\;=\; -\Bar\zeta\qquad\forall\,x\in\Rd\,.
\end{equation*}
Then $\inf_\Rd \Bar{V} \ge C_{H}^{-1}$.
Thus, the assumption that
$\lm(c+h)>\lm(c)$ for all $h\in\sCo$, then implies that $\Bar\zeta=0$.

It follows by \eqref{PP1.3K}--\eqref{PP1.3M}, that
\begin{equation*}
\overline{V}(x) \;\le\;  \Exp_x^{v}\Bigl[\E^{\int_{0}^{\uuptau}
[\Bar{c}_v(X_s) - \lm]\,\D{s}}\,\overline{V}(X_{\uuptau})\Bigr]
\qquad\forall\,x\in \sB^c\,,\quad\forall\,v\in\Usm^*\,.
\end{equation*}
The rest follows exactly as in the proof of Theorem~\ref{T1.2}.
\qed\end{pP1.3}

\begin{remark}\label{R3.2}
The reader might have noticed not only the different approach in
the proof of the existence results in Propositions~\ref{P1.1} and \ref{P1.3},
but also the difference between the method of proof of Theorem~\ref{T1.2} and
that of the analogous statement in Proposition~\ref{P1.3}.
This needs some explanation.
Consider the following definitions of eigenvalues.
\begin{align*}
\Hat\Lambda^* &\;\df\;
\inf\;\bigl\{\lambda\in\RR\,\colon \exists\, \varphi\in\Sobl^{2,d}(\Rd)\,,
\ \varphi>0\,,\ \min_{u\in\Act}\,[\Lg^u\varphi + (c(x,u)-\lambda)\varphi]
\le 0\ \text{a.e. in\ }\Rd\bigr\}\,,
\\[5pt]
\Hat\Lambda_{\mathrm{m}}^{*} &\;\df\; \inf_{v\in\Usm}\,
\inf\;\bigl\{\lambda\in\RR\,\colon \exists\, \varphi\in\Sobl^{2,d}(\Rd)\,,
\ \varphi>0\,,\ \Lg_v\varphi
+ (\Bar{c}_v-\lambda)\varphi \le 0\ \text{a.e. in\ }\Rd\bigr\}\,.
\end{align*}
Also define $\doublehat\Lambda^*$ and $\doublehat\Lambda_{\mathrm{m}}^{*}$
in direct analogy to these but with $\varphi>0$ in
the qualifier replaced by $\inf_{\Rd}\,\varphi>0$.
It is straightforward to verify that
$\Hat\Lambda^*=\Hat\Lambda_{\mathrm{m}}^{*}\le\doublehat\Lambda_{\mathrm{m}}^{*}$.
On the other hand since $c$ is near monotone relative to
$\Lambda^*$, then by Lemma~\ref{L2.1} and the proof of Theorem~\ref{T1.4}
we have $\doublehat\Lambda^*=\doublehat\Lambda_{\mathrm{m}}^{*}=\lm$,
so we obtain
\begin{equation}\label{ER3.2A}
\Hat\Lambda^*\;=\;\Hat\Lambda_{\mathrm{m}}^{*}
\;\le\;\Lambda^*\;\le\;
\doublehat\Lambda^*\;=\;\doublehat\Lambda_{\mathrm{m}}^{*}\;=\;\lm\,.
\end{equation}
Recall from the proof of Theorem~\ref{T3.4} that $\lambda^*$ denotes the limit,
as $n\to\infty$,
of the Dirichlet eigenvalues $\Hat\lambda_n$ defined in Lemma~\ref{L3.1}.
It is evident that $\lambda^*\le\Hat\Lambda^*$, and since $\lambda^*$
satisfies \eqref{ET3.4B}, we have in fact equality.
We have also shown that $\lambda^*=\Lambda^*$.
Note then that under Assumption~\ref{A1.1} we have
$\Hat\Lambda_{\mathrm{m}}^{*}=\doublehat\Lambda_{\mathrm{m}}^{*}$, so that
all the quantities in \eqref{ER3.2A} are equal.

However, in the absence of Assumption~\ref{A1.1}, we might
have $\Hat\Lambda_{\mathrm{m}}^{*}<\doublehat\Lambda_{\mathrm{m}}^{*}$,
which implies $\lambda^*<\doublehat\Lambda_{\mathrm{m}}^{*}$, so a limit point of
the Dirichlet eigensolutions will not in general yield a solution
to \eqref{PP1.3F} under the assumptions of Proposition~\ref{P1.3}.

The situation with the proof of Theorem~\ref{T1.2} is more subtle.
If, in the absence of Assumption~\ref{A1.1}, we consider the
Dirichlet problems in \eqref{PT1.2A} but with $\Lambda^*$ replaced
by $\lm$, then we indeed obtain a positive solution $\Phi$ of
\eqref{PT1.2B}, but we cannot argue that $\Phi$ is inf-compact.
Therefore, we cannot use the method in the proof of Lemma~\ref{L2.11}
to conclude that the
assumption $\lm(c+h)>\lm(c)$ for all $h\in\sCo$ implies $\alpha=0$
in \eqref{PT1.2B}, and the argument breaks down.
Instead, we use a more elaborate method for the
proof of the analogous statement in Proposition~\ref{P1.3}.
\end{remark}

%%%%%%%%%%%%%%%%%%%%%%%%%%%%%%%%%%%%%%%%%%%%%%%%%%%%%%%%%%%%%%%%%%%%%%%%%%%%%%%%
\section*{Acknowledgement}

We are indebted to the anonymous referee for his helpful comments, and
especially for discovering an error in the first version of this paper.
His/her review helped us to improve the paper substantially.

The research of Ari Arapostathis
was supported in part by the Office of Naval
Research through grant N00014-14-1-0196,
and in part by the Army Research Office through grant W911NF-17-1-001.
The research of Anup Biswas was supported in part by INSPIRE faculty
fellowship No. IFA13/MA-32.

%%%%%%%%%%%%%%%%%%%%%%%%%%%%%%%%%%%%%%%%%%%%%%%%%%%%%%%%%%%%%%%%%%%%%%%%%%%%%%%%
%\section*{References}

\def\polhk#1{\setbox0=\hbox{#1}{\ooalign{\hidewidth
  \lower1.5ex\hbox{`}\hidewidth\crcr\unhbox0}}}


\begin{thebibliography}{52}
\providecommand{\natexlab}[1]{#1}
\providecommand{\url}[1]{\texttt{#1}}
\expandafter\ifx\csname urlstyle\endcsname\relax
  \providecommand{\doi}[1]{doi: #1}\else
  \providecommand{\doi}{doi: \begingroup \urlstyle{rm}\Url}\fi

\bibitem[Arapostathis et~al.(1999)Arapostathis, Ghosh, and Marcus]{AA-Harnack}
A.~Arapostathis, M.~K. Ghosh, and S.~I. Marcus.
\newblock Harnack's inequality for cooperative weakly coupled elliptic systems.
\newblock \emph{Comm. Partial Differential Equations}, 24\penalty0
  (9-10):\penalty0 1555--1571, 1999.

\bibitem[Arapostathis et~al.(2011)Arapostathis, Borkar, and Ghosh]{book}
A.~Arapostathis, V.~S. Borkar, and M.~K. Ghosh.
\newblock \emph{Ergodic control of diffusion processes}, volume 143 of
  \emph{Encyclopedia of Mathematics and its Applications}.
\newblock Cambridge University Press, Cambridge, 2011.

\bibitem[Balaji and Meyn(2000)]{Balaji-00}
S.~Balaji and S.~P. Meyn.
\newblock Multiplicative ergodicity and large deviations for an irreducible
  {M}arkov chain.
\newblock \emph{Stochastic Process. Appl.}, 90\penalty0 (1):\penalty0 123--144,
  2000.

\bibitem[Bensoussan and Elliott(1995)]{BenEll-95}
A.~Bensoussan and R.~J. Elliott.
\newblock A finite-dimensional risk-sensitive control problem.
\newblock \emph{SIAM J. Control Optim.}, 33\penalty0 (6):\penalty0 1834--1846,
  1995.

\bibitem[Bensoussan and Nagai(2000)]{BenNag-00}
A.~Bensoussan and H.~Nagai.
\newblock Conditions for no breakdown and {B}ellman equations of risk-sensitive
  control.
\newblock \emph{Appl. Math. Optim.}, 42\penalty0 (2):\penalty0 91--101, 2000.

\bibitem[Bensoussan et~al.(1998)Bensoussan, Frehse, and Nagai]{BenNag-98}
A.~Bensoussan, J.~Frehse, and H.~Nagai.
\newblock Some results on risk-sensitive control with full observation.
\newblock \emph{Appl. Math. Optim.}, 37\penalty0 (1):\penalty0 1--41, 1998.

\bibitem[Berestycki and Rossi(2015)]{Berestycki-15}
H.~Berestycki and L.~Rossi.
\newblock Generalizations and properties of the principal eigenvalue of
  elliptic operators in unbounded domains.
\newblock \emph{Comm. Pure Appl. Math.}, 68\penalty0 (6):\penalty0 1014--1065,
  2015.

\bibitem[Berestycki et~al.(1994)Berestycki, Nirenberg, and
  Varadhan]{Berestycki-94}
H.~Berestycki, L.~Nirenberg, and S.~R.~S. Varadhan.
\newblock The principal eigenvalue and maximum principle for second-order
  elliptic operators in general domains.
\newblock \emph{Comm. Pure Appl. Math.}, 47\penalty0 (1):\penalty0 47--92,
  1994.

\bibitem[Bielecki and Pliska(1999)]{Bie-Pli}
T.~R. Bielecki and S.~R. Pliska.
\newblock Risk-sensitive dynamic asset management.
\newblock \emph{Appl. Math. Optim.}, 39\penalty0 (3):\penalty0 337--360, 1999.

\bibitem[Biswas(2011{\natexlab{a}})]{Biswas-11}
A.~Biswas.
\newblock Risk sensitive control of diffusions with small running cost.
\newblock \emph{Appl. Math. Optim.}, 64\penalty0 (1):\penalty0 1--12,
  2011{\natexlab{a}}.

\bibitem[Biswas(2011{\natexlab{b}})]{Biswas-11a}
A.~Biswas.
\newblock An eigenvalue approach to the risk sensitive control problem in near
  monotone case.
\newblock \emph{Systems Control Lett.}, 60\penalty0 (3):\penalty0 181--184,
  2011{\natexlab{b}}.

\bibitem[Biswas et~al.(2010)Biswas, Borkar, and Suresh~Kumar]{Biswas-10}
A.~Biswas, V.~S. Borkar, and K.~Suresh~Kumar.
\newblock Risk-sensitive control with near monotone cost.
\newblock \emph{Appl. Math. Optim.}, 62\penalty0 (2):\penalty0 145--163, 2010.

\bibitem[Bogachev et~al.(2001)Bogachev, Krylov, and R{\"o}ckner]{Bogachev-01}
V.~I. Bogachev, N.~V. Krylov, and M.~R{\"o}ckner.
\newblock On regularity of transition probabilities and invariant measures of
  singular diffusions under minimal conditions.
\newblock \emph{Comm. Partial Differential Equations}, 26\penalty0
  (11-12):\penalty0 2037--2080, 2001.

\bibitem[Borkar and Kumar(2014)]{BorSur-14}
V.~S. Borkar and K.~S. Kumar.
\newblock Small noise large time asymptotics for the normalized {F}eynman-{K}ac
  semigroup.
\newblock In \emph{Variational and optimal control problems on unbounded
  domains}, volume 619 of \emph{Contemp. Math.}, pages 31--48. Amer. Math.
  Soc., Providence, RI, 2014.

\bibitem[Borkar and Meyn(2002)]{BorMeyn-02}
V.~S. Borkar and S.~P. Meyn.
\newblock Risk-sensitive optimal control for {M}arkov decision processes with
  monotone cost.
\newblock \emph{Math. Oper. Res.}, 27\penalty0 (1):\penalty0 192--209, 2002.

\bibitem[Cavazos-Cadena(2010)]{Cavazos-10}
R.~Cavazos-Cadena.
\newblock Optimality equations and inequalities in a class of risk-sensitive
  average cost {M}arkov decision chains.
\newblock \emph{Math. Methods Oper. Res.}, 71\penalty0 (1):\penalty0 47--84,
  2010.

\bibitem[Cavazos-Cadena and Hern{\'a}ndez-Hern{\'a}ndez(2005)]{Cavazos-05}
R.~Cavazos-Cadena and D.~Hern{\'a}ndez-Hern{\'a}ndez.
\newblock A characterization of the optimal risk-sensitive average cost in
  finite controlled {M}arkov chains.
\newblock \emph{Ann. Appl. Probab.}, 15\penalty0 (1A):\penalty0 175--212, 2005.

\bibitem[Cavazos-Cadena and Hern{\'a}ndez-Hern{\'a}ndez(2009)]{Cavazos-09}
R.~Cavazos-Cadena and D.~Hern{\'a}ndez-Hern{\'a}ndez.
\newblock Necessary and sufficient conditions for a solution to the
  risk-sensitive {P}oisson equation on a finite state space.
\newblock \emph{Systems Control Lett.}, 58\penalty0 (4):\penalty0 254--258,
  2009.

\bibitem[Chen and Wu(1998)]{ChenWu}
Y.-Z. Chen and L.-C. Wu.
\newblock \emph{Second order elliptic equations and elliptic systems}, volume
  174 of \emph{Translations of Mathematical Monographs}.
\newblock American Mathematical Society, Providence, RI, 1998.

\bibitem[Di~Masi and Stettner(1999)]{diMassi-99}
G.~B. Di~Masi and L.~Stettner.
\newblock Risk-sensitive control of discrete-time {M}arkov processes with
  infinite horizon.
\newblock \emph{SIAM J. Control Optim.}, 38\penalty0 (1):\penalty0 61--78
  (electronic), 1999.

\bibitem[Di~Masi and Stettner(2001)]{diMassi-01}
G.~B. Di~Masi and {\L}.~Stettner.
\newblock Risk-sensitive control of an ergodic diffusion over an infinite
  horizon.
\newblock In \emph{Proceedings of the {S}eminar on {S}tability {P}roblems for
  {S}tochastic {M}odels, {P}art {I} ({N}al\polhk eczow, 1999)}, volume 105,
  pages 2541--2549, 2001.

\bibitem[Di~Masi and Stettner(2007)]{diMassi-07}
G.~B. Di~Masi and {\L}.~Stettner.
\newblock Infinite horizon risk sensitive control of discrete time {M}arkov
  processes under minorization property.
\newblock \emph{SIAM J. Control Optim.}, 46\penalty0 (1):\penalty0 231--252
  (electronic), 2007.

\bibitem[Down et~al.(1995)Down, Meyn, and Tweedie]{Down-95}
D.~Down, S.~P. Meyn, and R.~L. Tweedie.
\newblock Exponential and uniform ergodicity of {M}arkov processes.
\newblock \emph{Ann. Probab.}, 23\penalty0 (4):\penalty0 1671--1691, 1995.

\bibitem[Fleming(1997)]{Fleming-97}
W.~H. Fleming.
\newblock Some results and problems in risk sensitive stochastic control.
\newblock \emph{Mat. Apl. Comput.}, 16\penalty0 (2):\penalty0 99--115, 1997.

\bibitem[Fleming(2006)]{Fleming-06}
W.~H. Fleming.
\newblock Risk sensitive stochastic control and differential games.
\newblock \emph{Commun. Inf. Syst.}, 6\penalty0 (3):\penalty0 161--177, 2006.

\bibitem[Fleming and McEneaney(1995)]{Fleming-95}
W.~H. Fleming and W.~M. McEneaney.
\newblock Risk-sensitive control on an infinite time horizon.
\newblock \emph{SIAM J. Control Optim.}, 33\penalty0 (6):\penalty0 1881--1915,
  1995.

\bibitem[Fleming and Sheu(2002)]{Fleming-02}
W.~H. Fleming and S.~J. Sheu.
\newblock Risk-sensitive control and an optimal investment model. {II}.
\newblock \emph{Ann. Appl. Probab.}, 12\penalty0 (2):\penalty0 730--767, 2002.

\bibitem[Fort and Roberts(2005)]{Fort-05}
G.~Fort and G.~O. Roberts.
\newblock Subgeometric ergodicity of strong {M}arkov processes.
\newblock \emph{Ann. Appl. Probab.}, 15\penalty0 (2):\penalty0 1565--1589,
  2005.

\bibitem[Gilbarg and Trudinger(1983)]{GilTru}
D.~Gilbarg and N.~S. Trudinger.
\newblock \emph{Elliptic partial differential equations of second order},
  volume 224 of \emph{Grundlehren der Mathematischen Wissenschaften}.
\newblock Springer-Verlag, Berlin, second edition, 1983.

\bibitem[Gy{\"o}ngy and Krylov(1996)]{Gyongy-96}
I.~Gy{\"o}ngy and N.~Krylov.
\newblock Existence of strong solutions for {I}t\^o's stochastic equations via
  approximations.
\newblock \emph{Probab. Theory Related Fields}, 105\penalty0 (2):\penalty0
  143--158, 1996.

\bibitem[Has'minski\u{i}(1980)]{Hasminskii}
R.~Z. Has'minski\u{i}.
\newblock \emph{Stochastic stability of differential equations}.
\newblock Sijthoff \& Noordhoff, The Netherlands, 1980.

\bibitem[Howard and Matheson(1971/72)]{Howard-71}
R.~A. Howard and J.~E. Matheson.
\newblock Risk-sensitive {M}arkov decision processes.
\newblock \emph{Management Sci.}, 18:\penalty0 356--369, 1971/72.

\bibitem[Ichihara(2012)]{Ichihara-12}
N.~Ichihara.
\newblock Large time asymptotic problems for optimal stochastic control with
  superlinear cost.
\newblock \emph{Stochastic Process. Appl.}, 122\penalty0 (4):\penalty0
  1248--1275, 2012.

\bibitem[Ja{\'s}kiewicz(2007)]{Jaskiewicz-07}
A.~Ja{\'s}kiewicz.
\newblock Average optimality for risk-sensitive control with general state
  space.
\newblock \emph{Ann. Appl. Probab.}, 17\penalty0 (2):\penalty0 654--675, 2007.

\bibitem[Kaise and Sheu(2006)]{Kaise-06}
H.~Kaise and S.-J. Sheu.
\newblock On the structure of solutions of ergodic type {B}ellman equation
  related to risk-sensitive control.
\newblock \emph{Ann. Probab.}, 34\penalty0 (1):\penalty0 284--320, 2006.

\bibitem[Kontoyiannis and Meyn(2003)]{Kontoyiannis-02}
I.~Kontoyiannis and S.~P. Meyn.
\newblock Spectral theory and limit theorems for geometrically ergodic {M}arkov
  processes.
\newblock \emph{Ann. Appl. Probab.}, 13\penalty0 (1):\penalty0 304--362, 2003.

\bibitem[Kontoyiannis and Meyn(2005)]{Kontoyiannis-05}
I.~Kontoyiannis and S.~P. Meyn.
\newblock Large deviations asymptotics and the spectral theory of
  multiplicatively regular {M}arkov processes.
\newblock \emph{Electron. J. Probab.}, 10:\penalty0 no. 3, 61--123
  (electronic), 2005.

\bibitem[Krylov(1980)]{Krylov}
N.~V. Krylov.
\newblock \emph{Controlled diffusion processes}, volume~14 of
  \emph{Applications of Mathematics}.
\newblock Springer-Verlag, New York, 1980.

\bibitem[Krylov and R{\"o}ckner(2005)]{Krylov-05}
N.~V. Krylov and M.~R{\"o}ckner.
\newblock Strong solutions of stochastic equations with singular time dependent
  drift.
\newblock \emph{Probab. Theory Related Fields}, 131\penalty0 (2):\penalty0
  154--196, 2005.

\bibitem[Liptser and Shiryayev(1977)]{LiSh-I}
R.~S. Liptser and A.~N. Shiryayev.
\newblock \emph{Statistics of random processes. {I}}.
\newblock Springer-Verlag, New York, 1977.
\newblock General theory, Translated by A. B. Aries, Applications of
  Mathematics, Vol. 5.

\bibitem[Menaldi and Robin(2005)]{Menaldi-05}
J.-L. Menaldi and M.~Robin.
\newblock Remarks on risk-sensitive control problems.
\newblock \emph{Appl. Math. Optim.}, 52\penalty0 (3):\penalty0 297--310, 2005.

\bibitem[Nagai(1996)]{Nagai-96}
H.~Nagai.
\newblock Bellman equations of risk-sensitive control.
\newblock \emph{SIAM J. Control Optim.}, 34\penalty0 (1):\penalty0 74--101,
  1996.

\bibitem[Nagai(2012)]{Nagai-12}
H.~Nagai.
\newblock Downside risk minimization via a large deviations approach.
\newblock \emph{Ann. Appl. Probab.}, 22\penalty0 (2):\penalty0 608--669, 2012.

\bibitem[Nagengast et~al.(2010)Nagengast, Braun, and Wolpert]{Nagen-10}
A.~Nagengast, D.~Braun, and D.~Wolpert.
\newblock Risk-sensitive optimal feedback control accounts for sensorimotor
  behavior under uncertainty.
\newblock \emph{PLoS Computational Biology}, 6\penalty0 (7):\penalty0 e1000857,
  2010.

\bibitem[Pinsky(1995)]{Pinsky}
R.~G. Pinsky.
\newblock \emph{Positive harmonic functions and diffusion}, volume~45 of
  \emph{Cambridge Studies in Advanced Mathematics}.
\newblock Cambridge University Press, Cambridge, 1995.

\bibitem[Quaas and Sirakov(2008)]{Quaas-08a}
A.~Quaas and B.~Sirakov.
\newblock Principal eigenvalues and the {D}irichlet problem for fully nonlinear
  elliptic operators.
\newblock \emph{Adv. Math.}, 218\penalty0 (1):\penalty0 105--135, 2008.

\bibitem[Speyer(1976)]{Speyer}
J.~L. Speyer.
\newblock An adaptive terminal guidance scheme based on an exponential cost
  criterion with application to homing missile guidance.
\newblock \emph{IEEE Trans. Automatic Control}, 21\penalty0 (3):\penalty0
  371--375, 1976.

\bibitem[Suresh~Kumar and Pal(2015)]{Suresh-15}
K.~Suresh~Kumar and C.~Pal.
\newblock Risk-sensitive ergodic control of continuous time {M}arkov processes
  with denumerable state space.
\newblock \emph{Stoch. Anal. Appl.}, 33\penalty0 (5):\penalty0 863--881, 2015.

\bibitem[Whittle(1981)]{Whittle-81}
P.~Whittle.
\newblock Risk-sensitive linear/quadratic/{G}aussian control.
\newblock \emph{Adv. in Appl. Probab.}, 13\penalty0 (4):\penalty0 764--777,
  1981.

\bibitem[Whittle(1990)]{Whittle-90}
P.~Whittle.
\newblock \emph{Risk-sensitive optimal control}.
\newblock Wiley-Interscience Series in Systems and Optimization. John Wiley \&
  Sons, Ltd., Chichester, 1990.

\bibitem[Wu(1994)]{Wu-94}
L.~M. Wu.
\newblock Feynman-{K}ac semigroups, ground state diffusions, and large
  deviations.
\newblock \emph{J. Funct. Anal.}, 123\penalty0 (1):\penalty0 202--231, 1994.

\bibitem[Yoshimura(2006)]{Yoshimura-06}
Y.~Yoshimura.
\newblock A note on demi-eigenvalues for uniformly elliptic {I}saacs operators.
\newblock \emph{Viscosity Solution Theory of Differential Equations and its
  Developments}, pages 106--114, 2006.

\end{thebibliography}
\end{document}